\documentclass{amsart}
\usepackage[utf8]{inputenc}\usepackage{amsmath}
\usepackage{amsfonts}
\usepackage{amssymb}
\usepackage[english]{babel}
\usepackage{amsmath,amsthm}
\usepackage{amsfonts}
\usepackage{fancyhdr}
\usepackage{setspace}
\usepackage{color}
\usepackage[percent]{overpic}
\usepackage{fix-cm}
\usepackage{graphicx}
\usepackage[all]{xy}
\usepackage{mathrsfs}
\usepackage{mathdots}
\usepackage{amsbsy}
\usepackage{leftidx}
\usepackage[hidelinks]{hyperref}
\usepackage{xcolor}
\hypersetup{
    linktocpage,
    colorlinks,
    linkcolor={magenta!100!black},
    citecolor={cyan!100!black},
    urlcolor={yellow!100!black}
}

\usepackage{tikz}

\theoremstyle{plain}
\newtheorem{teo}{Theorem}[section]
\newtheorem{lema}[teo]{Lemma}
\newtheorem{prop}[teo]{Proposition}
\newtheorem*{propsn}{Proposition}
\newtheorem{cor}[teo]{Corollary}
\newtheorem*{probsnmain}{Main Problem}

\theoremstyle{definition}

\newtheorem{rem}[teo]{Remark}

\theoremstyle{definition}

\theoremstyle{definition}

\theoremstyle{definition}

\theoremstyle{definition}

\theoremstyle{remark}
\newtheorem{cla}[teo]{Claim}

\theoremstyle{remark}
\newtheorem{ex}[teo]{Example}

\newcommand{\xsypq}{\mathsf{Q}_{p,q}}
\newcommand{\so}{\mathsf{S}^o}
\newcommand{\xsy}{\mathsf{X}_\g}
\newcommand{\xsyg}{\mathsf{X}_\g}
\newcommand{\adg}{\leftidx{^{\star_o}}g}
\newcommand{\fvo}{\mathsf{F}(V)^o}
\newcommand{\fvoc}{(\mathsf{F}(V)^o)^{(2)}}
\newcommand{\fvc}{\mathsf{F}(V)^{(2)}}

\newcommand{\orho}{\pmb{\Omega}_\rho}

\newcommand{\sgo}{\textnormal{sg}_o}
\newcommand{\bo}{\mathscr{B}_o}

\newcommand{\bog}{\mathscr{B}_{o,\g}}
\newcommand{\bd}{\beta}

\newcommand{\gfo}{(\g\times\mathsf{F}(V))^o}
\newcommand{\gro}{\mathbb{G}}

\newcommand{\liep}{\mathfrak{p}}
\newcommand{\liek}{\mathfrak{k}}
\newcommand{\lieq}{\mathfrak{q}}
\newcommand{\lieh}{\mathfrak{h}}
\newcommand{\lieg}{\mathfrak{g}}
\newcommand{\lieb}{\mathfrak{b}}
\newcommand{\liea}{\mathfrak{a}}

\newcommand{\liegto}{\mathfrak{g}^{\tau o}}

\newcommand{\g}{\textnormal{G}}
\newcommand{\h}{\textnormal{H}}
\newcommand{\ko}{\textnormal{K}}

\newcommand{\wbh}{\hat{\textnormal{W}}}
\newcommand{\wb}{\textnormal{W}}

\newcommand{\pmin}{\textnormal{P}_{\liea^+}}

\newcommand{\war}{\textnormal{W}_{\rho,{\liea^+}}}
\newcommand{\n}{\textnormal{N}}
\newcommand{\mb}{\textnormal{M}}

\newcommand{\rr}{\mathbb{R}}

\newcommand{\cc}{\mathbb{C}}

\newcommand{\kk}{\mathbb{K}}

\newcommand{\pp}{\mathsf{P}}

\newcommand{\bb}{\mathbb{B}}

\newcommand{\bc}{\begin{center}}
\newcommand{\ec}{\end{center}}

\newcommand{\too}{\rightarrow}

\newcommand{\cont}{\mathscr{L}_\rho}
\newcommand{\cono}{\mathscr{L}^{p,q}_\rho}

\newcommand{\pst}{\psi_\rho}

\newcommand{\bg}{\partial_\infty\Gamma}
\newcommand{\bgc}{\partial_\infty^{2}\Gamma}

\newcommand{\gh}{\Gamma_{\tn{H}}}

\newcommand{\tn}{\textnormal}
\newcommand{\tc}{\textcolor}

\numberwithin{equation}{section}

\newcounter{tmp}

\begin{document}

\title[Growth of quadratic forms]{Growth of quadratic forms under Anosov subgroups}
\author{León Carvajales}
\thanks{Research partially funded by CSIC Grupo 618 ``Sistemas Din\'amicos'', CSIC MIA (2019), CSIC Iniciación C068 (2018), MathAmsud RGSD MOV\_CO\_2018\_1\_1008354 and Fondo Clemente Estable ANII-FCE-135352.}

\address[\textbf{León Carvajales}]{\newline Centro de Matemática\newline Universidad de la República\newline Iguá 4225, 11400, Montevideo, Uruguay. \newline Sorbonne Universit\'e - Campus Pierre et Marie Curie\newline 
Institut de Math\'ematiques de Jussieu\newline 
4, place Jussieu - Boite Courrier 247, 75252 Paris Cedex 05, France.}
\email{lcarvajales@cmat.edu.uy}

\maketitle

\begin{abstract}

Let $\rho:\Gamma\too\tn{PSL}_d(\kk)$ be a Zariski dense Borel-Anosov representation, for $\kk$ equal to $\rr$ or $\cc$. Let $o$ be a form of signature $(p,d-p)$ on $\kk^d$ (where $0<p<d)$. Let $\so$ be the corresponding geodesic copy of the Riemannian symmetric space of $\tn{PSO}(o)$, inside the Riemannian symmetric space of $\tn{PSL}_d(\kk)$. For certain choices of $o$ and every $t$ large enough, we show exponential bounds for the number of $\gamma\in\Gamma$ for which the distance between $\so$ and $\rho\gamma\cdot\so$ is smaller than $t$. Under an extra assumption, satisfied for instance when the boundary of $\Gamma$ is connected, we show an asymptotic as $t\too\infty$ for the counting function relative to a functional in the interior of the dual limit cone.
\end{abstract}

\tableofcontents
\addtocontents{toc}{\setcounter{tocdepth}{1}}

\section{Introduction}

Let $d=p+q\geq 3$, where $p$ and $q$ are positive integers and $V$ be a $\kk$-vector space of dimension $d$, where $\kk=\rr$ or $\kk=\cc$. Let $\g:=\tn{PSL}(V)$ and $\xsyg$ be the Riemannian symmetric space of $\g$. A \textit{form of signature} $(p,q)$ on $V$ is a quadratic or Hermitian form on $V$ of that signature, depending respectively on whether $\kk=\rr$ or $\kk=\cc$. The space of homothety classes of forms of signature $(p,q)$ on $V$ is denoted by $\xsypq$. The group of projectivized linear isometries of a basepoint $o\in\xsypq$ is denoted by $\h^o$. It is isomorphic to $\tn{PSO}(p,q)$ (resp. $\tn{PSU}(p,q)$) if $\kk=\rr$ (resp. $\kk=\cc$). We let $\so\subset\xsyg$ be the corresponding totally geodesic copy of the Riemannian symmetric space of $\h^o$. In this paper we study the following problem.

\begin{probsnmain}
Let $\Xi$ be a discrete subgroup of $\g$. Describe the set of points $o\in\xsypq$ for which the \tn{counting function}
\begin{equation}\label{eq counting function so}
t\mapsto\#\lbrace g\in\Xi:\hspace{0,3cm} d_{\xsyg}(\so,g\cdot\so)\leq t\rbrace
\end{equation}
\noindent is finite\footnote{Observe that if this is the case then the intersection $\Xi\cap\h^o$ must be finite.} for every positive $t$ and study its asymptotic behaviour as $t\too\infty$. 
\end{probsnmain}

In the previous formulation $d_{\xsyg}(\cdot,\cdot)$ denotes the distance on $\xsyg$ coming from a $\g$-invariant Riemannian structure and, for closed subsets $A$ and $B$ of $\xsyg$, we let

\bc
$d_{\xsyg}(A,B):=\inf\lbrace d_{\xsyg}(a,b):\hspace{0,3cm} a\in A, b\in B\rbrace$.
\ec 

In this paper we contribute to the study of the previous problem when $\Xi$ is the image of a word hyperbolic group $\Gamma$ under a $\Delta$-\textit{Anosov} representation $\rho:\Gamma\too\g$, where $\Delta$ denotes the set of simple roots of some Weyl chamber. Anosov representations were introduced by Labourie \cite{Lab} and further extended by Guichard-Wienhard \cite{GW}. They provide a (stable) class of faithful and discrete representations from word hyperbolic groups into semisimple Lie groups that share many features with holonomies of convex co-compact hyperbolic manifolds. They have been object of intensive research in recent years (see e.g. Kassel \cite{Kas}, Pozzetti \cite{PozBourbaki} or Wienhard \cite{Wie}). Reminders on this notion are given in Subsection \ref{subsec def anosov}, but we recall here that these representations come equipped with a continuous equivariant \textit{limit map}

\bc
$\xi_\rho:\bg\too\mathsf{F}(V)$.
\ec

\noindent Here $\bg$ denotes the Gromov boundary of $\Gamma$ and $\mathsf{F}(V)$ denotes the \textit{full flag manifold} of $\g$, that is, the space of $d$-uples of the form

\bc
$\xi=(\xi^1\subset\dots\subset\xi^{d})$
\ec

\noindent where $\xi^j$ is a $j$-dimensional subspace of $V$ for each $j=1,\dots,d$. A central feature about the limit map $\xi_\rho$ is that it is \textit{transverse}, i.e. for every $x\neq y$ in $\bg$ and every $j=1,\dots,d-1$, the subspace $\xi_\rho^j(x)$ is linearly disjoint from $\xi_\rho^{d-j}(y)$ (see Subsection \ref{subsec def anosov} for further precisions and references).

Let $\rho:\Gamma\too\g$ be a $\Delta$-Anosov representation and define $\orho\subset\xsypq$ to be the set consisting of forms for which the subspaces $\xi_\rho^j(x)$ are non degenerate for every $x\in\bg$ and every $j=1,\dots,d$. In Subsection \ref{subsec orho} we discuss examples of Anosov representations for which the set $\orho$ is non empty. 

Towards the understanding of the Main Problem above we prove the following result.

\begingroup
\setcounter{tmp}{\value{teo}}
\setcounter{teo}{0} 
\renewcommand\theteo{\Alph{teo}}

\begin{cor}[Corollaries {\ref{cor sigmagammainvgamma is loxod and Uo close to Utau} \& \ref{cor counting function and critical exponent}}]\label{cor bound for counting problem introduction}

Assume that $\rho$ is Zariski dense and let $o\in\orho$. Then there exist positive constants $\delta_\rho$, $\mathtt{C}_1$ and $\mathtt{C}_2$ such that for every $t$ large enough one has

\bc
$\mathtt{C}_1e^{\delta_\rho t}\leq  \#\lbrace \gamma\in\Gamma:\hspace{0,3cm} d_{\xsyg}(\so,\rho\gamma\cdot\so)\leq t\rbrace  \leq\mathtt{C}_2e^{\delta_\rho t}$.
\ec

\end{cor}

The constant $\delta_\rho$ in Corollary \ref{cor bound for counting problem introduction} is independent of the choice of the basepoint $o$ and coincides with the \textit{critical exponent}

\bc
$\displaystyle\limsup_{t\too\infty}\dfrac{\log\#\left\lbrace \gamma\in\Gamma: \hspace{0,3cm} d_{\xsyg}(\tau,\rho\gamma\cdot\tau)\leq t\right\rbrace}{t}$
\ec

\noindent of $\rho$ (for any point $\tau\in\xsyg$). Corollary \ref{cor bound for counting problem introduction} is a consequence of a uniform estimate of the distance $d_{\xsyg}(\so,\rho\gamma\cdot\so)$ in terms of the Cartan projection of $\rho\gamma$, and a corresponding counting theorem for this projection due to Sambarino \cite{Sam2}. In \cite{Sam2} Sambarino proves his theorem when $\Gamma$ is the fundamental group of a closed negatively curved manifold. In Appendix \ref{app proof of ditribution} we explain how his proof adapts to our more general setting.

\subsection{Main result}

The main result of this paper (Theorem \ref{TEO CONTEO DIRECCIONAL BO EN QPQ INTRODUCCION} below) provides a precise asymptotic, as $t\too\infty$, for a counting function similar to (\ref{eq counting function so}) but with the distance $d_{\xsyg}(\cdot,\cdot)$ replaced by the choice of a linear functional in the interior of the dual limit cone of $\rho$. We now state this result in a proper way.

A (maximal) \textit{flat} of $\xsyg$ is a totally geodesic copy of $\rr^{d-1}$ inside $\xsyg$. The space of flats identifies naturally with the space of Cartan subspaces $\liea$ of the Lie algebra $\lieg$ of $\g$. Cartan subspaces corresponding to flats orthogonal to $\so$ will be denoted with the symbol $\lieb$.

Let $o$ be a point in $\xsypq$ and define $\bog$ to be the set consisting of elements $g\in\g$ for which there exists a flat of $\xsyg$ orthogonal both to $\so$ and $g\cdot\so$. The Lie theoretic description of $\bog$ goes as follows. Fix a Cartan subspace $\lieb$ whose corresponding flat is orthogonal to $\so$ and let $\tau$ be the intersection point between this flat and $\so$. We think here the point $\tau\in\so$ as a Cartan involution of $\lieg$ that commutes with the derivative $d\sigma^o$ of the involution 

\bc
$\sigma^o:\g\too\g$,
\ec

\noindent whose fixed point set is $\h^o=\tn{PSO}(o)$ (see Subsections \ref{subsec notations} and \ref{subsec so}).

Let $\liegto$ be the subalgebra of $\lieg$ consisting of fixed points of the involution $\tau(d\sigma^o)$ and $\lieb^+$ be a (closed) Weyl chamber of the set of restricted roots $\Sigma(\liegto,\lieb)$ (c.f. Subsection \ref{subsec weyl chambers}). Consider the Weyl group $\wb:=\tn{N}_{\ko^\tau}(\lieb)/\tn{Z}_{\ko^\tau}(\lieb)$, where $\ko^\tau$ is the maximal compact subgroup of $\g$ associated to $\tau$ and $\tn{N}_{\ko^\tau}(\lieb)$ (resp. $\tn{Z}_{\ko^\tau}(\lieb)$) is the normalizer (resp. centralizer) of $\lieb$ in $\ko^\tau$.

\begin{propsn}[Propositions \ref{prop equivalences HWbH} and \ref{prop uniqueness of liebplus}]
One has a decomposition

\bc
$\bog=\h^o\wb\exp(\lieb^+)\h^o$.
\ec

\noindent Furthermore, the $\lieb^+$-coordinate in this decomposition is uniquely determined.
\end{propsn}

The previous decomposition of $\bog$ will be called $(p,q)$-\textit{Cartan decomposition}. We then introduce a $(p,q)$-\textit{Cartan projection}

\bc
$b^o:\bog\too\lieb^+$
\ec

\noindent which is characterized by the equality

\bc
$g=hw\exp(b^o(g))h'$
\ec

\noindent for every $g\in\bog$, where $h,h'\in\h^o$ and $w\in\wb$. In Lemma \ref{lema flat orthogonal to So and gSo in HWBH coordinates and distance} we show the equality

\bc
$\Vert b^o(g)\Vert_{\lieb}=d_{\xsyg}(\so,g\cdot \so)$
\ec

\noindent for every $g\in\bog$, where $\Vert\cdot\Vert_\lieb$ is the Euclidean norm on $\lieb$ induced by the $\g$-invariant Riemannian structure of $\xsyg$.

In Corollary \ref{cor sigmagammainvgamma is loxod and Uo close to Utau} we prove that for every $\Delta$-Anosov representation $\rho:\Gamma\too\g$ and every basepoint $o\in\orho$, then apart from possibly finitely many exceptions $\gamma\in\Gamma$ one has $\rho\gamma\in\bog$.

We can now state our main result. Recall that since $\lieb$ is a Cartan subspace of $\lieg$ we can use the notation $\liea=\lieb$. A closed Weyl chamber of the system $\Sigma(\lieg,\lieb)=\Sigma(\lieg,\liea)$ will be denoted by $\liea^+$. It will be said to be $\lieb^+$-\textit{compatible} if the inclusion $\liea^+\subset\lieb^+$ holds. The corresponding \textit{asymptotic cone}, as introduced by Benoist in \cite{Ben4}, will be denoted by $\cont\subset\tn{int}(\liea^+)$. We let $\cont^*$ be the dual cone.

\begin{teo}[Proposition \ref{prop linear counting for varphibo when torsion}]\label{TEO CONTEO DIRECCIONAL BO EN QPQ INTRODUCCION}
Let $\rho:\Gamma\too\g$ be a Zariski dense $\Delta$-Anosov representation and $o$ be a basepoint in $\orho$ such that for each $j=1,\dots,d$ the signature of the form $o$ restricted to $\xi_\rho^j(x)$ is independent on $x\in\bg$. Then there exists a $\lieb^+$-compatible Weyl chamber $\liea^+$ such that for every linear functional $\varphi$ in the interior of $\cont^*$ there exist constants $h_\rho^\varphi>0$ and $\mathtt{m}=\mathtt{m}_{\rho,o,\varphi}>0$ satisfying

\bc
$\mathtt{m}e^{-h_\rho^\varphi t}\# \lbrace\gamma\in\Gamma: \hspace{0,3cm} \rho\gamma\in\bog \tn{ and } \varphi(b^o(\rho\gamma))\leq t\rbrace\too 1$
\ec

\noindent as $t\too\infty$.
\end{teo}

\endgroup

The constant $h_\rho^\varphi$ in Theorem \ref{TEO CONTEO DIRECCIONAL BO EN QPQ INTRODUCCION} coincides with the $\varphi$-\textit{entropy} of $\rho$, defined by

\bc
$h_\rho^\varphi:=\displaystyle\limsup_{t\too\infty}\dfrac{\log\# \lbrace[\gamma]\in[\Gamma]: \hspace{0,3cm}  \varphi(\lambda(\rho\gamma))\leq t\rbrace}{t}$.
\ec

\noindent Here $\lambda(\cdot)$ denotes the Jordan projection of $\g$ and $[\gamma]$ denotes the conjugacy class of $\gamma\in\Gamma$. In other words, the choice of $\varphi$ induces a H\"older reparametrization $\phi_t^\varphi$ of the geodesic flow of $\rho$, and  $h_\rho^\varphi$ is the topological entropy of $\phi_t^\varphi$. Recall that the \textit{geodesic flow} of $\rho$ was introduced by Bridgeman-Canary-Labourie-Sambarino \cite{BCLS}. On the other hand, the constant $\mathtt{m}=\mathtt{m}_{\rho,o,\varphi}$ is related to the total mass of a specific measure in the Bowen-Margulis measure class of this reparametrization (i.e. the homothety class of measures maximizing the entropy of the flow $\phi_t^\varphi$).

\subsection{Method and outline of the proof}

After defining the $(p,q)$-Cartan projection we begin the study of its asymptotic properties. Given a $\lieb^+$-compatible Weyl chamber $\liea^+$, we let $\Delta\subset\Sigma(\lieg,\liea)$ be the set of simple roots and 

\bc
$\g=\ko^\tau\exp(\liea^+)\ko^\tau$
\ec

\noindent be the associated \textit{Cartan decomposition} of $\g$. We denote by

\bc
$a^\tau:\g\too\liea^+$
\ec

\noindent the corresponding \textit{Cartan projection}. Recall that an element $g\in\g$ is said to have a \textit{gap of index $\Delta$} if $\alpha(a^\tau(g))$ is positive for every $\alpha\in\Delta$. In this case we have well defined full flags $U^\tau(g)$ and $S^\tau(g)$ which are called respectively the \textit{Cartan attractor} and \textit{Cartan repellor} of $g$ (see Subsection \ref{subsec reminders cartan and jordan}).

A full flag $\xi\in\mathsf{F}(V)$ is said to be $o$-\textit{generic} if for every $j=1,\dots,d$ the restriction of the form $o$ to the subspace $\xi^j$ is non degenerate. The space of $o$-generic flags is denoted by $\fvo$ and coincides with the union of open orbits of the action $\h^o\curvearrowright\mathsf{F}(V)$ (see Subsection \ref{subsub open orbits in flags}).

Under the assumption that $g\in\g$ has a ``sufficiently strong" gap of index $\Delta$ and $o$-generic Cartan attractor and repellor, we compute in Subsection \ref{subsec computation weyl chamber and hatw} the $\lieb^+$-compatible Weyl chamber that contains $b^o(g)$. This makes the study of the $(p,q)$-Cartan projection tractable. Indeed for a given $\lieb^+$-compatible Weyl chamber $\liea^+$, in Subsection \ref{subsec linear alg interp of bo and first estimates} we show the inequality
\begin{equation}\label{eq bo and cartan introduction}
\Vert b^o(g)-w_g\cdot a^\tau(g)\Vert_\lieb\leq D
\end{equation}
\noindent for some $D>0$ and an element $w_g$ of the Weyl group that we can precisely describe. Further, we can describe $b^o(g)$ using the Jordan projection $\lambda$ of $\g$:
\begin{equation}\label{eq bo and jordan introduction}
b^o(g)=\frac{1}{2}w_g\cdot\lambda(\sigma^o(g^{-1})g).
\end{equation}
\noindent Note that the estimate (\ref{eq bo and cartan introduction}) and the equality (\ref{eq bo and jordan introduction}) are not canonical: they strongly depend on the choice of a $\lieb^+$-compatible Weyl chamber $\liea^+$. Because of this, we emphasize that we do \textbf{not} fix the $\lieb^+$-compatible Weyl chamber $\liea^+$, only the Weyl chamber $\lieb^+$ is fixed beforehand.

If a Zariski dense $\Delta$-Anosov representation $\rho:\Gamma\too\g$ and a basepoint $o$ in $\orho$ are given, the estimate (\ref{eq bo and cartan introduction}) combined with the work of Sambarino \cite{Sam2} allows us to prove Corollary \ref{cor bound for counting problem introduction}. Furthermore, let $\cono$ be the $(p,q)$-\textit{asymptotic cone} of $\rho$. By definition, it is the subset of $\lieb^+$ consisting on all possible limits of the form

\bc
$\dfrac{b^o(\rho\gamma_n)}{t_n}$
\ec

\noindent where $t_n\too\infty$. Recall that for a given Weyl chamber $\liea^+$ of the system $\Sigma(\lieg,\liea)$, the asymptotic cone of $\rho$ (in the sense of Benoist \cite{Ben4}) is denoted by $\cont$. We show the following.

\begin{propsn}[Proposition \ref{prop limit cono}]
Let $\liea^+$ be a $\lieb^+$-compatible Weyl chamber and $\cont\subset\tn{int}(\liea^+)$ be the associated asymptotic cone. Then there exists a subset $\war$ of the Weyl group $\wb$ for which one has

\bc
$\cono=\displaystyle\bigcup_{w\in\war}w\cdot\cont$.
\ec
\end{propsn}

The subset $\war$ is in one to one correspondence with the set open orbits of the action of $\h^o\curvearrowright\mathsf{F}(V)$ that intersect the limit set $\xi_\rho(\bg)$. In contrast with Benoist's asymptotic cone \cite{Ben4}, the $(p,q)$-asymptotic cone is not necessarily convex (c.f. Remark \ref{rem war equal to one iff limit set in unique open orbit}).

We then begin the study of finer asymptotic properties of the $(p,q)$-Cartan projection. In the classical setting (i.e. for the Cartan projection $a^\tau(\cdot)$), the key object that appears is the (vector valued) \textit{Busemann cocycle}\footnote{Sometimes also called the \textit{Iwasawa cocycle} of $\g$.} of $\g$, introduced by Quint \cite{Qui2}. In Section \ref{sec busemann cocycles and gromov} we introduce an analogue of the Busemann cocycle, that we call the $o$-\textit{Busemann cocycle}, and that plays the role of the classical Busemann cocycle in our setting. We remark here that its definition depends on the choice of a $\lieb^+$-compatible Weyl chamber.

The assumption over the limit set in Theorem \ref{TEO CONTEO DIRECCIONAL BO EN QPQ INTRODUCCION} is equivalent to the fact that $\xi_\rho(\bg)$ is contained in a single open orbit of the action $\h^o\curvearrowright\mathsf{F}(V)$. As we shall see (c.f. Subsection \ref{subsec compatible weyl and o generic}) this assumption canonically selects a $\lieb^+$-compatible Weyl chamber $\liea^+$ and for this Weyl chamber we have the equality

\bc
$\cono=\cont$.
\ec

\noindent Furthermore, for every $\gamma$ large enough the equality
\bc
$b^o(\rho\gamma)=\frac{1}{2}\lambda(\sigma^o(\rho\gamma^{-1})\rho\gamma)$
\ec
\noindent can be assumed to hold (see Corollary \ref{cor if limit set in one orbit then HwexplieaplusH}) and we have also a well defined vector valued $o$-Busemann cocycle for $\rho$ (Subsection \ref{subsec buseman and gromov for rho}). Benoist's framework \cite{Ben1} allows us to obtain a precise comparison between $\frac{1}{2}\lambda(\sigma^o(\rho\gamma^{-1})\rho\gamma)$ and $\lambda(\rho\gamma)$ in terms of some appropiate vector valued \textit{cross-ratio}, which turns out to be closely related with the $o$-Busemann cocycle of $\rho$ (c.f. Corollaries \ref{cor gromov is cross ratio} and \ref{cor gromov is cross ratio for rho and estimate with bogamma}). Equipped with this precise estimate, if a functional $\varphi$ in the interior of the dual cone $\cont^*$ is given we are in position of applying Sambarino's adaptation \cite{Sam} of Roblin's method \cite{Rob} to obtain Theorem \ref{TEO CONTEO DIRECCIONAL BO EN QPQ INTRODUCCION}.

\subsection{Final remarks}

To finish this introduction we discuss some work related to the present paper.

\subsubsection{\tn{\textbf{Domains of discontinuity}}}
In joint work with F. Stecker, which is still in progress, we prove that $\orho$ is a \textit{domain of discontinuity} for $\rho$, i.e. the action of $\Gamma$ on $\orho$ induced by $\rho$ is properly discontinuous. In fact, our construction provides examples of domains of discontinuity for Anosov representations in a large class of $\g$-homogeneous spaces (for a connected semisimple Lie group $\g$ with no compact factors) that include \textit{symmetric spaces} of $\g$. This construction generalizes that of Kapovich-Leeb-Porti \cite{KLP2} (see Stecker \cite{SteThesis} or C. \cite{CarThesis} for further details). We mention here that in the present paper we do not use the fact that the action of $\Gamma$ on $\orho$ is properly discontinuous.

\subsubsection{\tn{\textbf{The classical counting problem}}}

Classically, the \textit{orbital counting problem} concerns the study of the asymptotic behaviour of the function
\begin{equation}\label{eq counting function classical}
t\mapsto\#\lbrace g\in\Xi:\hspace{0,3cm} d_{\mathsf{X}}(o,g\cdot o)\leq t \rbrace
\end{equation}
\noindent as $t\too\infty$, where $\Xi$ is a discrete group of isometries of a given proper non compact metric space $\mathsf{X}$, and $o$ is a basepoint in $\mathsf{X}$. This problem has been studied in many situations, by authors among who we find notably Gauss, Huber, Patterson and Margulis. Of course, this list is highly incomplete (we refer the reader to Babillot's survey \cite{Ba} for a more complete picture). Let us mention here that the asymptotic behaviour of the function (\ref{eq counting function classical}) has been studied when $\mathsf{X}$ coincides with the Riemannian symmetric space of a semisimple Lie group $\g$ with no compact factors. Indeed, when $\Xi <\g$ is a lattice one finds the work of Duke-Rudnick-Sarnak \cite{DRS} and more generally that of Eskin-McMullen \cite{EMcM}. In the non lattice case one also finds the work of Quint \cite{Qui} and Sambarino \cite{Sam2} (who deal with $\Delta$-Anosov subgroups of $\g$), and the work of Thirion \cite{ThirionTHESIS} (who deals with Ping-Pong subgroups of $\tn{SL}_d(\rr)$). The approach by Sambarino and Thirion is inspired by Roblin's method \cite{Rob}.

\subsubsection{\tn{\textbf{Relation with the work of Parkkonen-Paulin}}}
When $d=2$ the Main Problem of this paper has been studied by Parkkonen-Paulin \cite{ParkPaulin}. The results of \cite{ParkPaulin} are valid also in some situations of variable curvature bounded above by a negative constant, and in some of these situations the authors obtain estimates on the error terms for their counting results (see \cite{ParkPaulin} for precisions). Parkkonen and Paulin's work generalizes (to the variable curvature setting) previous work of Eskin-McMullen \cite{EMcM}, Oh-Shah \cite{OS3,OS2,OS4,OS} and Mohammadi-Oh \cite{MO}, which include counting problems associated to symmetric spaces of $\tn{SO}_0(n,1)$.

\subsubsection{\tn{\textbf{Relation with the work of Edwards-Lee-Oh}}}

In a recent preprint \cite{EdwLeeOh}, Edwards, Lee and Oh treat a counting problem for Zariski dense $\Delta$-Anosov subgroups which is related to ours. They look at a space of the form $\g/\h$ (where $\g$ is a semisimple Lie group and $\h\subset\g$ consists of fixed points of an involution of $\g$), and prove a counting theorem for the \textit{polar projection} (see e.g. \cite[Section 7]{Sch}) of discrete $\Gamma$-orbits\footnote{The authors do not assume that the intersection $\rho(\Gamma)\cap\h$ is finite.} in $\g/\h$. The norm of the polar projection of an element $g\in\g$ can be interpreted as the distance between a basepoint in $\mathsf{S}^\h$ and the submanifold $g\cdot\mathsf{S}^\h$, where $\mathsf{S}^\h\subset\xsyg$ is a totally geodesic copy of the Riemannian symmetric space of $\h$ (see C. \cite[Proposition 1.4.6]{CarThesis}). Here by ``norm'' we mean the one induced by a $\g$-invariant Riemannian structure on $\xsyg$ and part of the results of \cite{EdwLeeOh} include, in some specific situations, counting theorems with respect to this norm (see \cite[Theorem 1.11]{EdwLeeOh} for precisions).

We mention here that counting problems for the polar projection and also the Main Problem of the present paper have been studied as well in C. \cite{CarHpq} for $\g=\tn{PSO}(p,q)$, $\h=\tn{PSO}(p,q-1)$ and $\rho:\Gamma\too\g$ a \textit{projective} Anosov representation.

\subsection*{Plan of the paper}

Sections \ref{sec quad forms of fixed signautre} and \ref{sec weyl chambers and generic flags} are mainly intended to fix terminology and notations. Proposition \ref{prop param of open orbits by lieb compatible flags} is the most important result of those sections, as it will be helpful in the study of the $(p,q)$-Cartan projection. The $(p,q)$-Cartan decomposition is introduced in Section \ref{section pq-Cartan} and in Section \ref{sec pq Cartan for elements with gaps} we begin the study of the associated projection. Notably, the contents of Subsection \ref{subsec linear alg interp of bo and first estimates} will be of central importance for the rest of the paper (we prove the estimate (\ref{eq bo and cartan introduction}) and the equality (\ref{eq bo and jordan introduction})). In Section \ref{sec busemann cocycles and gromov} we introduce the vector valued $o$-Busemann cocycle and study some basic properties. Corollary \ref{cor bound for counting problem introduction} is proved in Subsection \ref{subsec exponential rate} and Theorem \ref{TEO CONTEO DIRECCIONAL BO EN QPQ INTRODUCCION} is proved in Section \ref{sec counting}. In Subsection \ref{subsec def anosov} we recall the definition and main properties of Anosov representations, and in Appendix \ref{app proof of ditribution} we explain how Sambarino's results \cite{Sam,Sam2} still hold in our setting. 

Dependence between sections is as follows:

\begin{center}
\begin{tikzpicture}
\node[left] (A) at (-5.8,-3) {\ref{sec quad forms of fixed signautre}};
\node[left] (B) at (-4.8,-3) {\ref{sec weyl chambers and generic flags}};
\node[left] (C) at (-3.8,-3) {\ref{section pq-Cartan}};
\node[left] (D) at (-2.8,-3) {\ref{sec pq Cartan for elements with gaps}};
\node[left] (E) at (-1.4,-2.5) {\ref{sec busemann cocycles and gromov}};
\node[left] (F) at (-2.7,-4.2) {\ref{subsec def anosov}};
\node[left] (G) at (-1.2,-3.5) {\ref{subsec asymptotic cone}};
\node[left] (H) at (0.6,-3) {\ref{sec counting}};
\node[left] (I) at (-0.2,-4.2) {\ref{app proof of ditribution}};
\node[left] (J) at (-0.7,-5) {\ref{subsec orho} to \ref{subsec exponential rate}};
\draw[->] (A) -- (B);
\draw[->] (B) -- (C);
\draw[->] (C) -- (D);
\draw[->] (D) -- (J);
\draw[->] (D) -- (G);
\draw[->] (E) -- (H);
\draw[->] (G) -- (H);
\draw[->] (I) -- (H);
\draw[->] (F) -- (J);
\draw[->] (F) -- (G);
\draw[->] (F) -- (I);
\end{tikzpicture}
\end{center}

\subsection*{Acknowledgements}

I am deeply grateful to Rafael Potrie and Andr\'es Sambarino for numerous helpful discussions and comments, as well as for their guidance and continued support and encouragement. The author would also like to thank Beatrice Pozzetti, Florian Stecker and Anna Wienhard for interesting discussions, and Hee Oh for helpful comments on a draft version of this paper. I also want to express my gratitude to the referees of this paper for careful reading and valuable comments, and in particular for asking if our results could hold over $\cc$. Finally I thank Pierre-Louis Blayac for pointing out a detail in the Appendix of this paper.

\section{Forms of fixed signature}\label{sec quad forms of fixed signautre}

From now on we fix a real or complex vector space $V$ of dimension $d\geq 3$ and denote by $\g$ the group $\tn{PSL}(V)$ of projectivized elements of $\tn{SL}(V)$. We let $\lieg$ be the Lie algebra of $\g$ and $\kappa:\lieg\times\lieg\too\rr$ be the Killing form of $\lieg$\footnote{When $\kk=\cc$, we compute the Killing form using the trace form associated to the underlying real structure of $\lieg$.}.

\subsection{Notations and terminology}\label{subsec notations}

Consider two non negative integers $p$ and $q$ such that $p+q=d$. A \textit{form of signature} $(p,q)$ on $V$ is a quadratic or Hermitian form of that signature, depending respectively on whether $\kk=\rr$ or $\kk=\cc$. We denote by $\xsypq$ the space of homothety classes of forms of signature $(p,q)$ on $V$, where two forms are said to be \textit{homothetic} if they differ by multiplication by a positive real number. Note that $\g$ acts on $\xsypq$ in a transitive way.

\subsubsection{\tn{\textbf{Structure of symmetric space}}}\label{subsub structure of symm space}

Let $o$ be point in $\xsypq$ and $\langle\cdot,\cdot\rangle_o$ be the form associated to a representative of $o$ . Since $o$ is non degenerate, we can define the $o$-\textit{adjoint operator}

\bc
$\leftidx{^{\star_o}}\cdot: \mathfrak{gl}(V)\too\mathfrak{gl}(V) $
\ec

\noindent where, for $T\in\mathfrak{gl}(V)$,  $\leftidx{^{\star_o}}T$ is the unique linear transformation of $V$ that satisfies the equality

\bc
$\langle T\cdot u,v\rangle_o=\langle u,\leftidx{^{\star_o}}T\cdot v\rangle_o$
\ec

\noindent for all $u$ and $v$ in $V$. Note that the $o$-adjoint $\leftidx{^{\star_o}}T$ does not depend on the choice of the representative $\langle\cdot,\cdot\rangle_o$ of $o$. Moreover, $\leftidx{^{\star_o}}\cdot$ preserves $\tn{SL}(V)$ and descends to a map $\g\too\g$, that we still denote by $\leftidx{^{\star_o}}\cdot$. Define $\sigma^o$ to be the involutive automorphism of $\g$ given by

\bc
$\sigma^o(g):=\adg{^{-1}}$.
\ec

\noindent Then $\xsypq$ identifies with $\g/\h^o$, where $\h^o$ is the subgroup of $\g$ consisting of fixed points of the involution $\sigma^o$. The space $\xsypq$ is then a \textit{symmetric space} of $\g$. 

Let $d\sigma^o$ be the derivative of $\sigma^o$ at the identity element of $\g$. The map

\bc
$X\mapsto \left.\frac{d}{dt}\right\vert_{t=0}\exp(tX)\cdot o$
\ec

\noindent gives a $\g$-equivariant identification between

\bc
$\lieq^o:=\lbrace d\sigma^o=-1\rbrace$
\ec

\noindent and the tangent space to $\xsypq$ at the point $o$. If $\lieh^o$ denotes the subalgebra of $\lieg$ consisting of fixed points of $d\sigma^o$, one has the following decomposition

\bc

$\lieg=\lieh^o\oplus\lieq^o$.

\ec

\noindent This decomposition is orthogonal with respect to the Killing form $\kappa$.

\begin{rem}\label{rem invariant form on exterior powers}
For $j=1,\dots,d$ consider the group $S_j$ of permutations of $\lbrace 1,\dots,j\rbrace$. Let $\varepsilon(\omega)$ denote the sign of $\omega\in S_j$. Let $\Lambda^j$ be the $j^{\tn{th}}$-exterior power representation of $\g$. Then the form on $\Lambda^jV$ defined by

\bc
$\langle v_1\wedge\dots\wedge v_j,  v_1'\wedge\dots\wedge v_j'\rangle_{o_j}:=\dfrac{1}{j!}\displaystyle\sum_{\omega\in S_j}\varepsilon(\omega)\displaystyle\prod_{i=1}^j\langle v_i, v_{\omega(i)}'\rangle_o$
\ec

\noindent is non degenerate and invariant under the action of $\Lambda^j \h^o$. Denote by $o_j$ the ray containing the form associated to $\langle\cdot,\cdot\rangle_{o_j}$ and observe that for every $g\in\g$ one has
\begin{equation}\label{eq sigma on lambda i commutes with sigma}
\sigma^{o_j}(\Lambda^jg)=\Lambda^j\sigma^o(g).
\end{equation}
\end{rem}

\subsubsection{\tn{\textbf{The Riemannian symmetric space}}}

A \textit{Cartan involution} of $\lieg$ is an $\rr$-bilinear involutive automorphism $\tau:\lieg\too\lieg$ for which the form on $\lieg$ given by

\bc
$(X,Y)\mapsto-\kappa(X,\tau\cdot Y)$
\ec

\noindent is positive definite. The Lie group $\g$ acts on the space $\xsyg$ of Cartan involutions of $\lieg$ in a transitive way and the stabilizer $\ko^\tau$ of a point $\tau$ in $\xsyg$ is a maximal compact subgroup of $\g$ (see Knapp \cite[Corollary 6.19 and Theorem 6.31]{Kna}). In particular, the space $\xsyg$ can be endowed with a $\g$-invariant Riemannian metric which is necessarily non positively curved (see Helgason \cite[Theorem 4.2 of Ch. IV]{Hel}). We fix once and for all such a Riemannian metric and call the space $\xsyg$ the \textit{Riemannian symmetric space} of $\g$. We let $d_{\xsy}(\cdot,\cdot)$ denote the ($\g$-invariant) distance on $\xsy$ induced by the Riemannian structure. For a point $\tau\in\xsyg$ we define

\bc
$\liep^{\tau}:=\lbrace \tau=-1\rbrace$ and $\liek^{\tau}:=\lbrace \tau=1\rbrace$.
\ec

\noindent The Riemannian structure on $\xsyg$ induces a Euclidean norm on $\liep^\tau$.

\begin{rem}\label{rem Qd is xsyg and cartan subspace is a basis of lines}

Note that one has natural identifications

\bc
$\mathsf{Q}_{0,d}\cong\mathsf{Q}_{d,0}\cong\xsyg$.
\ec

\noindent Indeed, for $o\in\mathsf{Q}_{d,0}$ we have that $d\sigma^o\in\xsy$. For this reason (and to avoid confusions), from now on the notation $\xsypq$ will always mean that $p$ and $q$ are positive. However, we will identify points in $\xsyg$ with homothety classes of inner products on $V$ whenever convenient.
\end{rem}

\subsubsection{\tn{\textbf{Cartan subspaces and flats}}}

Fix a basepoint $o\in\xsypq$. There exist Cartan involutions $\tau$ of $\lieg$ that commute with $d\sigma^o$ and two of them differ by the action of $\h^o_0$, the connected component of $\h^o$ containing the identity element (see Matsuki \cite[Lemmas 3 and 4]{Mat}).

Take a point $\tau\in\xsyg$ for which $\tau$ and $d\sigma^o$ commute and let $\lieb\subset\liep^{\tau}\cap\lieq^o$ be a (necessarily abelian) maximal subalgebra. The norm on $\lieb$ induced by the Riemannian structure on $\xsyg$ is denoted by $\Vert\cdot\Vert_\lieb$. For all $X\in\lieb$ one has
\begin{equation}\label{eq definition normlieb and distance on xsyg}
d_{\xsyg}(\tau,\exp(X)\cdot\tau)=\Vert X\Vert_\lieb.
\end{equation}

Let us now describe a maximal subalgebra $\lieb\subset\liep^\tau\cap\lieq^o$ in a more concrete way. In the process we fix some terminology that will remain valid for the rest of the paper. 

Fix a representative $\langle\cdot,\cdot\rangle_o$ of $o$. For a subspace $\pi$ of $V$ denote 

\bc
$\pi^{\perp_o}:=\lbrace v\in V:\hspace{0,3cm} \langle v,u\rangle_o=0 \tn{ for all } u\in\pi\rbrace$.
\ec

\noindent A set is said to be $o$-\textit{orthogonal} if its elements are pairwise orthogonal with respect to the form $\langle\cdot,\cdot\rangle_o$. The $o$-\textit{sign} of a vector $v$ in $V$ is defined by

\bc
$\sgo(v):=\left\lbrace\begin{array}{l}
1 \hspace{0,5cm}\tn{ if } \langle v,v\rangle_o>0\\
-1 \hspace{0,2cm}\tn{ if } \langle v,v\rangle_o<0\\
0 \hspace{0,5cm}\tn{ if } \langle v,v\rangle_o=0
\end{array}
\right.$.
\ec 

\noindent This vector is said to be $\langle\cdot,\cdot\rangle_o$-\textit{unitary} if

\bc
$\langle v,v\rangle_o\in\lbrace - 1,1\rbrace$.
\ec

\noindent A basis $\mathcal{B}$ of $V$ is said to be $\langle\cdot,\cdot\rangle_o$-\textit{orthonormal} if it is $o$-orthogonal and its elements are $\langle\cdot,\cdot\rangle_o$-unitary. Finally, the $o$-\textit{sign} of a line $\ell$ in $V$ is defined in an analogous way and a \textit{basis of lines} of $V$ is a set of $d$ lines in $V$ that span $V$.

\begin{ex} \label{ex explicit example Qpq}
Let $\mathcal{B}$ be a $\langle\cdot,\cdot\rangle_o$-orthonormal basis of $V$ and $\langle\cdot,\cdot\rangle$ be the inner product of $V$ for which this basis is orthonormal. The associated point in $\xsyg$ will be denoted by $\tau$ and we emphasize the link between the inner product $\langle\cdot,\cdot\rangle$ and the point $\tau$ by denoting $\langle\cdot,\cdot\rangle_\tau:=\langle\cdot,\cdot\rangle$. It can be seen that $d\sigma^o$ and $\tau$ commute.

Pick $\lieb$ to be the subset of $\lieg$ consisting of elements which are diagonal in the basis $\mathcal{B}$, with real eigenvalues. Since $\mathcal{B}$ is both $\langle\cdot,\cdot\rangle_o$-orthonormal and $\langle\cdot,\cdot\rangle_\tau$-orthonormal, one has

\bc
$\lieb\subset \liep^\tau\cap\lieq^o$
\ec 

\noindent and one can see that $\lieb$ is a maximal subalgebra in $\liep^\tau\cap\lieq^o$.
\end{ex}

From Example \ref{ex explicit example Qpq} we conclude that maximal subalgebras $\lieb\subset\liep^\tau\cap\lieq^o$ are in fact maximal in $\liep^\tau$, that is, they are \textit{Cartan subspaces} of $\lieg$. It is standard to denote Cartan subspaces of $\lieg$ with the symbol $\liea$ instead of $\lieb$. We will use the notation $\lieb$ if we want to emphasize that this is a maximal subalgebra in $\liep^\tau\cap\lieq^o$ and use the notation $\liea:=\lieb$ if we want to emphasize that this subalgebra is maximal in $\liep^\tau$ and apply the classical theory to it.

As outlined in Example \ref{ex explicit example Qpq}, the space of Cartan subspaces of $\lieg$ is in natural bijection with the space of bases of lines of $V$. A Cartan subspace is contained in $\liep^\tau$ (resp. $\lieq^o$) if and only if the corresponding basis of lines is $\tau$-orthogonal (resp. $o$-orthogonal).

A \textit{flat} in $\xsyg$ is a maximal dimensional totally geodesic submanifold of $\xsy$ on which sectional curvature vanishes. The space of flats in $\xsyg$ (through the basepoint $\tau$) is in one to one correspondence with the space of Cartan subspaces of $\lieg$ (contained in $\liep^\tau$). Concretely, if $\liea$ is a Cartan subspace of $\lieg$ (contained in $\liep^\tau$) and $\mathcal{C}$ is the associated ($\tau$-orthogonal) basis of lines of $V$, then

\bc
$\lbrace \tau'\in\xsy:\hspace{0,3cm} \mathcal{C} \tn{ is } \tau'\tn{-orthogonal} \rbrace$
\ec

\noindent is a flat in $\xsyg$. It coincides with $\exp(\liea)\cdot\tau$.

\subsection{The submanifold $\so$}\label{subsec so}

Let

\bc
$\so:=\lbrace \tau\in\xsyg:\hspace{0,3cm} \tau (d\sigma^o)=(d\sigma^o) \tau\rbrace$.
\ec

\noindent Concretely, an homothety class of inner products $\tau$ belongs to $\so$ if and only if there exist representatives $\langle\cdot,\cdot\rangle_o$ of $o$ and $\langle\cdot,\cdot\rangle_\tau$ of $\tau$ and a basis $\mathcal{B}$ of $V$ which is both $\langle\cdot,\cdot\rangle_o$-orthonormal and $\langle\cdot,\cdot\rangle_\tau$-orthonormal.

The Riemannian symmetric space $\mathsf{X}_{\h^o}$ of $\h^o$ can be identified with the space of $q$-dimensional subspaces of $V$ on which the form $o$ is negative definite. We then find an isometry 

\bc
$\so\too\mathsf{X}_{\h^o}$
\ec

\noindent that sends each $\tau\in\so$ to the subspace of $V$ spanned by the vectors of $\mathcal{B}$ which are negative for the form $o$, where $\mathcal{B}$ is a basis of $V$ as in the above paragraph. It follows also that $\so$ is $\h^o$-invariant and $\so=\h^o_0\cdot \tau$ for any $\tau\in\so$. In particular, the tangent space to $\so$ at $\tau$ identifies with

\bc
$\liep^\tau\cap\lieh^o$
\ec

\noindent and $\so$ is totally geodesic (see e.g. Helgason \cite[Theorem 7.2 of Ch. IV]{Hel}). We deduce the following.

\begin{cor}\label{cor orthogonal flats to so}

Let $\tau$ be a point in $\so$ and $\liea\subset \liep^\tau$ be a maximal subalgebra. Then the following are equivalent:

\begin{enumerate}
\item The flat $\exp(\liea)\cdot\tau$ is orthogonal to $\so$ at $\tau$.
\item The inclusion $\liea\subset\liep^\tau\cap\lieq^o$ holds.
\item The basis of lines $\mathcal{C}$ associated to $\liea$ is $o$-orthogonal (and $\tau$-orthogonal).
\end{enumerate}
\end{cor}

\subsection{Generic flags}

For $j=1,\dots,d$ we denote by $\mathsf{Gr}_j(V)$ the \textit{Grassmannian} of $j$-dimensional subspaces of $V$. Denote by $\mathsf{F}(V)$ the space of \textit{complete} (or \textit{full}) \textit{flags} of $V$, that is,

\bc
$\mathsf{F}(V):=\lbrace \xi=(\xi^1\subset\dots\subset\xi^d): \hspace{0,3cm} \xi^j\in\mathsf{Gr}_j(V) \tn{ for every }j=1,\dots,d
\rbrace$.
\ec

\noindent Two complete flags $\xi_1$ and $\xi_2$ are said to be \textit{transverse} if for every $j=1,\dots,d$ the subspace $\xi_1^j$ is linearly disjoint from $\xi_2^{d-j}$. Equivalently, recall that for every $j=1,\dots,d$ there are equivariant maps

\bc
$\Lambda^j:\mathsf{F}(V)\too\pp(\Lambda^jV)$ and $\Lambda_*^j:\mathsf{F}(V)\too\pp(\Lambda^jV^*)$
\ec

\noindent into the projective spaces of $\Lambda^jV$ and $\Lambda^jV^*$ respectively. Then $\xi_1$ is transverse to $\xi_2$ if and only if for every $j=1,\dots,d-1$ the line $\Lambda^j(\xi_1)$ is linearly disjoint from the hyperplane\footnote{As usual, for a finite dimensional vector space $W$ one identifies $\pp(W^*)$ with the Grassmannian of codimension-one subspaces of $W$ by the map $\vartheta\mapsto \ker\vartheta$.} $\Lambda^j_*(\xi_2)$. The space of ordered pairs of transverse full flags of $V$ is denoted by $\fvc$.

\subsubsection{\tn{\textbf{Genericity of flags with respect to a basepoint}}}

In this subsection we introduce the notion of $o$-\textit{generic} flag and discuss some characterizations. This notion is introduced by means of an involution

\bc
$\cdot^{\perp_o}:\mathsf{F}(V)\too\mathsf{F}(V)$
\ec

\noindent which is defined in the following way. Given a full flag $\xi=(\xi^1,\dots,\xi^d)$ in $\mathsf{F}(V)$ we denote by $\xi^{\perp_o}$ the complete flag of $V$ defined by the equalities

\bc
$(\xi^{\perp_o})^j:=({\xi^{d-j}})^{\perp_o}$
\ec

\noindent for $j=1,\dots,d$.

The following lemma is direct.

\begin{lema}\label{lema action of sigma on flags}

For every $g$ in $\g$ and every $\xi$ in $\mathsf{F}(V)$ one has

\bc
$\sigma^{o}(g)\cdot(\xi^{\perp_o})=(g\cdot\xi)^{\perp_o}=g\cdot(\xi^{\perp_{g^{-1}\cdot o}})$.
\ec

\noindent In particular, the following holds:

\bc
$(g\cdot\xi)^{\perp_{g\cdot o}}=g\cdot(\xi^{\perp_{o}})$.
\ec
\end{lema}

A full flag $\xi\in\mathsf{F}(V)$ is said to be $o$-\textit{generic} if it is transverse to $\xi^{\perp_o}$. We have the following useful equivalences, which hold by definitions.

\begin{lema}\label{lema equivalences for o generic flags}

Let $\xi=(\xi^1,\dots,\xi^d)$ be an element of $\mathsf{F}(V)$. Then the following are equivalent:

\begin{enumerate}

\item The flag $\xi$ is $o$-generic.
\item For every $j=1,\dots,d$, the restriction of the form $o$ to $\xi^j$ is non degenerate.
\item There exists a unique $o$-orthogonal ordered basis of lines 

\bc
$\lbrace\ell_1^o(\xi),\dots,\ell^o_d(\xi)\rbrace$
\ec

\noindent of $V$ such that the equality

\bc
$\xi^j=\ell_1^o(\xi)\oplus\dots\oplus\ell_j^o(\xi)$
\ec

\noindent holds for every $j=1,\dots,d$.
\item For every $j=1,\dots,d$, the line $\Lambda^j\xi$ is transverse to the hyperplane $(\Lambda^j\xi)^{\perp_{o_j}}$.

\end{enumerate}

\end{lema}

The following remark is useful.

\begin{rem}\label{rem ordered basis induced by o generic and double perp and xi o generic iff gxi go generic}

Fix $j=1,\dots,d$ and let $\xi$ be an $o$-generic flag. Then the following equality holds

\bc
$\Lambda^j_* (\xi^{\perp_o})=(\Lambda^j\xi)^{\perp_{o_j}}$.
\ec

\noindent In particular, if $\xi'\in\mathsf{F}(V)$ is a flag transverse to $\xi$ then the hyperplane $(\Lambda^j(\xi^{\perp_o}))^{\perp_{o_j}}$ is transverse to the line $\Lambda^j\xi'$.
\end{rem}

\subsubsection{\tn{\textbf{Open orbits of point-stabilizers}}}\label{subsub open orbits in flags}

Denote by $\fvo$ the subset of $\mathsf{F}(V)$ consisting of $o$-generic flags. One can see that $\fvo$ coincides with the union of open orbits of the action of $\h^o$ on $\mathsf{F}(V)$.

The proof of the following proposition is direct.

\begin{prop}\label{prop equivalences of two generic flags in same orbit}
Let $\xi$ and $\xi'$ be two $o$-generic flags. Then the following are equivalent:

\begin{enumerate}
\item The flags $\xi$ and $\xi'$ belong to the same orbit of the action $\h^o\curvearrowright\mathsf{F}(V)$.
\item For every $j=1,\dots,d$, the signature of $o$ restricted to $\xi^j$ coincides with the signature of $o$ restricted to $\xi'^j$.
\item For every $j=1,\dots,d$, one has $\sgo(\ell^o_j(\xi))=\sgo(\ell_j^o(\xi'))$.
\item For every $j=1,\dots,d$, one has

\bc
$\tn{sg}_{o_j}(\Lambda^j\xi)=\tn{sg}_{o_j}(\Lambda^j\xi')$.
\ec
\end{enumerate}
\end{prop}

\section{Weyl chambers and generic flags}\label{sec weyl chambers and generic flags}

Fix a point $\tau\in\so$ and a maximal subalgebra $\lieb\subset\liep^\tau\cap\lieq^o$. Let $\mathcal{C}$ be the $o$-orthogonal and $\tau$-orthogonal basis of lines of $V$ determined by this choice.

\subsection{Weyl group}

Recall that $\lieb$ is a maximal subalgebra in $\liep^\tau$ and that we use the notation $\liea:=\lieb$ whenever we want to emphazise this. Let $\wbh:=\tn{N}_{\ko^\tau}(\liea)$ (resp. $\mb:=\tn{Z}_{\ko^\tau}(\liea)$) be the normalizer (resp. centralizer) of $\liea$ in $\ko^\tau$ and let

\bc
$\wb:=\wbh/\mb$
\ec

\noindent be the \textit{Weyl group} of the pair $(\lieg,\liea)$. If $\hat{w}$ belongs to $\wbh$, we denote by $w$ its class in $\wb$. Conversely, for a given $w\in\wb$ we denote by $\hat{w}\in\wbh$ any representative of $w$.

Fix a representative $\langle\cdot,\cdot\rangle_o$ of $o$. Since the involutions $d\sigma^o$ and $\tau$ commute we can find a representative $\langle\cdot,\cdot\rangle_\tau$ of $\tau$ for which the following holds: for each line $\ell\in\mathcal{C}$ and each vector $v\in\ell$ one has

\bc
$\vert\langle v, v\rangle_{o}\vert=\langle v, v\rangle_{\tau}$.
\ec

\noindent From this observation the following technical lemma can be easily deduced.

\begin{lema}\label{lema if preserve mathcalC equivalences belong to H and to K}

Let $\hat{w}$ be an element of $\g$ that preserves the set $\mathcal{C}$. Then the following holds:

\begin{enumerate}
\item If $\hat{w}$ belongs to $\h^o$, then it belongs to $\ko^\tau$.
\item Suppose that $\hat{w}$ belongs to $\ko^\tau$. Then for every line $\ell\in\mathcal{C}$ and every vector $v\in\ell$ one has

\bc
$\vert\langle \hat{w}\cdot v, \hat{w}\cdot v\rangle_{o}\vert=\vert\langle v, v\rangle_{o}\vert$.
\ec

\noindent In particular, if $\hat{w}$ preserves the $o$-sign of each line of $\mathcal{C}$ then $\hat{w}$ belongs to $\h^o$.
\end{enumerate}
\end{lema}

By Lemma \ref{lema if preserve mathcalC equivalences belong to H and to K} we have that $\mb$ coincides with the centralizer of $\lieb$ in $\h^o$ and in particular it is contained in $\h^o$.

\subsection{Restricted roots and Weyl chambers}\label{subsec weyl chambers}

Let $\tilde{\lieg}$ be a subalgebra of $\lieg$ invariant under the (adjoint) action of $\lieb$. A non zero functional $\alpha\in\lieb^*$ is called a \textit{restricted root of $\lieb$ in $\tilde{\lieg}$} if the subspace

\bc
$\tilde{\lieg}_\alpha:=\lbrace Y\in\tilde{\lieg}:\hspace{0,3cm}[X,Y]=\alpha(X)Y \tn{ for all } X\in\lieb\rbrace$
\ec

\noindent is non zero. In that case, $\tilde{\lieg}_\alpha$ is called the associated \textit{root space}. The set of restricted roots of $\lieb$ in $\tilde{\lieg}$ is denoted by $\Sigma(\tilde{\lieg},\lieb)$. A (closed) \textit{Weyl chamber} of this set is the closure of a connected component of

\bc
$\lieb\setminus\displaystyle\bigcup_{\alpha\in\Sigma(\tilde{\lieg},\lieb)}\tn{ker}(\alpha)$.
\ec

\noindent Let $\Sigma^+(\tilde{\lieg},\lieb)$ be the corresponding \textit{positive system}, that is, the set of restricted roots in $\Sigma(\tilde{\lieg},\lieb)$ which are non negative on this Weyl chamber. In this paper we look at Weyl chambers of two different sets of rectricted roots.

\subsubsection{\tn{\textbf{The case $\Sigma(\lieg,\liea)$}}}\label{subsub weyl chambers lieg,liea}

In this case we think $\lieb$ as a maximal subalgebra in $\liep^\tau$ and therefore we denote $\lieb=\liea$. Weyl chambers of the system $\Sigma(\lieg,\liea)$ will be denoted with the symbol $\liea^+$.

Given a line $\ell\in\mathcal{C}$, let $\varepsilon_\ell\in\liea^*$ be the functional that assigns to each element $X\in\liea$ its eigenvalue in the line $\ell$. For different $\ell$ and $\ell'$ in $\mathcal{C}$ let

\bc
$\alpha_{\ell\ell'}(X):=\varepsilon_\ell(X)-\varepsilon_{\ell'}(X)$.
\ec

\noindent Then one has the equality

\bc
$\Sigma(\lieg,\liea)=\lbrace \alpha_{\ell\ell'}:\hspace{0,3cm} \ell\neq\ell' \tn{ in }\mathcal{C}\rbrace$
\ec

\noindent and the choice of a closed Weyl chamber corresponds to the choice of a total order

\bc
$\mathcal{C}=\lbrace \ell_1,\dots,\ell_d\rbrace$
\ec

\noindent on $\mathcal{C}$. If this choice is given we set

\bc
$\varepsilon_j:=\varepsilon_{\ell_j}$ and $\alpha_{ji}:=\varepsilon_j-\varepsilon_i$
\ec
\noindent for each $j$ different from $i$ in $\lbrace 1 ,\dots,d\rbrace$, and the corresponding positive system is

\bc
$\Sigma^+(\lieg,\liea):=\left\lbrace\alpha_{\ell_j^+\ell_i^+}:\hspace{0,3cm} 1\leq j<i\leq d\right\rbrace$.
\ec

\subsubsection{\tn{\textbf{The case $\Sigma(\liegto,\lieb)$}}}
Let 

\bc
$\liegto:=(\liep^{\tau}\cap\lieq^o)\oplus (\liek^{\tau}\cap\lieh^o)$
\ec

\noindent be the ($\lieb$-invariant) Lie algebra consisting of fixed points of the involution $\tau (d\sigma^o)$. As suggested in Schlichtkrull \cite[p. 117]{Sch}, Weyl chambers of the set $\Sigma(\liegto,\lieb)$ will be denoted with the symbol $\lieb^+$.

We have the equality

\bc
$\Sigma(\lieg^{\tau o},\lieb)=\lbrace \alpha_{\ell\ell'}\in\Sigma(\lieg,\liea):\hspace{0,3cm} \ell\neq\ell' \tn{ in } \mathcal{C} \tn{ and } \sgo(\ell)=\sgo(\ell') \rbrace$.
\ec

\noindent Indeed, one has $\Sigma(\lieg^{\tau o},\lieb)\subset\Sigma(\lieg,\liea)$ and therefore a non zero linear functional $\alpha\in\lieb^*$ belongs to $\Sigma(\liegto,\lieb)$ if and only if it belongs to $\Sigma(\lieg,\liea)$ and

\bc
$\liegto\cap\lieg_\alpha\neq\lbrace 0\rbrace$.
\ec

\noindent Now for $\alpha=\alpha_{\ell\ell'}\in\Sigma(\lieg,\liea)$ the associated root space $\lieg_\alpha$ intersects the subalgebra $\liegto$ if and only if $\sgo(\ell)$ coincides with $\sgo(\ell')$ and the claim follows. 

Note then that $\Sigma(\liegto,\lieb)$ is not a root system, because it does not generate $\lieb^*$. However, it is in one to one correspondence with

\bc
$\Sigma(\mathfrak{sl}_p,\liea_{\mathfrak{sl_p}})\sqcup\Sigma(\mathfrak{sl}_q,\liea_{\mathfrak{sl_q}})$.
\ec

An explicit way of prescribing a Weyl chamber $\lieb^+$ is the following. Write $\mathcal{C}=\mathcal{C}^+\sqcup\mathcal{C}^-$ where $\mathcal{C}^+$ (resp. $\mathcal{C}^-$) is the subset of $\mathcal{C}$ consisting of lines which are positive (resp. negative) for the form $o$. Fix total orders

\bc
$\mathcal{C}^+=\lbrace \ell_1^+,\dots,\ell_p^+\rbrace$ and $\mathcal{C}^-=\lbrace \ell_1^-,\dots,\ell_q^-\rbrace$
\ec

\noindent on $\mathcal{C}^+$ and $\mathcal{C}^-$ respectively. Then a positive system in $\Sigma(\liegto,\lieb)$ is given by

\bc
$\Sigma^+(\lieg^{\tau o},\lieb):=\left\lbrace\alpha_{\ell_j^+\ell_i^+}:\hspace{0,3cm} 1\leq j<i\leq p\right\rbrace\sqcup\left\lbrace\alpha_{\ell_j^-\ell_i^-}:\hspace{0,3cm} 1\leq j<i\leq q\right\rbrace$.
\ec

\noindent Conversely, the choice of a Weyl chamber $\lieb^+$ induces total orders on $\mathcal{C}^+$ and $\mathcal{C}^-$.

Note that a Weyl chamber $\liea^+$ is contained in $\lieb^+$ if and only if the total order on $\mathcal{C}$ induced by $\liea^+$ restricts to the total orders on $\mathcal{C}^+$ and $\mathcal{C}^-$ determined by $\lieb^+$. In particular, each $\lieb^+$ contains $\left(\begin{smallmatrix}
d\\
p
\end{smallmatrix}\right)=\frac{d!}{p!(d-p)!}$ Weyl chambers of $\Sigma(\lieg,\liea)$  (c.f. Figure \ref{fig weyl chamber}).

\begin{figure}[h!]
\bc
\scalebox{0.5}{%
\begin{overpic}[scale=1, width=1\textwidth, tics=5]{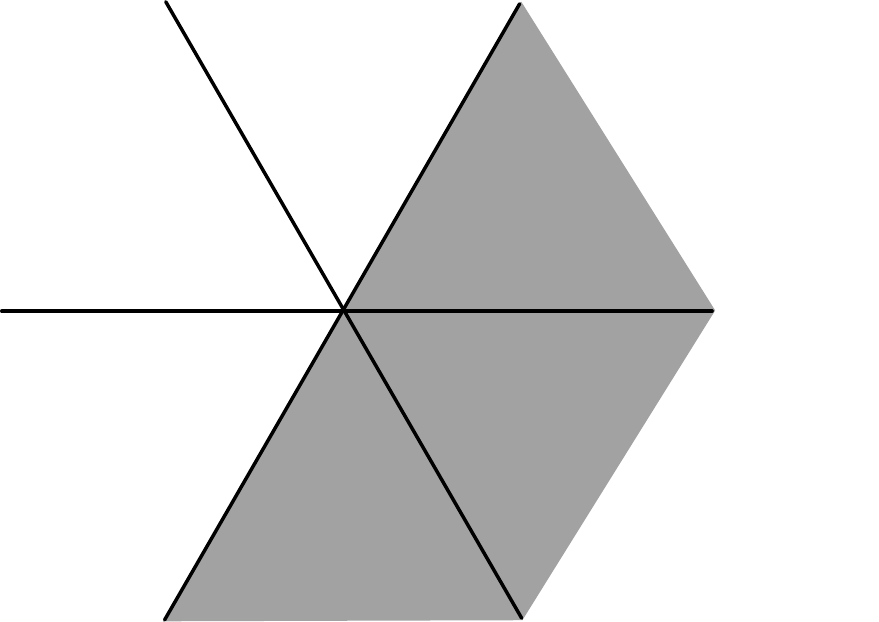}

\put (55,45) { \begin{Huge}
$\liea_1^+$
\end{Huge}}

\put (55,20) { \begin{Huge}
$\liea_2^+$
\end{Huge}}

\put (32,10) { \begin{Huge}
$\liea_3^+$
\end{Huge}}
  \end{overpic}}
 
\hspace{0,3cm}

\begin{changemargin}{3cm}{3cm}    
    \caption{Weyl chambers for $\mathsf{Q}_{2,1}$. In light grey, a Weyl chamber $\lieb^+$ of the set $\Sigma(\liegto,\lieb)$. This Weyl chamber is a union of three Weyl chambers $\liea^+_1$, $\liea^+_2$ and $\liea^+_3$ of the system $\Sigma(\lieg,\liea)$.}
  \label{fig weyl chamber}
\end{changemargin}

  \ec
\end{figure}

\begin{rem}
Even though it will not be used in the future, we mention here that Weyl chambers $\lieb^+$ of the set $\Sigma(\liegto,\lieb)$ admit the following geometric interpretation: a vector $X$ belongs to the interior $\tn{int}(\lieb^+)$ of $\lieb^+$ if and only if the flat $\exp(\lieb)\cdot\tau$ is the unique flat of $\xsyg$ that contains the geodesic $\exp(\rr X)\cdot\tau$ and that is orthogonal to $\so$ at $\tau$.
\end{rem}

\subsection{Opposition involution of $\lieb^+$}\label{subsec oppostion liebplus}

For the rest of the section we fix a closed Weyl chamber $\lieb^+$ of the set $\Sigma(\liegto,\lieb)$. Let $w_{\lieb^+}\in\wb$ be the unique element that preserves $\mathcal{C}^{+}$ and $\mathcal{C}^-$ and acts on these sets by reversing the total order induced by $\lieb^+$. By Lemma \ref{lema if preserve mathcalC equivalences belong to H and to K} we have that $\hat{w}_{\lieb^+}$ belongs to $\h^o$ and, by definition, it satisfies

\bc
$w_{\lieb^+}\cdot(-\lieb^+)=\lieb^+$.
\ec

\noindent The \textit{opposition involution} of $\lieb^+$ is defined by

\bc
$\iota_{\lieb^+}:\lieb\too\lieb: \hspace{0,3cm} \iota_{\lieb^+}(X):=-w_{\lieb^+}\cdot X$.
\ec

\noindent Note that $\iota_{\lieb^+}$ preserves $\lieb^+$.

\subsection{Compatible Weyl chambers and generic flags}\label{subsec compatible weyl and o generic}

A $\lieb^+$-\textit{compatible Weyl chamber} is a Weyl chamber of the system $\Sigma(\lieg,\liea)$ that is contained in $\lieb^+$.

\begin{lema}\label{lema every standard weyl chamber taken into a liebplus compatile by wcaph}
Let $\liea^+\subset\lieb$ be any Weyl chamber of the system $\Sigma(\lieg,\liea)$. Then there exists an element $\hat{w}\in\wbh\cap\h^o$ such that 

\bc
$w\cdot\liea^+\subset\lieb^+$.
\ec

\noindent Furthermore, if $\hat{w}'\in\wbh\cap\h^o$ satisfies $w'\cdot\liea^+\subset\lieb^+$ then $\hat{w}$ and $\hat{w}'$ define the same element in $\wb$.
\end{lema}

\begin{proof}
Let $\lbrace\ell_1,\dots,\ell_d\rbrace$ be the total order on $\mathcal{C}$ induced by the choice of $\liea^+$ and
\begin{equation}\label{eq total orders Cplus and Cminus}
\mathcal{C}^+=\lbrace\ell_1^+,\dots,\ell_p^+\rbrace\tn{ and }\mathcal{C}^-=\lbrace\ell_1^-,\dots,\ell_q^-\rbrace
\end{equation}

\noindent be the total orders on $\mathcal{C}^+$ and $\mathcal{C}^-$ determined by the choice of $\lieb^+$. 

We define the element $\hat{w}$ inductively. If $\sgo(\ell_1)=1$ we define $\hat{w}\cdot\ell_1$ to be $\ell_1^+$, and if $\sgo(\ell_1)=-1$ we define $\hat{w}\cdot\ell_1$ to be $\ell_1^-$. Say that $\sgo(\ell_1)=1$ (the other case being analogous). Now define $\hat{w}\cdot\ell_2$ to be $\ell_2^+$ if $\sgo(\ell_2)=1$, or to be $\ell_1^-$ if $\sgo(\ell_2)=-1$. Procede inductively to obtain an element $\hat{w}\in\wbh$ such that the restrictions of the total order

\bc
$\lbrace\hat{w}\cdot\ell_1,\dots,\hat{w}\cdot\ell_d\rbrace$
\ec

\noindent to both $\mathcal{C}^+$ and $\mathcal{C}^-$ coincide with the total orders of equation (\ref{eq total orders Cplus and Cminus}). That is, the Weyl chamber $w\cdot\liea^+$ is $\lieb^+$-compatible. Furthermore, by construction we have $\sgo(\hat{w}\cdot\ell_j)=\sgo(\ell_j)$ for every $j=1,\dots,d$. Lemma \ref{lema if preserve mathcalC equivalences belong to H and to K} implies $\hat{w}\in\h^o$.

On the other hand, suppose that $\hat{w}'\in\wbh\cap\h^o$ satisfies $w'\cdot\liea^+\subset\lieb^+$. Say $\sgo(\ell_1)=1$ (the other case being analogous). Then since $\hat{w}'\in\h^o$ we have $\sgo(\hat{w}'\cdot\ell_1)=1$. Furthermore, since $w'\cdot\liea^+\subset\lieb^+$ the total order

\bc
$\lbrace\hat{w}'\cdot\ell_1,\dots,\hat{w}'\cdot\ell_d\rbrace$
\ec

\noindent on $\mathcal{C}$ must restrict to the total order of $\mathcal{C}^+$ given by equation (\ref{eq total orders Cplus and Cminus}). We see then that $\hat{w}'\cdot\ell_1$ must be equal to $\ell_1^+=\hat{w}\cdot\ell_1$. Proceeding inductively we conclude that $\hat{w}\cdot\ell_j=\hat{w}'\cdot\ell_j$ for all $j=1,\dots,d$, and this completes the proof.
\end{proof}

Recall that there exists a $\wb$-equivariant identification between the set of Weyl chambers of the system $\Sigma(\lieg,\liea)$ and the set of ($o$-generic) flags determined by the lines of $\mathcal{C}$. If $\liea^+\subset\lieb$ is such a Weyl chamber the corresponding flag is denoted by $\xi_{\liea^+}$. Conversely, the Weyl chamber of the system $\Sigma(\lieg,\liea)$ determined by a flag $\xi$ spanned by $\mathcal{C}$ will be denoted by $\liea^+_\xi$. A flag $\xi\in\mathsf{F}(V)$ is said to be $\lieb^+$-\textit{compatible} if it is spanned by the elements of $\mathcal{C}$ and the Weyl chamber $\liea^+_{\xi}$ is $\lieb^+$-compatible.

The following is a concrete instance of a result due to Matsuki \cite[Section 3]{Mat} and Rossmann \cite[Theorem 13 and Corollaries 15 to 17]{Ross}.

\begin{prop}\label{prop param of open orbits by lieb compatible flags}
The map

\bc
$\xi\mapsto \fvo_\xi:= \h^o\cdot \xi$
\ec

\noindent defines a one to one correspondence between the set of $\lieb^+$-compatible flags and the set consisting on open orbits of the action $\h^o\curvearrowright\mathsf{F}(V)$. 
\end{prop}

For a $\lieb^+$-compatible Weyl chamber $\liea^+$ we use the notation

\bc
$\fvo_{\liea^+}:=\h^o\cdot\xi_{\liea^+}$.
\ec

\begin{proof}[Proof of Proposition \ref{prop param of open orbits by lieb compatible flags}]

Let $\xi$ and $\xi'$ be two $\lieb^+$-compatible flags in the same $\h^o$-orbit. By Proposition \ref{prop equivalences of two generic flags in same orbit} we have $\sgo(\ell_j^o(\xi))=\sgo(\ell_j^o(\xi'))$ for every $j=1,\dots,d$ and, as unordered sets, one has
\bc
$\lbrace\ell_1^o(\xi),\dots,\ell_d^o(\xi)\rbrace=\mathcal{C}=\lbrace\ell_1^o(\xi'),\dots,\ell_d^o(\xi')\rbrace$.
\ec

\noindent By Lemma \ref{lema if preserve mathcalC equivalences belong to H and to K} there exists an element $\hat{w}\in\wbh\cap\h^o$ such that $\hat{w}\cdot\ell_j^o(\xi)=\ell_j^o(\xi')$ for all $j=1,\dots,d$. That is, $w\cdot\liea^+_\xi=\liea^+_{\xi'}$ and Lemma \ref{lema every standard weyl chamber taken into a liebplus compatile by wcaph} implies $w=1$. Hence $\xi=\xi'$.

For surjectivity, let $\xi'$ be any $o$-generic flag. By Lemma \ref{lema equivalences for o generic flags} we can find an element $h\in\h^o$ such that $h\cdot\ell_j^o(\xi')\in\mathcal{C}$ for every $j=1,\dots,d$. That is, $h\cdot\xi'$ determines a Weyl chamber of the system $\Sigma(\lieg,\liea)$. By Lemma \ref{lema every standard weyl chamber taken into a liebplus compatile by wcaph} the proof is complete.

\end{proof}

\section{$(p,q)$-Cartan decomposition}\label{section pq-Cartan}

The choice of a point $\tau\in\xsyg$, a Cartan subspace $\liea\subset\liep^\tau$ and a Weyl chamber $\liea^+$ of $\Sigma(\lieg,\liea)$ induces a \textit{Cartan decomposition} $\g=\ko^\tau\exp(\liea^+)\ko^\tau$ of $\g$ and a \textit{Cartan projection}

\bc
$a^\tau:\g\too\liea^+$
\ec

\noindent (see e.g. Knapp \cite[Chapter VI]{Kna}). Geometrically, for all $g\in\g$ we have

\bc
$d_{\xsyg}(\tau,g\cdot\tau)=\Vert a^\tau(g)\Vert$,
\ec

\noindent where $\Vert\cdot\Vert$ is the norm on $\liep^\tau$ induced by the $\g$-invariant Riemannian structure on $\xsyg$. In this section we describe a way to generalize this picture to our context.

Fix a basepoint $o\in\xsypq$ and let $\bo$ be the set consisting of $o$-orthogonal bases of lines of $V$. Define

\bc
$\bog:=\lbrace g\in\g:\hspace{0,3cm} \exists\hspace{0,1cm}\mathcal{C}\in\bo\tn{ such that }g\cdot\mathcal{C}\in\bo\rbrace$,
\ec

\noindent that is, $\bog$ is the set of elements in $\g$ that take some $o$-orthogonal basis of lines of $V$ into an $o$-orthogonal basis of lines of $V$. We remark here that $\bog$ is different from $\g$ and that it is not open, but has non empty interior (see Remark \ref{rem bog not open but non empty interior} below).

Fix a point $\tau$ in $\so$, a maximal subalgebra $\lieb\subset\liep^\tau\cap\lieq^o$ and a Weyl chamber $\lieb^+\subset\lieb$ of the set $\Sigma(\liegto,\lieb)$. Let $\mathcal{C}\in\bo$ be the element determined by $\lieb$.

\begin{rem}\label{rem sigmawinvw belongs to M}

For every $\hat{w}\in\wbh$ the element $m:=\sigma^o(\hat{w}^{-1})\hat{w}$ belongs to $\mb$. Indeed, as $\tau$ and $d\sigma^o$ commute we have that $m$ belongs to $\ko^\tau$. Furthermore, since $d\sigma^o(X)=-X$ for all $X\in\lieb$ we have

\bc
$\sigma^o(\hat{w})\exp(X)\sigma^o(\hat{w}^{-1})=\sigma^o(\hat{w}\exp(X)^{-1}\hat{w}^{-1})=\hat{w}\exp(X)\hat{w}^{-1}$
\ec

\noindent and this proves the claim.

Note moreover that when $\kk=\cc$, the eigenvalues of $m$ are real (hence equal to $\pm 1$).
\end{rem}

The analogue of the Cartan decomposition is the following.

\begin{prop}\label{prop equivalences HWbH}

Let $g$ be an element of $\g$. Then the following are equivalent:

\begin{enumerate}
\item The element $g$ belongs to $\bog$.
\item The element $g$ belongs to $\h^o\wbh\exp(\lieb^+)\h^o$.
\item The element $\sigma^o(g^{-1})g$ is diagonalizable with real eigenvalues\footnote{Formally, elements of $\g$ are not linear transformations of $V$ but rather projective classes of linear transformations. Nevertheless, we can say that $g\in\g$ is \textit{diagonalizable} if it preserves the elements of some basis of lines of $V$. We say that this basis of lines \textit{diagonalizes} the projective class $g$.}.
\end{enumerate}

\noindent In this case, let $\tilde{\mathcal{C}}$ be an element of $\bo$. Then $g\cdot\tilde{\mathcal{C}}$ belongs to $\bo$ if and only if $\tilde{\mathcal{C}}$ diagonalizes $\sigma^o(g^{-1})g$.

\end{prop}

A decomposition $g=h\hat{w}\exp(X)\tilde{h}$ of an element $g\in\bog$, where $h,\tilde{h}\in\h^o$, $\hat{w}\in\wbh$ and $X\in\lieb^+$ will be called a $(p,q)$-\textit{Cartan decomposition} of $g$.

\begin{proof}[Proof of Proposition \ref{prop equivalences HWbH}]

Suppose first that $g$ belongs to $\bog$ and let $\tilde{\mathcal{C}}\in\bo$ be such that $g\cdot \tilde{\mathcal{C}}\in\bo$. Write $\tilde{\mathcal{C}}=\tilde{\mathcal{C}}^{+}\sqcup\tilde{\mathcal{C}}^{-}$ the decomposition of $\tilde{\mathcal{C}}$ into positive and negative lines for the form $o$. There exists $h\in\h^o$ such that $h^{-1}g\cdot\tilde{\mathcal{C}}=\mathcal{C}$. Further, we can take $\hat{w}$ in $\wbh$ such that $\hat{w}^{-1}h^{-1}g\cdot\tilde{\mathcal{C}}^{\pm}=\mathcal{C}^{\pm}$. Now let $\tilde{h}\in\h^o$ be such that $\tilde{h}^{-1}\cdot\mathcal{C}=\tilde{\mathcal{C}}$. We can assume that $\hat{w}^{-1}h^{-1}g\tilde{h}^{-1}$ fixes each line of $\mathcal{C}$ and therefore one has

\bc
$\hat{w}^{-1}h^{-1}g\tilde{h}^{-1}=m\exp(X)$
\ec

\noindent for some $X\in\lieb$ and $m\in \mb$. Since $\mb$ is contained in $\wbh$ and $X$ is conjugate to an element in $\lieb^+$ by an element in $\wbh\cap\h^o$, we conclude that $g$ belongs to  $\h^o\wbh\exp(\lieb^+)\h^o$.

Now suppose that $g$ admits a $(p,q)$-Cartan decomposition $g=h\hat{w}\exp(X)\tilde{h}$. Then
\begin{equation}\label{eq sigmaginvg in HWbH coordinates}
\sigma^o(g^{-1})g=\tilde{h}^{-1}\exp(X)\sigma^o(\hat{w}^{-1})\hat{w}\exp(X)\tilde{h}.
\end{equation}
\noindent By Remark \ref{rem sigmawinvw belongs to M} we conclude that $\sigma^o(g^{-1})g$ is diagonalizable with real eigenvalues.

Finally, suppose that $\sigma^o(g^{-1})g$ is diagonalizable with real eigenvalues. It is not hard to show that in this case $\sigma^o(g^{-1})g$ must be diagonalizable in an $o$-orthogonal basis of lines $\tilde{\mathcal{C}}$ of $V$. Given $\ell\neq\ell'$ in $\tilde{\mathcal{C}}$ we have
\bc
$\langle  g\cdot\ell ,g\cdot\ell' \rangle_o=\langle  \sigma^o(g^{-1})g\cdot\ell , \ell' \rangle_o=\langle \ell , \ell' \rangle_o=0$,
\ec
\noindent where the above equalities hold up to scalar multiples. Hence $g\cdot\tilde{\mathcal{C}}\in\bo$.

\end{proof}

\begin{rem}\label{rem bog not open but non empty interior}
Take an element $k\in\exp(\liek\cap\lieq)$. Then $\sigma^o(k^{-1})k=k^2$ which, in general, is not diagonalizable with real eigenvalues. This shows that $\bog$ is strictly contained in $\g$ and that it is not open: one can approximate $1\in\bog$ by a sequence $\lbrace k_n\rbrace_n\subset\exp(\liek\cap\lieq)$. However, we will see that $\bog$ has non empty interior (c.f. Remark \ref{rem bog non empty interior} below).

\end{rem}

\subsection{$(p,q)$-Cartan projection}

Define the $(p,q)$-\textit{Cartan projection}

\bc
$b^o:\bog\too\lieb^+$,
\ec

\noindent where $g=h\hat{w}\exp(b^o(g))\tilde{h}$ is a $(p,q)$-Cartan decomposition of $g\in\bog$. We now prove that this map is well defined.

\begin{prop}\label{prop uniqueness of liebplus}
Let $g$ be an element of $\bog$ and write 

\bc
$h_1\hat{w}_1\exp(X_1)\tilde{h}_1=g=h_2\hat{w}_2\exp(X_2)\tilde{h}_2$
\ec

\noindent two $(p,q)$-Cartan decompositions of $g$. Then $X_1=X_2$.
\end{prop}

\begin{proof}

Let $\mu$ be an eigenvalue of $\sigma^o(g^{-1})g$. By Remark \ref{rem sigmawinvw belongs to M} and equation (\ref{eq sigmaginvg in HWbH coordinates}) we have that $\vert\mu\vert$ is an eigenvalue of $\exp(2X_i)$ for $i=1,2$. Moreover if $V^i_{\vert\mu\vert}$ is the corresponding eigenspace, the signature of $o$ restricted to $V_{\vert\mu\vert}^i$ is independent on $i=1,2$. We will show by induction that $V_{\vert\mu\vert}^i$ itself does not depend on $i=1,2$.

Recall that the choice of $\lieb^+$ induces total orders on $\mathcal{C}^+$ and $\mathcal{C}^-$. An element $X\in\lieb$ belongs to $\lieb^+$ if and only if its eigenvalues on the lines of $\mathcal{C}^+$ (resp. $\mathcal{C}^-$) are ordered decreasingly, according to the total order of $\mathcal{C}^+$ (resp. $\mathcal{C}^-$).

Take an eigenvalue $\mu$ of $\sigma^o(g^{-1})g$ such that $\vert\mu\vert$ is maximal and let $(p_{\vert\mu\vert},q_{\vert\mu\vert})$ be the signature of $o$ restricted to $V^i_{\vert\mu\vert}$ (for $i=1,2$). It follows that $V_{\vert\mu\vert}^i$ is the sum of the first $p_{\vert\mu\vert}$ lines of $\mathcal{C}^+$ and the first $q_{\vert\mu\vert}$ lines of $\mathcal{C}^-$. Since this description is independent on $i=1,2$, we conclude that $V_{\vert\mu\vert}^1=V_{\vert\mu\vert}^2$. Proceeding inductively over all eigenvalues of $\sigma^o(g^{-1})g$, the result follows.

\end{proof}

\subsection{Geometric interpretation}

\begin{lema}\label{lema flat orthogonal to So and gSo in HWBH coordinates and distance}
Let $g=h\hat{w}\exp(b^o(g))\tilde{h}$ be a $(p,q)$-Cartan decomposition of $g\in\bog$. Then

\bc
$h\hat{w}\exp(\lieb)\cdot\tau=h\exp(\lieb)\cdot\tau$
\ec

\noindent is a flat of $\xsy$ orthogonal to $\so$ at $h\cdot\tau$ and to $g\cdot \so$ at $h\hat{w}\exp(b^o(g))\cdot\tau$. In particular,

\bc
$d_{\xsy}(\so,g\cdot \so)=\Vert b^o(g)\Vert_\lieb$.
\ec

\end{lema}

\begin{proof}
Indeed, this follows from the fact that $\exp(\lieb)\cdot\tau$ is orthogonal to $\so$ (resp. $\exp(b^o(g))\cdot \so$) at $\tau$ (resp. $\exp(b^o(g))\cdot\tau$) and the fact that $\wbh$ fixes $\tau$ and preserves $\exp(\lieb)\cdot\tau$ (see Figure \ref{fig geom interpretation bo}).
\end{proof}

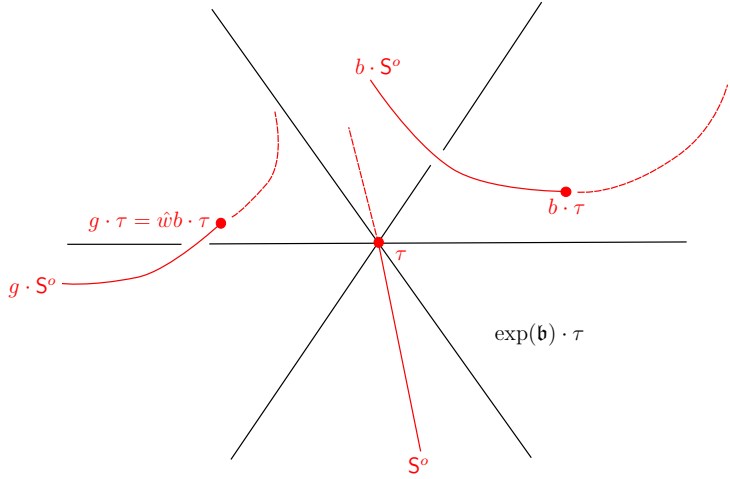
\begin{figure}[h!]
\bc
\scalebox{0.7}{%
\begin{overpic}[scale=1, width=1\textwidth, tics=5]{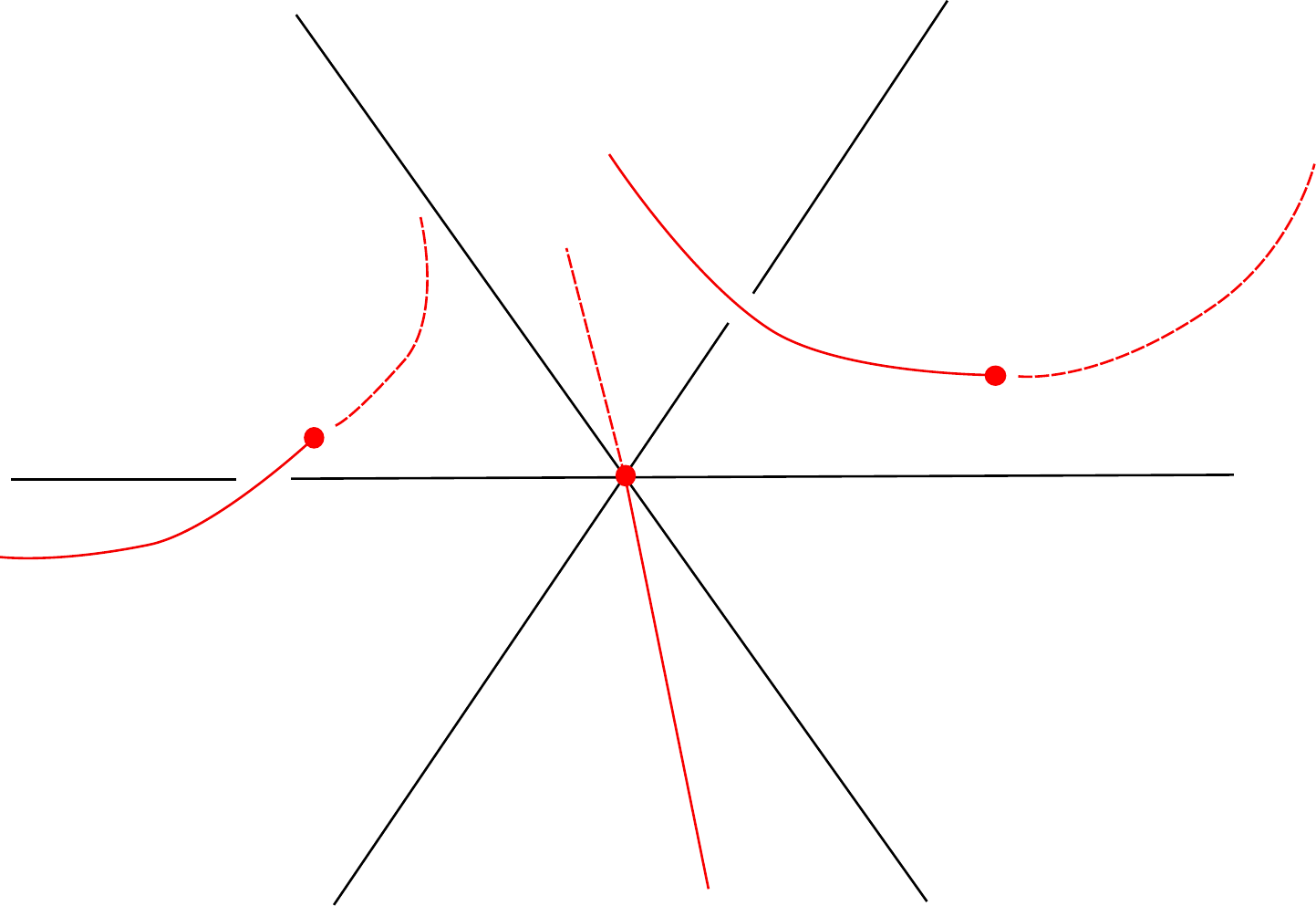}

\put (50,30) {\tc{red}{\Large$\tau$}}

\put (73,37) {\tc{red}{\Large$b\cdot\tau$}}

\put (44,58) {\tc{red}{\Large$b\cdot\so$}}

\put (4,35) {\tc{red}{\Large$g\cdot\tau=\hat{w}b\cdot\tau$}}

\put (-8,25) {\tc{red}{\Large$g\cdot\so$}}

\put (52,-2) {\tc{red}{\Large$\so$}}

\put (65,18) {\Large$\exp(\lieb)\cdot\tau$}
  \end{overpic}}
 
\hspace{0,3cm}

\begin{changemargin}{3cm}{3cm}    
    \caption{The proof of Lemma \ref{lema flat orthogonal to So and gSo in HWBH coordinates and distance} for $g=\hat{w}b\tilde{h}$ (and $b=\exp(b^o(g))$). The red curves represent the submanifolds $\so$, $b\cdot\so$ and $g\cdot\so$, that intersect orthogonally the flat $\exp(\lieb)\cdot\tau$.}
  \label{fig geom interpretation bo}
\end{changemargin}

  \ec
\end{figure}

\section{$(p,q)$-Cartan decomposition for elements with gaps}\label{sec pq Cartan for elements with gaps}

We now begin a more precise study of the $(p,q)$-Cartan decomposition for elements that have ``strong enough dynamics'' on $\mathsf{F}(V)$ and $o$-generic attractor and repellor. Subsection \ref{subsec reminders cartan and jordan} is intended to fix notations and terminology. Subsection \ref{subsec pq attractors} is preparatory to the contents of Subsection \ref{subsec computation weyl chamber and hatw}, where we compute the $\lieb^+$-compatible Weyl chamber that contains $b^o(g)$ for $g$ as above. The contents of Subsection \ref{subsec linear alg interp of bo and first estimates} will be crucial for our understanding of the new projection: we interpret the vector $b^o(g)$ using the Jordan projection of $\sigma^o(g^{-1})g$ and estimate $b^o(g)$ in terms of the Cartan projection $a^\tau(g)$.

\subsection{Reminders on Cartan and Jordan projections}\label{subsec reminders cartan and jordan}

Given a Weyl chamber $\liea^+\subset\lieb=\liea$ of the system $\Sigma(\lieg,\liea)$, let $\Delta\subset\Sigma^+(\lieg,\liea)$ be the set of simple roots. For each $j=1,\dots,d$ we set $a_j^\tau:=\varepsilon_j\circ a^\tau$ (recall the notation introduced in Subsection \ref{subsub weyl chambers lieg,liea}).

An element $g\in\g$ is said to have a \textit{gap of index $\Delta$} if for every $\alpha\in\Delta$ one has $\alpha(a^\tau(g))>0$. In this case, the \textit{Cartan attractor} of $g$ is the full flag

\bc
$U^\tau(g):=k\cdot\xi_{\liea^+}$,
\ec

\noindent where $\xi_{\liea^+}\in\mathsf{F}(V)$ is the flag determined by $\liea^+$ and $g=k\exp(a^\tau(g))l$ is a Cartan decomposition of $g$.  Since $g$ has a gap of index $\Delta$, the Cartan attractor does not depend on the particular choice of the Cartan decomposition of $g$ (c.f. \cite[Chapter VII]{Kna}). Note that $g^{-1}$ also has a gap of index $\Delta$ and we denote

\bc
$S^\tau(g):=U^\tau(g^{-1})$.
\ec

\noindent This flag is called the \textit{Cartan repellor} of $g$.

The choice of $\tau$ induces a continuous distance in each exterior power of $V$ and in $\mathsf{F}(V)$. By abuse of notations, $d(\cdot,\cdot)$ will denote any of these distances. Let $j=1,\dots,d-1$ and $\varepsilon$ be a positive number. For a line $\ell\in\pp(\Lambda^jV)$ and a hyperplane $\vartheta\in\pp(\Lambda^jV^*)$ we let

\bc
$b_\varepsilon(\ell):=\lbrace \ell'\in\pp(\Lambda^jV):\hspace{0,3cm} d(\ell,\ell')\leq\varepsilon\rbrace$
\ec

\noindent and

\bc
$B_\varepsilon(\vartheta):=\lbrace \ell'\in\pp(\Lambda^jV):\hspace{0,3cm} d(\ell',\vartheta)\geq\varepsilon\rbrace$,
\ec

\noindent where $d(\ell',\vartheta)$ denotes the minimal distance between $\ell'$ and lines contained in $\vartheta$.

The following remark is classical (see e.g. Horn-Johnson \cite[Section 7.3 of Chapter 7]{HJ}).

\begin{rem} \label{rem complemento de Sdmenosuno va en Uuno}
Fix a positive $\varepsilon$. There exists $L>0$ such that for every $j=1,\dots,d-1$ and every $g$ in $\g$ with

\bc
$\displaystyle\min_{\alpha\in\Delta}\alpha(a^\tau(g))>L$
\ec

\noindent one has

\bc
$\Lambda^jg\cdot B_\varepsilon(\Lambda^j_*S^\tau(g))\subset b_\varepsilon(\Lambda^jU^\tau(g))$.
\ec
\end{rem}

Recall that the \textit{Jordan projection} $\lambda:\g\too\liea^+$ can be defined by the formula
\bc
$\lambda(g):=\displaystyle\lim_{n\too\infty}\dfrac{a^\tau(g^n)}{n}$
\ec
\noindent for every $g$ in $\g$. For each $j=1,\dots,d$ we set $\lambda_j:=\varepsilon_j\circ \lambda$.

An element $g$ in $\g$ is said to be \textit{proximal} (on $\pp(V)$) if $\lambda_1(g)>\lambda_2(g)$ and is said to be \textit{loxodromic} if $\Lambda^jg$ is proximal for every $j=1,\dots,d-1$. If $g$ is loxodromic, we let $g_+\in\mathsf{F}(V)$ (resp. $g_-\in\mathsf{F}(V)$) denote the unique attracting (resp. repelling) fixed point of $g$ acting on $\mathsf{F}(V)$. For every $j=1,\dots,d-1$, the attracting (resp. repelling) fixed line (resp. hyperplane) of $\Lambda^jg$ acting on $\pp(\Lambda^jV$) is $\Lambda^j(g_+)$ (resp. $\Lambda^j_*(g_-)$).

In order to study products of loxodromic elements we will need the following quantified version of proximality. Given real numbers $0<\varepsilon\leq r$, we say that a loxodromic element $g$ is $(r,\varepsilon)$-\textit{loxodromic} if for every $j=1,\dots,d-1$ one has 

\bc
$d(\Lambda^j(g_+),\Lambda^j_*(g_-))\geq 2r$
\ec

\noindent and

\bc
$g\cdot B_\varepsilon(\Lambda^j_*(g_-))\subset b_\varepsilon(\Lambda^j(g_+))$.
\ec

Since $\sigma^o$ preserves $\ko^\tau$ and acts as $-\tn{id}$ on $\lieb=\liea$, Lemma \ref{lema action of sigma on flags} implies the following.

\begin{cor}\label{cor sigmaginver has gaps/hyperbolic and attractrep}

For every $g\in\g$ one has

\bc
$a^\tau(\sigma^o(g^{-1}))=a^\tau(g)$ and $\lambda(\sigma^o(g^{-1}))=\lambda(g)$.
\ec

\noindent Moreover, if $g$ has a gap of index $\Delta$ (resp. if $g$ is loxodromic) then

\bc
$U^\tau(\sigma^o(g^{-1}))=S^\tau(g)^{\perp_o}$ (resp. $\sigma^{o}(g^{-1})_+=(g_-)^{\perp_o}$).
\ec
\end{cor}

\subsection{$(p,q)$-Cartan attractors}\label{subsec pq attractors}

We now focus on elements $g\in\g$ such that $\sigma^o(g^{-1})g$ is loxodromic.

\begin{rem}\label{rem bog non empty interior}
It is not hard to see that if $\sigma^o(g^{-1})g$ is loxodromic then its eigenvalues must be real numbers. In particular, for all such $g$ we have $g\in\bog$. Moreover, since being loxodromic is an open condition in $\g$ we conclude that $\bog$ has non empty interior.

\end{rem}

By Proposition \ref{prop equivalences HWbH} and Remark \ref{rem sigmawinvw belongs to M}, the element $\sigma^o(g^{-1})g$ is loxodromic if and only if there exists a $\lieb^+$-compatible Weyl chamber $\liea^+$ for which one has

\bc
$g\in\h^o\wbh\exp(\tn{int}(\liea^+))\h^o$.
\ec

\noindent In this case we set

\bc
$U^o(g):=(g\sigma^o(g^{-1}))_+\in\mathsf{F}(V)$.
\ec

\noindent By similar reasons the element $\sigma^o(g)g^{-1}$ is loxodromic as well and we denote

\bc
$S^o(g):=U^o(g^{-1})$.
\ec

\noindent Note that

\bc
$S^o(g)=(\sigma^o(g^{-1})g)_-$.
\ec

\begin{rem}\label{rem Uog and Sog in HWbH coordinates}

Suppose that $g=h\hat{w}\exp(X)\tilde{h}$ is a $(p,q)$-Cartan decomposition of $g$ such that $X\in\tn{int}(\liea^+)$ for some Weyl chamber $\liea^+\subset\lieb^+$. Then 

\bc
$U^o(g)=h\hat{w}\cdot\xi_{\liea^+}$ and $S^o(g)=\tilde{h}^{-1}\cdot\xi_{-\liea^+}=\tilde{h}^{-1}\cdot(\xi_{\liea^+}^{\perp_o})$.
\ec

\noindent In particular, both $U^o(g)$ and $S^o(g)$ belong to the set $\fvo$ of $o$-generic flags of $V$.
\end{rem}

\begin{lema}\label{lema Uo cerca de Utau y lo mismo con S}
Let $C\subset\fvo$ be a compact set. Then there exist $0<\varepsilon_0\leq r_0$ such that for every $0<\varepsilon\leq r$ with $\varepsilon\leq\varepsilon_0$ and $r\leq r_0$ there exists a positive $L$ with the following property: fix a $\lieb^+$-compatible Weyl chamber $\liea^+$. For every $g\in\g$ for which

\bc
$\displaystyle\min_{\alpha\in\Sigma^+(\lieg,\liea)}\alpha(a^\tau(g))>L$
\ec

\noindent holds, and such that $U^\tau(g)\in C$ and $S^\tau(g)\in C$, then $\sigma^o(g^{-1})g$ is $(2r,2\varepsilon)$-loxodromic. Furthermore one has

\bc
$d(U^o(g),U^\tau(g))\leq\varepsilon$ and $d(S^o(g),S^\tau(g))\leq\varepsilon$.
\ec

\end{lema}

\begin{proof}

We apply a ``ping-pong'' argument (see Figure \ref{fig pingpong} below). Namely, because of Lemma \ref{lema equivalences for o generic flags} there exists a positive $r_0$ for which for every $j=1,\dots,d-1$ and every $\xi\in C$ one has

\bc
$d(\Lambda^j\xi,(\Lambda^j\xi)^{\perp_{o_j}})\geq 6r_0$ and $d(\Lambda^j(\xi^{\perp_o}),(\Lambda^j(\xi^{\perp_o}))^{\perp_{o_j}})\geq 6r_0$.
\ec

\noindent By Remark \ref{rem ordered basis induced by o generic and double perp and xi o generic iff gxi go generic} this implies
\begin{equation}\label{eq in sigmaging loxod uno}
d(\Lambda^j\xi,\Lambda^j_*(\xi^{\perp_o}))\geq 6r_0 \tn{ and } d(\Lambda^j(\xi^{\perp_o}),\Lambda^j_*\xi)\geq 6r_0
\end{equation}
\noindent for all $\xi\in C$. We now find $\varepsilon_0\leq r_0$ for which for every $j=1,\dots,d-1$ and every $\xi \in C$ one has
\begin{equation}\label{eq in sigmaging loxod dos}
b_{\varepsilon_0}(\Lambda^j\xi)\subset B_{\varepsilon_0}(\Lambda^j_*(\xi^{\perp_o})) \tn{ and } b_{\varepsilon_0}(\Lambda^j(\xi^{\perp_o}))\subset B_{\varepsilon_0}(\Lambda^j_*\xi).
\end{equation}

Fix $0<\varepsilon\leq r$ with $\varepsilon\leq\varepsilon_0$ and $r\leq r_0$. By Remark \ref{rem complemento de Sdmenosuno va en Uuno} we can find a positive $L$ for which for every $g$ such that

\bc
$\displaystyle\min_{\alpha\in\Sigma^+(\lieg,\liea)}\alpha(a^\tau(g))>L$
\ec

\noindent one has
\begin{equation}\label{eq in sigmaging loxod tres}
\Lambda^jg\cdot B_\varepsilon\left(\Lambda^j_*S^\tau(g)\right)\subset b_\varepsilon(\Lambda^jU^\tau(g)).
\end{equation}
\noindent for every $j=1,\dots,d-1$. Further, by Corollary \ref{cor sigmaginver has gaps/hyperbolic and attractrep} we have as well
\begin{equation}\label{eq in sigmaging loxod cuatro}
\Lambda^j\sigma^o(g^{-1})\cdot B_\varepsilon\left(\Lambda^j_*S^\tau(\sigma^o(g^{-1}))\right)\subset b_\varepsilon(\Lambda^jU^\tau(\sigma^o(g^{-1}))).
\end{equation}

\noindent Now suppose moreover that $U^\tau(g)\in C$ and $S^\tau(g)\in C$. By equation (\ref{eq in sigmaging loxod uno}) and Corollary \ref{cor sigmaginver has gaps/hyperbolic and attractrep} we have
\begin{equation}\label{eq in sigmaging loxod cinco}
d\left(\Lambda^j U^\tau(\sigma^o(g^{-1})),\Lambda^j_*S^\tau( g)\right)\geq 6r.
\end{equation}

\noindent Now by equation (\ref{eq in sigmaging loxod tres}) we have

\bc
$\Lambda^j(\sigma^o(g^{-1})g)\cdot B_\varepsilon(\Lambda^j_*S^\tau(g))\subset\Lambda^j\sigma^o(g^{-1})\cdot b_\varepsilon(\Lambda^jU^\tau(g))  $
\ec

\noindent and by Corollary \ref{cor sigmaginver has gaps/hyperbolic and attractrep} and equation (\ref{eq in sigmaging loxod dos}) this is contained in

\bc
$\Lambda^j\sigma^o(g^{-1})\cdot B_\varepsilon(\Lambda^j_*S^\tau(\sigma^o(g^{-1})))$.
\ec

\noindent By equation (\ref{eq in sigmaging loxod cuatro}) we have therefore

\bc
$\Lambda^j(\sigma^o(g^{-1})g)\cdot B_\varepsilon(\Lambda^j_*S^\tau(g))\subset  b_\varepsilon(\Lambda^jU^\tau(\sigma^o(g^{-1})))$.
\ec

\noindent By Benoist \cite[Lemme 1.2]{Ben1}, this inclusion together with equation (\ref{eq in sigmaging loxod cinco}) gives that $\sigma^o(g^{-1})g$ is $(2r,2\varepsilon)$-loxodromic and 

\bc
$d((\sigma^o(g^{-1})g)_-,S^\tau(g))\leq \varepsilon$.
\ec

Working with $g\sigma^o(g^{-1})$ (instead of $\sigma^o(g^{-1})g$) we can also assume

\bc
$d(U^o(g),U^\tau(g))\leq\varepsilon$.
\ec

\end{proof}

\begin{figure}[h!]
\bc
\scalebox{0.8}{%
\begin{overpic}[scale=1, width=1\textwidth, tics=5]{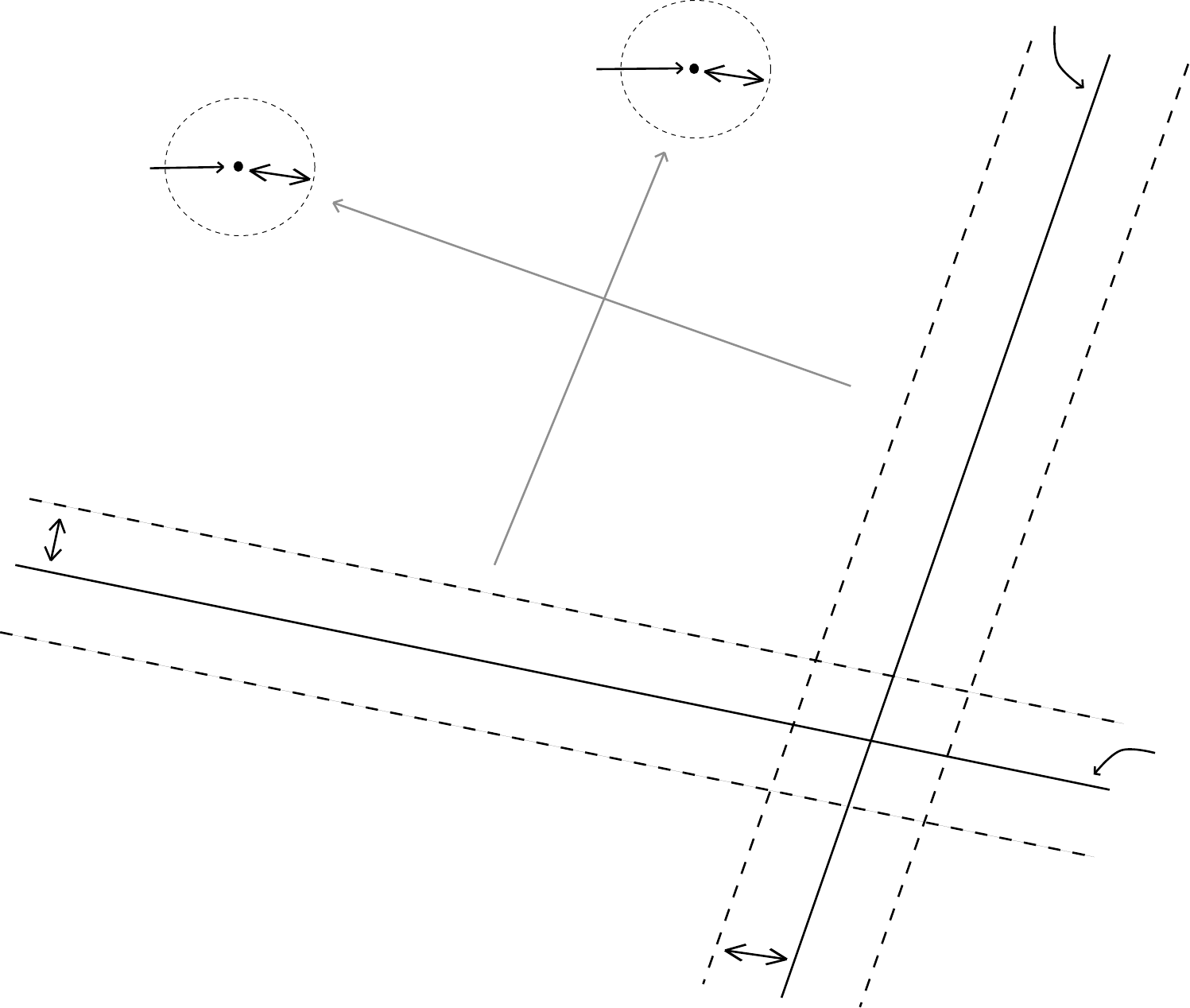}

\put (56,50) {\tc{gray}{\Large$\Lambda^j\sigma^o(g^{-1})$}}

\put (63,6) {\Large$\varepsilon$}

\put (2,39) {\Large$\varepsilon$}

\put (23,71) {\Large$\varepsilon$}

\put (-9,70) {\Large$\Lambda^jU^\tau(\sigma^o(g^{-1}))$}

\put (38,78) {\Large$\Lambda^jU^\tau(g)$}

\put (84,83) {\Large$\Lambda^j_*S^\tau(\sigma^o(g^{-1}))$}

\put (98,21) {\Large$\Lambda^j_*S^\tau(g)$}

\put (61,76) {\Large$\varepsilon$}

\put (37,40) {\Large\tc{gray}{$\Lambda^jg$}}
  \end{overpic}}
 
\hspace{0,3cm}

\begin{changemargin}{3cm}{3cm}    
    \caption{The proof of Lemma \ref{lema Uo cerca de Utau y lo mismo con S}.}
  \label{fig pingpong}
\end{changemargin}

  \ec
\end{figure}

\subsection{Computation of the Weyl chamber and the $\wbh$-coordinate}\label{subsec computation weyl chamber and hatw}

Recall that $\mathcal{C}$ is the basis of lines determined by $\lieb$. Given two flags $\xi$ and $\xi'$ spanned by $\mathcal{C}$, we denote by $w_{\xi\xi'}$ the unique element of $\wb$ for which

\bc
$w_{\xi\xi'}\cdot\xi'=\xi$.
\ec

\noindent We also denote by $\iota_{\lieb^+}(\xi)$ the flag determined by the $\lieb^+$-compatible Weyl chamber $\iota_{\lieb^+}(\liea^+_\xi)$.

Recall that $\fvo_\xi$ denotes the (open) $\h^o$-orbit of $\xi$.

\begin{prop}\label{prop w and standard weyl chamber for element with gap in sing val}
Let $\xi_s$ and $\xi_u$ be two $\lieb^+$-compatible flags. Fix two compact sets $C_s\subset\fvo_{\xi_s}$ and $C_u\subset\fvo_{\xi_u}$. Then there exists a positive $L$ with the following property: let $\liea^+$ be a $\lieb^+$-compatible Weyl chamber. For every $g\in\g$ for which

\bc
$\displaystyle\min_{\alpha\in\Sigma^+(\lieg,\liea)}\alpha(a^\tau(g))>L$
\ec

\noindent holds, and such that $S^\tau(g)\in C_s$ and $U^\tau(g)\in C_u$, one has

\bc
$g\in\h^ow_{\xi_u\iota_{\lieb^+}(\xi_s)}\exp(\tn{int}(\iota_{\lieb^+}(\liea^+_{\xi_s})))\h^o$.
\ec

\end{prop}

\begin{proof}
The proof is illustrated in Figure \ref{fig compweylchamber} below. 

By Lemma \ref{lema Uo cerca de Utau y lo mismo con S} we can take a positive $L$ such that for every $g$ as in the statement one has

\bc
$U^o(g)\in\fvo_{\xi_u}$ and $S^o(g)\in\fvo_{\xi_s}$.
\ec

\noindent If $h\hat{w}\exp(b^o(g))\tilde{h}$ is a $(p,q)$-Cartan decomposition of $g$ we then have

\bc
$U^o(g)=h\hat{w}\cdot b^o(g)_+\in\h^o\cdot\xi_u$ and $S^o(g)=\tilde{h}^{-1}\cdot b^o(g)_-\in\h^o\cdot\xi_s$.
\ec

\noindent In particular, 

\bc
$\hat{w}\cdot b^o(g)_+\in\h^o\cdot\xi_u$ and $ b^o(g)_-\in\h^o\cdot\xi_s$.
\ec

\noindent We conclude that

\bc
$b^o(g)_+=b^o(g)_-^{\perp_o}\in\h^o\cdot(\xi_s)^{\perp_o}=\h^o\cdot\iota_{\lieb^+}(\xi_s)$,
\ec

\noindent because $w_{\lieb^+}$ belongs to $\h^o$ (c.f. Subsection \ref{subsec oppostion liebplus}). Since both $b^o(g)_+$ and $\iota_{\lieb^+}(\xi_s)$ are $\lieb^+$-compatible, Proposition \ref{prop param of open orbits by lieb compatible flags} implies $b^o(g)_+=\iota_{\lieb^+}(\xi_s)$ and therefore $b^o(g)$ belongs to $\tn{int}( \iota_{\lieb^+}(\liea^+_{\xi_s}))$. Further, we have

\bc
$\hat{w}\cdot \iota_{\lieb^+}(\xi_s)\in\h^o\cdot\xi_u$
\ec

\noindent and Lemma \ref{lema if preserve mathcalC equivalences belong to H and to K} implies the existence of an element $\hat{w}'\in\wbh\cap\h^o$ such that

\bc
$\hat{w}'\hat{w}\cdot \iota_{\lieb^+}(\xi_s)=\xi_u$.
\ec

\noindent Then $w'w=w_{\xi_u\iota_{\lieb^+}(\xi_s)}$ and the proof is complete.

\end{proof}

\begin{figure}[h!]
\bc
\scalebox{0.8}{%
\begin{overpic}[scale=1, width=1\textwidth, tics=5]{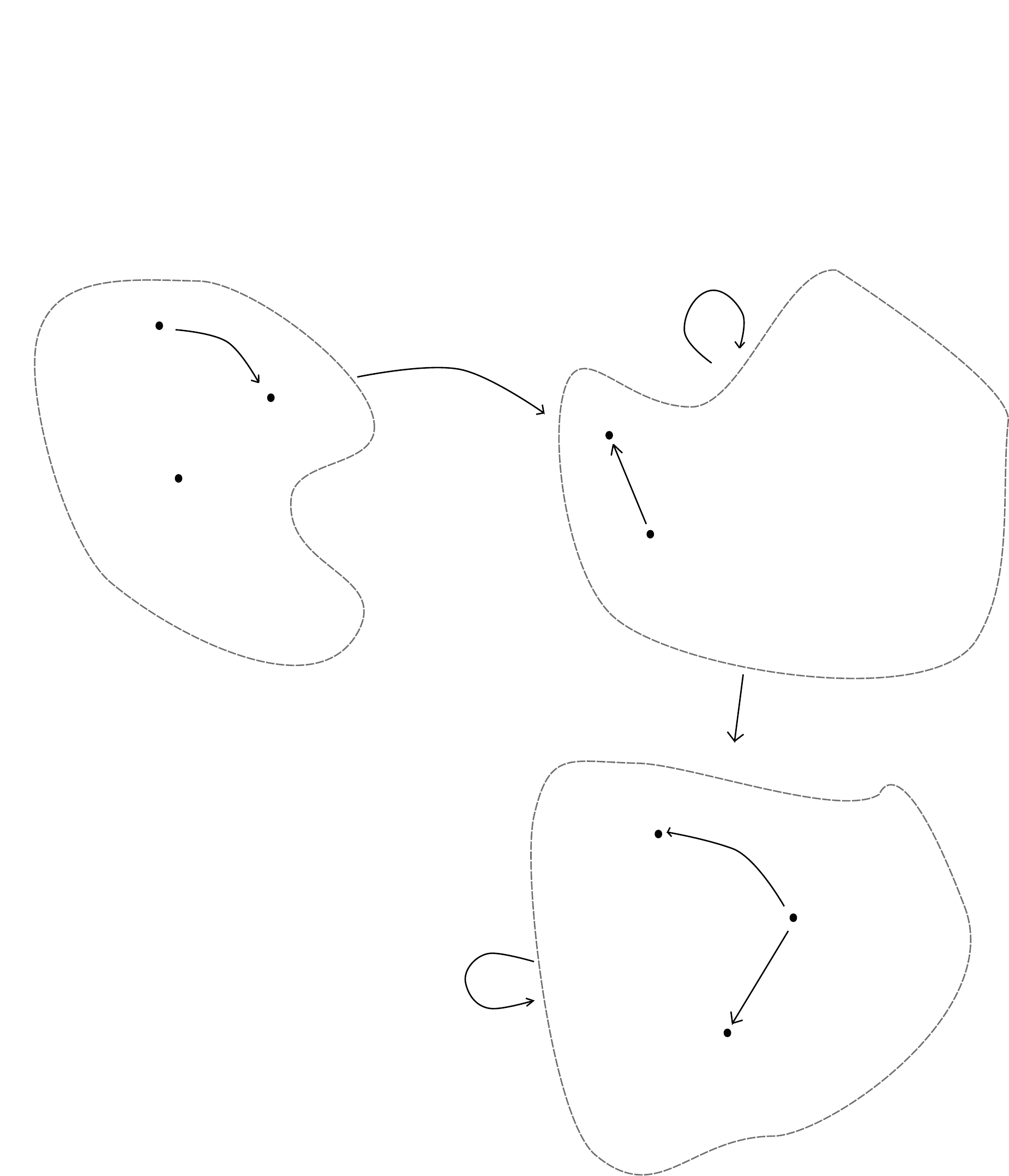}

\put (52,48) {\tc{gray}{\Large$\h^o\cdot(\xi_s^{\perp_o})=\h^o\cdot\iota_{\lieb^+}(\xi_s)$}}

\put (20,46) {\tc{gray}{\Large$\h^o\cdot\xi_s$}}

\put (51,2) {\tc{gray}{\Large$\h^o\cdot\xi_u$}}

\put (36,16) {\Large$\hat{w}'$}

\put (65,15) {\Large$\hat{w}'$}

\put (60,9) {\Large$\xi_u$}

\put (49,30) {\Large$U^o(g)$}

\put (63,28) {\Large$h$}

\put (68,21) {\Large$\hat{w}\cdot b^o(g)_+$}

\put (64,39) {\Large$\hat{w}$}

\put (56,54) {\Large$\xi_s^{\perp_o}$}

\put (53,62) {\Large$b^o(g)_-^{\perp_o}=b^o(g)_+=\iota_{\lieb^+}(\xi_s)$}

\put (49,58) {\Large$\hat{w}_{\lieb^+}$}

\put (59,77) {\Large$\hat{w}_{\lieb^+}$}

\put (37,69) {\Large$\cdot^{\perp_o}$}

\put (22,63) {\Large$b^o(g)_-$}

\put (12,59) {\Large$\xi_s$}

\put (18,72) {\Large$\tilde{h}$}

\put (6,72) {\Large$S^o(g)$}
  \end{overpic}}
 
\hspace{0,3cm}

\begin{changemargin}{3cm}{3cm}    
    \caption{The proof of Proposition \ref{prop w and standard weyl chamber for element with gap in sing val}.}
  \label{fig compweylchamber}
\end{changemargin}

  \ec
\end{figure}
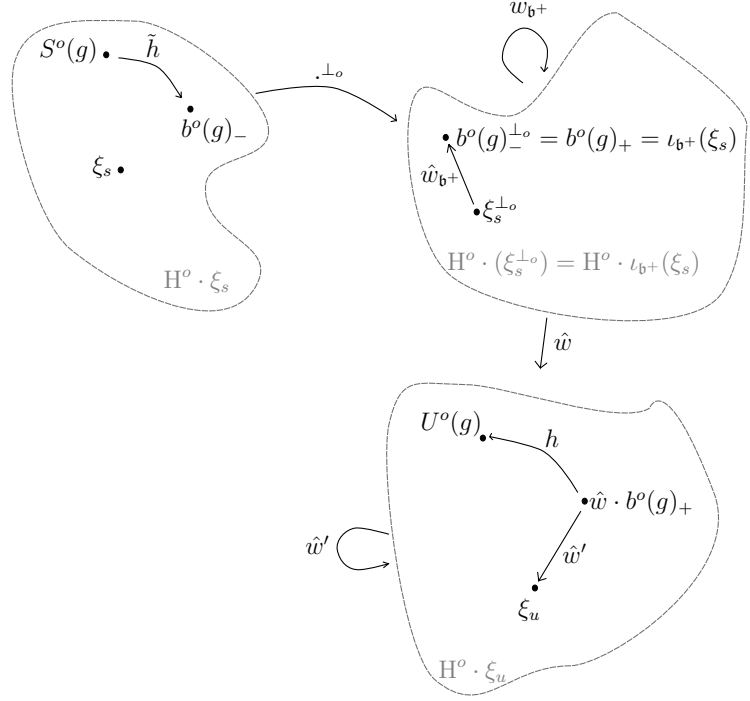

\begin{cor}\label{cor cartan decomposition for Sg and Ug in single open orbit}

Let $\xi$ be a $\lieb^+$-compatible flag and $C\subset\fvo_\xi$ be a compact set. Then there exists a positive $L$ with the following property: let $\liea^+$ be a $\lieb^+$-compatible Weyl chamber. For every $g\in\g$ for which

\bc
$\displaystyle\min_{\alpha\in\Sigma^+(\lieg,\liea)}\alpha(a^\tau(g))>L$
\ec

\noindent holds, and such that $S^\tau(g)\in C$ and $U^\tau(g)\in C$, one has

\bc
$g\in\h^ow_{\xi\iota_{\lieb^+}(\xi)}\exp(\tn{int}(\iota_{\lieb^+}(\liea^+_{\xi})))\h^o$.
\ec

\end{cor}

\subsection{Linear algebraic interpretation and first estimates}\label{subsec linear alg interp of bo and first estimates}

Fix a $\lieb^+$-compatible Weyl chamber $\liea^+$. Remark \ref{rem sigmawinvw belongs to M} and equation (\ref{eq sigmaginvg in HWbH coordinates}) imply the following: for every element $g$ of $\bog$ there exists an element $w_g\in\wb$ such that 

\bc
$b^o(g)=\frac{1}{2}w_g\cdot\lambda(\sigma^o(g^{-1})g)$.
\ec 

\noindent In particular, one has:

\bc
$\Vert b^o(g)\Vert_\lieb=\frac{1}{2}\Vert\lambda(\sigma^o(g^{-1})g)\Vert_\lieb$.
\ec

\noindent Whenever $g$ has a (sufficiently large) gap of index $\Delta$ and $o$-generic Cartan attractor and repellor we have a more precise result.

\begin{prop}\label{prop jordan versus generalized cartan}

Let $\xi_s$ be a $\lieb^+$-compatible flag and fix compact sets $C_s\subset\fvo_{\xi_s}$ and $C\subset\fvo$. Then there exists a positive $L$ with the following property: for every $g\in\g$ for which

\bc
$\displaystyle\min_{\alpha\in\Sigma^+(\lieg,\liea)}\alpha(a^\tau(g))>L$
\ec

\noindent holds, and such that $S^\tau(g)\in C_s$ and $U^\tau(g)\in C$, one has

\bc
$b^o(g)=\frac{1}{2} w_{\iota_{\lieb^+}(\xi_s)\xi_{\liea^+}}\cdot\lambda(\sigma^o(g^{-1})g)$.
\ec

\noindent In particular, if $\iota_{\lieb^+}({\xi_s})=\xi_{\liea^+}$ we have

\bc
$b^o(g)=\frac{1}{2}\lambda(\sigma^o(g^{-1})g)$.
\ec

\end{prop}

\begin{proof}
Apply Proposition \ref{prop w and standard weyl chamber for element with gap in sing val} to each intersection of $C$ with each open orbit of the action $\h^o\curvearrowright\mathsf{F}(V)$. For every $g$ as in the statement we know that there exists some $X\in\tn{int}(\liea^+)$ such that

\bc
$b^o(g)=w_{\iota_{\lieb^+}(\xi_s)\xi_{\liea^+}}\cdot X$.
\ec

\noindent Hence

\bc
$\frac{1}{2}\lambda(\sigma^o(g^{-1})g)  = \lambda(\exp(b^o(g)))=X=(w_{\iota_{\lieb^+}(\xi_s)\xi_{\liea^+}})^{-1}\cdot b^o(g)$. 
\ec

\end{proof}

\begin{prop}\label{prop bo close to wdotcartan}

Let $\xi_s$ be a $\lieb^+$-compatible flag and fix compact sets $C_s\subset\fvo_{\xi_s}$ and $C\subset\fvo$. Then there exist positive numbers $L$ and $D$ with the following property: for every $g\in\g$ for which

\bc
$\displaystyle\min_{\alpha\in\Sigma^+(\lieg,\liea)}\alpha(a^\tau(g))>L$
\ec

\noindent holds, and such that $S^\tau(g)\in C_s$ and $U^\tau(g)\in C$, one has

\bc
$\Vert b^o(g)-  w_{\iota_{\lieb^+}(\xi_s)\xi_{\liea^+}}\cdot a^\tau(g)\Vert_\lieb\leq D$.
\ec

\end{prop}

\begin{proof}

As we saw in the proof of Lemma \ref{lema Uo cerca de Utau y lo mismo con S}, there exists a positive constant $r_0$ such that for every $j=1,\dots,d-1$ and every $g$ as in the statement one has

\bc
$d\left(\Lambda^jU^\tau(g),\Lambda^j_*S^\tau(\sigma^o(g^{-1}))\right)\geq r_0$.
\ec

\noindent It is not hard to show (see e.g. \cite[Lemma A.7]{BPS}) that this implies the existence of a constant $D$ such that, for all $j=1,\dots,d-1$,

\bc
$\vert a^\tau_j(\sigma^o(g^{-1}) g)-a^\tau_j(\sigma^o(g^{-1}))-a^\tau_j(g)\vert\leq D$.
\ec

\noindent By Corollary \ref{cor sigmaginver has gaps/hyperbolic and attractrep} we have

\bc
$\left\vert \frac{1}{2}a^\tau_j(\sigma^o(g^{-1})g)-a^\tau_j( g)\right\vert\leq D/2$
\ec

\noindent and we conclude that, up to changing $D$ by a larger constant if necessary, one has

\bc
$\left\Vert \frac{1}{2}a^\tau(\sigma^o(g^{-1})g)- a^\tau(g)\right\Vert_\lieb\leq D$.
\ec

Fix $r$, $\varepsilon$ and $L$ as in Lemma \ref{lema Uo cerca de Utau y lo mismo con S} and Proposition \ref{prop jordan versus generalized cartan} (for each intersection of $C$ with $\fvo$). For every $g$ as in the statement we know that $\sigma^o(g^{-1})g$ is $(2r,2\varepsilon)$-loxodromic and therefore we can enlarge $D$ if necessary in order to have

\bc
$\left\Vert \frac{1}{2}a^\tau(\sigma^o(g^{-1})g)- \frac{1}{2}\lambda(\sigma^o(g^{-1})g)\right\Vert_\lieb\leq D$
\ec

\noindent (c.f. Benoist \cite[Lemme 1.3]{Ben1}). To finish apply Proposition \ref{prop jordan versus generalized cartan} and the fact that the norm $\Vert\cdot\Vert_\lieb$ is $\wb$-invariant.

\end{proof}

\section{Busemann cocycles}\label{sec busemann cocycles and gromov}

Keep the notations from the previous subsection. In particular, we emphasize here that we have fixed a $\lieb^+$-compatible Weyl chamber $\liea^+$. The goal of this section is to introduce an analogue of the \textit{Busemann cocycle} of $\g$ adapted to our setting, and a corresponding \textit{Gromov product}. As we shall see in Section \ref{sec counting}, these will be key objects in the study of precise asymptotic properties of the $(p,q)$-Cartan projection.

\subsection{$\tau$-Busemann cocycle}
For future reference we begin by recalling the definition of the Busemann cocycle of $\g$. 

Let $\n$ (resp. $\pmin$) be the unipotent radical (resp. minimal parabolic) associated to $\liea^+$. Quint \cite{Qui2} introduces the\footnote{Sometimes also called the \textit{Iwasawa cocycle} of $\g$.} \textit{Busemann cocycle} of $\g$

\bc
$\beta^\tau:\g\times\mathsf{F}(V)\too\lieb$
\ec

\noindent which is defined by means of the Iwasawa decomposition of $\g$. Indeed, for $g\in\g$ and $\xi\in\mathsf{F}(V)$ the Busemann cocycle is characterized by the equality

\bc
$gk=l\exp(\beta^\tau(g,\xi))n$
\ec

\noindent where $k,l\in\ko^\tau$, $n\in\n$ and $k\cdot\xi_{\liea^+}=\xi$. Note that for every $g_1$ and $g_2$ in $\g$, and every $\xi\in\mathsf{F}(V)$ one has

\bc
$\beta^\tau(g_1g_2,\xi)=\beta^\tau(g_1,g_2\cdot\xi)+\beta^\tau(g_2,\xi)$.
\ec

\noindent In this paper we call this cocycle the $\tau$-\textit{Busemann cocycle} of $\g$.

\subsection{$o$-Busemann cocycle}\label{subsec o busemann}

For $j=1,\dots,d-1$ we denote by $\chi_j\in\lieb^*$ the highest weight of the exterior power representation $\Lambda^j$. Let 

\bc
$\gfo:=\lbrace (g,\xi)\in\g\times\fvo:\hspace{0,3cm} g\cdot\xi\in\fvo\rbrace$.
\ec

\noindent Define the $o$-\textit{Busemann cocycle} of $\g$

\bc
$\beta^o:\gfo\too\lieb$
\ec

\noindent by the equations for $j=1,\dots,d-1$:

\bc
$\chi_j(\beta^o(g,\xi)):=\dfrac{1}{2}\log\left\vert\dfrac{ \langle \Lambda^j g\cdot v,\Lambda^j g\cdot v\rangle_{o_j} }{ \langle v, v\rangle_{o_j} }\right\vert$,
\ec

\noindent where $\langle\cdot,\cdot\rangle_{o_j}$ is any form representing $o_j$ and $v$ is any non zero vector in the line $\Lambda^j\xi$. Thanks to Lemma \ref{lema equivalences for o generic flags} the map $\beta^o$ is well defined.

The proof of the following is straightforward.

\begin{lema}\label{lema o busemann cocycle is a cocycle}
Let $g_1$ and $g_2$ be two elements of $\g$, $\xi$ be a flag in $V$ and suppose that $(g_1g_2,\xi)$ and $(g_2,\xi)$ belong to $\gfo$. Then

\bc
$\beta^o(g_1g_2,\xi)=\beta^o(g_1,g_2\cdot\xi)+\beta^o(g_2,\xi)$.
\ec
\end{lema}

\begin{rem}\label{rem o Busemann cohomologous to tau Busemann}

The $o$-Busemann cocycle generalizes the $\tau$-Busemann cocycle of $\g$, in the sense that whenever $pq=0$ and $o=\tau$ one has 

\bc
$\gfo=\g\times\mathsf{F}(V)$
\ec

\noindent and $\beta^o$ coincides with $\beta^\tau$ (c.f. Quint \cite[Lemma 6.4]{Qui2}). Since we are assuming $pq\neq 0$, the set $\gfo$ does not coincide with $\g\times\mathsf{F}(V)$, but we still have a relation between $\beta^o$ and $\beta^\tau$: there exists a smooth function 

\bc
$V_{o\tau}:\fvo\too\lieb$
\ec

\noindent for which one has

\bc
$V_{o\tau}(g\cdot\xi)-V_{o\tau}(\xi)=\beta^o(g,\xi)-\beta^\tau(g,\xi)$
\ec
\noindent for every pair $(g,\xi)\in\gfo$. Indeed, it suffices to take $V_{o\tau}$ defined by the formulas

\bc
$\chi_j(V_{o\tau}(\xi)):=\dfrac{1}{2}\log\left\vert\dfrac{ \langle  v, v\rangle_{o_j} }{ \langle v, v\rangle_{\tau_j} }\right\vert$
\ec

\noindent for $j=1,\dots,d-1$ and $0\neq v\in\Lambda^j\xi$, and for the inner product $\tau_j$ on $\Lambda^j V$ induced by $\tau$.
\end{rem}

\subsubsection{\tn{\textbf{$(p,q)$-Iwasawa decomposition}}}

The Lie theoretic description of the $o$-Busemann cocycle is as follows. A $(p,q)$-\textit{Iwasawa decomposition} of an element $g$ in $\g$ is a decomposition of the form

\bc
$g=h\hat{w}\exp(X)n$
\ec

\noindent where $h\in\h^o$, $\hat{w}\in\wbh$, $X\in\lieb$ and $n\in\n$. Note that if this decomposition holds, then the element $X$ is uniquely determined: for every $j=1,\dots,d-1$ one has

\bc
$\chi_j(X)=\dfrac{1}{2}\log\left\vert\dfrac{ \langle \Lambda^j g\cdot v,\Lambda^j g\cdot v\rangle_{o_j} }{ \langle v, v\rangle_{o_j} }\right\vert$,
\ec

\noindent where $\langle\cdot,\cdot\rangle_{o_j}$ is any representative of $o_j$ and $v$ is any non zero vector in the line $\Lambda^j\xi_{\liea^+}$. Indeed, this follows from the fact that $\Lambda^j\h^o$ preserves the form $o_j$ and the fact that one has the equality
\begin{equation}\label{eq exterior power of w actinc in mathcal C}
\vert\langle\Lambda^j\hat{w}\cdot v,\Lambda^j\hat{w}\cdot v\rangle_{o_j}\vert=\vert\langle  v,  v\rangle_{o_j}\vert
\end{equation}
\noindent for every $j=1,\dots,d-1$ (c.f. Lemma \ref{lema if preserve mathcalC equivalences belong to H and to K}).

On the other hand, if an element $g\in\g$ admits a $(p,q)$-Iwasawa decomposition then one has

\bc
$(g,\xi_{\liea^+})\in\gfo$.
\ec

\noindent Conversely, we have the following.

\begin{lema}\label{lema iwasawa and unqiqueness lieb}
Suppose that the pair $(g,\xi_{\liea^+})$ belongs to $\gfo$. Consider an element $\hat{w}\in\wbh$ for which the Weyl chamber $\hat{w}\cdot\liea^+$ is $\lieb^+$-compatible and such that $g\cdot\xi_{\liea^+}\in\fvo_{\hat{w}\cdot\liea^+}$.  Then $g$ admits a $(p,q)$-Iwasawa decomposition of the form

\bc
$g=h\hat{w}\exp(X)n$.
\ec
\end{lema}

\begin{proof}
There exists $\tilde{h}\in \h^o$ such that $g\cdot\xi_{\liea^+}=\tilde{h}\hat{w}\cdot\xi_{\liea^+}$ and therefore $\hat{w}^{-1}\tilde{h}^{-1}g$ belongs to $ \pmin$. Since $\pmin=\mb\exp(\lieb)\n$ we conclude that

\bc
$\hat{w}^{-1}\tilde{h}^{-1}g=m\exp(X)n$
\ec

\noindent for some $m\in\mb$, $X\in\lieb$ and $n\in\n$. Now $\hat{w}m=m'\hat{w}$ for some $m'\in\mb\subset\h^o$ and the lemma follows.

\end{proof}

\begin{cor}\label{cor betao is the generalized iwasawa}
Let $(g,\xi)$ be an element in $\gfo$ and take $h'\in\h^o$ and $\hat{w}'\in\wbh$ such that $h'\hat{w}'\cdot\xi_{\liea^+}=\xi$. Then $gh'\hat{w}'$ admits a $(p,q)$-Iwasawa decomposition of the form

\bc
$gh'\hat{w}'=h\hat{w}\exp(\beta^o(g,\xi))n$.
\ec
\end{cor}

\begin{proof}
Since $gh'\hat{w}'\cdot\xi_{\liea^+}$ is $o$-generic, there exists $\hat{w}\in\wbh$ such that $\hat{w}\cdot\xi_{\liea^+}$ is $\lieb^+$-compatible and $gh'\hat{w}'\cdot\xi_{\liea^+}\in\fvo_{\hat{w}\cdot\liea^+}$. By Lemma \ref{lema iwasawa and unqiqueness lieb} we find a $(p,q)$-Iwasawa decomposition of $gh'\hat{w}'$ of the form

\bc
$gh'\hat{w}'=h\hat{w}\exp(X)n$.
\ec

Now the element $X$ in this decomposition is characterized by the equalities

\bc
$\chi_j(X)=\dfrac{1}{2}\log\left\vert\dfrac{ \langle \Lambda^j (gh'\hat{w}')\cdot v,\Lambda^j (gh'\hat{w}')\cdot v\rangle_{o_j} }{ \langle v, v\rangle_{o_j} }\right\vert$,
\ec

\noindent where $v$ is any non zero vector in the line $\Lambda^j\xi_{\liea^+}$. Since $\Lambda^j\h^o$ preserves the form $o_j$, equality (\ref{eq exterior power of w actinc in mathcal C}) finishes the proof.

\end{proof}

\subsubsection{\tn{\textbf{Dual cocycle}}}

Let $\iota_{\liea^+}:\lieb\too\lieb$ be the \textit{opposition involution} associated to the choice of $\liea^+$:

\bc
$\iota_{\liea^+}:X\mapsto -w_{\liea^+}\cdot X$,
\ec

\noindent where $w_{\liea^+}\in\wb$ is the unique element that takes the Weyl chamber $-\liea^+$ to $\liea^+$.

\begin{rem}
Recall that the choice of $\liea^+$ determines a total order $\mathcal{C}=\lbrace\ell_1,\dots,\ell_d\rbrace$ in the basis of lines associated to $\lieb$. Even though we will not use it in the future, we mention that it is possible to show that the equality $\iota_{\lieb^+}=\iota_{\liea^+}$ holds if and only if

\bc
$\sgo(\ell_j)=\sgo(\ell_{d-j+1})$
\ec

\noindent holds for every $j=1,\dots,d$. In particular, such a choice of $\liea^+$ is not always possible.
\end{rem}

In higher rank, cocycles come usually in pairs (c.f. \cite{Sam,Sam2}). The following corollary gives an explicit description of the cocycle ``dual" to $\beta^o$.

\begin{cor}\label{cor iwasawa for sigma g}

Let $(g,\xi)$ be an element in $\gfo$. Then $(\sigma^o(g),\xi^{\perp_o})$ belongs to $\gfo$ and one has

\bc
$\beta^o(\sigma^o(g),\xi^{\perp_o})=\iota_{\liea^+}\circ \beta^o(g,\xi)$.
\ec

\end{cor}

\begin{proof}
Lemma \ref{lema action of sigma on flags} implies that $(\sigma^o(g),\xi^{\perp_o})$ belongs to $\gfo$. Further, let $h'\in\h^o$ and $\hat{w}'\in\wbh$ be two elements such that $h'\hat{w}'\cdot\xi_{\liea^+}=\xi$. By Corollary \ref{cor betao is the generalized iwasawa} we can write

\bc
$gh'\hat{w}'=h\hat{w}\exp(\beta^o(g,\xi))n$
\ec

\noindent and therefore

\bc
$\sigma^o(g)=h\sigma^o(\hat{w})\exp(-\beta^o(g,\xi))\sigma^o(n)\sigma^o(\hat{w}')^{-1}(h')^{-1}$.
\ec

\noindent Now by Lemma \ref{lema action of sigma on flags} we have

\bc
$\sigma^o(\hat{w}')^{-1}(h')^{-1}\cdot(\xi^{\perp_o})=\xi_{\liea^+}^{\perp_o}=w_{\liea^+}\cdot \xi_{\liea^+}$
\ec

\noindent and since $\sigma^o(n)w_{\liea^+}=w_{\liea^+}n'$ for some $n'\in\n$, the result follows.

\end{proof}

\subsubsection{\tn{\textbf{Geometric interpretation}}}\label{subsub geom int o busemann}

We now discuss a geometric interpretation for the $o$-Busemann cocycle. This is not formally needed for the reminder of the paper, and the reader not interested in this discussion may go directly to Subsection \ref{subsection gromov product in Qpq}.

Define a map 

\bc
$\Pi^o:\fvo\too \so$
\ec

\noindent in the following way: given $\xi\in\fvo$, consider the $o$-orthogonal basis of lines of $V$

\bc
$\lbrace\ell_1^o(\xi),\dots,\ell_d^o(\xi)\rbrace$
\ec

\noindent given by Proposition \ref{lema equivalences for o generic flags}. Define $\Pi^o(\xi)$ to be the unique element of $\so$ for which this basis of lines is orthogonal. Geometrically, the projection $\Pi^o(\xi)$ is the intersection between $\so$ and the unique flat of $\xsyg$ which is orthogonal to $\so$ and that contains a Weyl chamber ``asymptotic" to $\xi$. Note that for every $(g,\xi)\in\gfo$ one has
\begin{equation}\label{eq equivariance for Pio}
\Pi^{g^{-1}\cdot o}(\xi)=g^{-1}\cdot\Pi^o(g\cdot\xi).
\end{equation}

On the other hand, the \textit{Busemann function} (in $\xsyg$) is the map

\bc
$\bd:\xsy\times\xsy\times\mathsf{F}(V)\too\lieb$
\ec

\noindent given by

\bc
$(\tau_1,\tau_2,\xi)\mapsto \bd_\xi(\tau_1,\tau_2):=\beta^\tau(g_1^{-1},\xi)-\beta^\tau(g_2^{-1},\xi)$,
\ec

\noindent where $g_i\cdot\tau=\tau_i$ for $i=1,2$. A geometric interpretation of the $o$-Busemann cocycle is given by the following proposition (c.f. Figure \ref{fig geom interpretation obusemann} below).

\begin{prop}\label{prop betao in xsyg}
For every $(g,\xi)\in\gfo$ one has
\bc
$\beta^o(g,\xi)=\beta_\xi(\Pi^{g^{-1}\cdot o}(\xi),\Pi^o(\xi))$.
\ec
\end{prop}

\begin{proof}
Let $h'\in\h^o$ and $\hat{w}'\in\wbh$ be such that $h'\hat{w}'\cdot\xi_{\liea^+}=\xi$ and write

\bc
$gh'\hat{w}'=h\hat{w}\exp(\beta^o(g,\xi))n$.
\ec

\noindent Equivariance property (\ref{eq equivariance for Pio}) gives the following:

\bc
$\Pi^{g^{-1}\cdot o}(\xi)=g^{-1}h\hat{w}\cdot \tau$ and $\Pi^o(\xi)=h'\hat{w}'\cdot\tau$.
\ec

\noindent We then have

\bc
$\beta_\xi(\Pi^{g^{-1}\cdot o}(\xi),\Pi^{ o}(\xi))=\beta^\tau(\exp(\beta^o(g,\xi))n,\xi_{\liea^+})=\beta^o(g,\xi)$.
\ec
\end{proof}

\begin{figure}[h!]
\bc
\scalebox{0.5}{%
\begin{overpic}[scale=1, width=1\textwidth, tics=5]{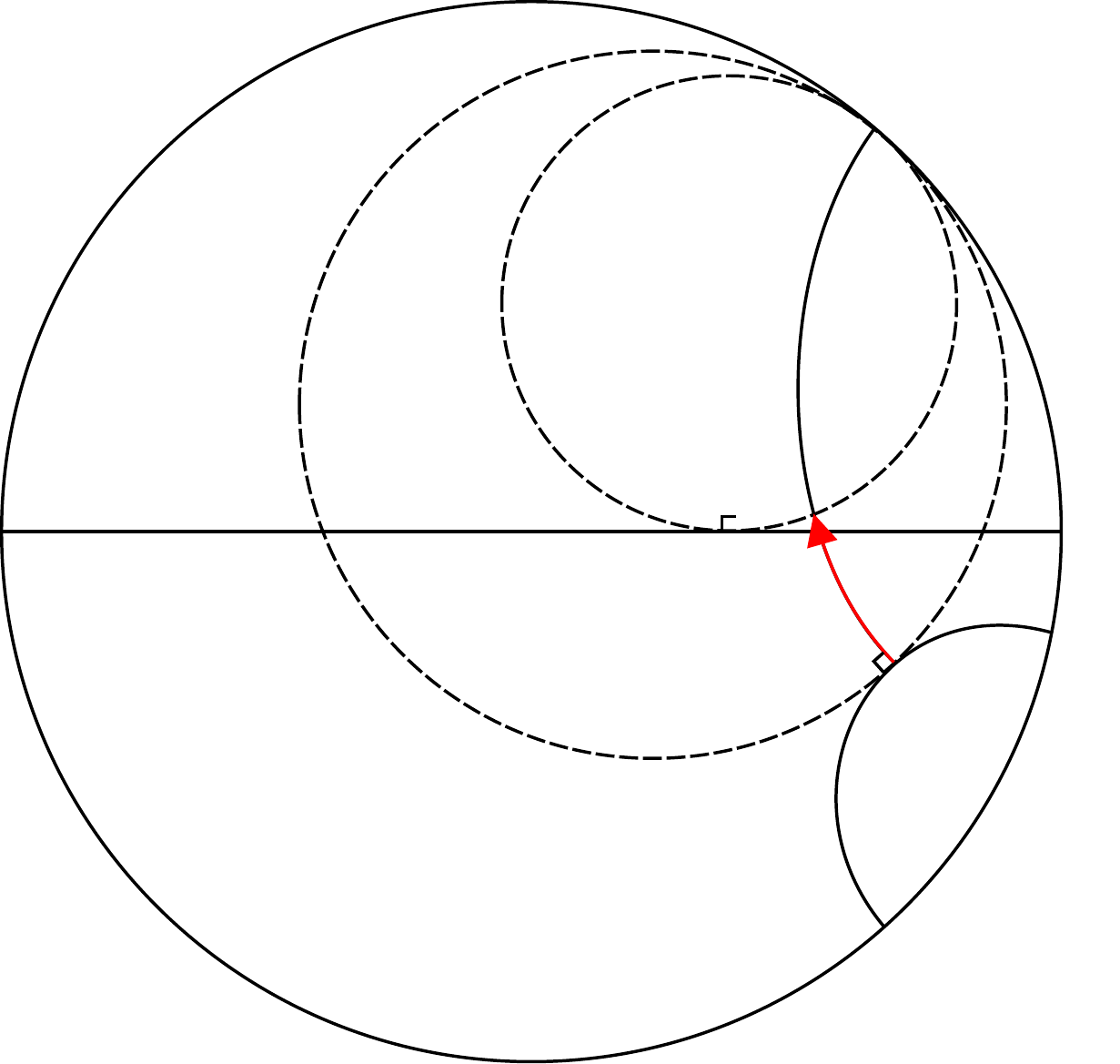}

\put (75,85) { \begin{Huge}
$\xi$
\end{Huge}}

\put (8,49) { \begin{Huge}
$\so$
\end{Huge}}

\put (15,29) { \begin{Huge}
$\xsyg$
\end{Huge}}

\put (59,13) { \begin{Huge}
$\mathsf{S}^{g^{-1}\cdot o}$
\end{Huge}}

  \end{overpic}}
 
\hspace{0,3cm}

\begin{changemargin}{3cm}{3cm}    
    \caption{The $o$-Busemann cocycle: the red arrow represents the vector $\beta^o(g,\xi)$.}
  \label{fig geom interpretation obusemann}
\end{changemargin}

  \ec
\end{figure}
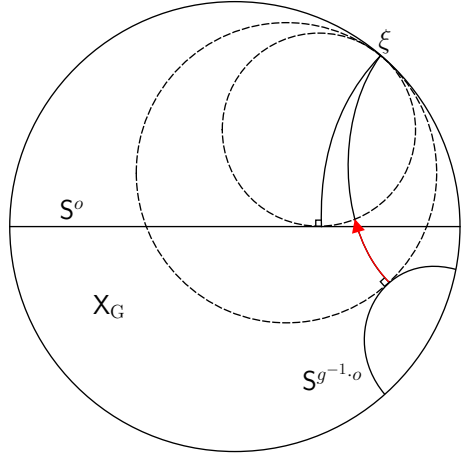

\subsection{$o$-Gromov product}\label{subsection gromov product in Qpq}

We now introduce the ``Gromov product" associated to the pair $(\iota_{\liea^+}\circ\beta^o,\beta^o)$. Its definition goes as follows. Let

\bc
$\fvoc:=\lbrace (\xi,\xi')\in\fvo\times\fvo:\hspace{0,3cm} \xi \tn{ is transverse to } \xi'\rbrace$.
\ec

\noindent For an element $(\xi,\xi')$ in $\fvoc$ and $j=1,\dots,d-1$ we know by Remark \ref{rem ordered basis induced by o generic and double perp and xi o generic iff gxi go generic} that the hyperplane $(\Lambda^j(\xi^{\perp_o}))^{\perp_{o_j}}$ is transverse to the line $\Lambda^j\xi'$. Further, by Lemma \ref{lema equivalences for o generic flags} the lines $\Lambda^j(\xi^{\perp_o})$ and $\Lambda^j\xi'$ are not isotropic for the form $o_j$. Therefore the following $o$-\textit{Gromov product}, that naturally generalizes the one introduced by Sambarino in \cite{Sam}, is well defined: let

\bc
$\gro_o:\fvoc\too\lieb$
\ec

\noindent be defined by the equalities

\bc
$\chi_j (\gro_o(\xi,\xi')):=\dfrac{1}{2}\log\left\vert\dfrac{\langle v,v'\rangle_{o_j}\langle v,v'\rangle_{o_j}}{\langle v,v\rangle_{o_j}\langle v',v'\rangle_{o_j}}\right\vert$
\ec

\noindent for every $j=1,\dots,d-1$, where $v$ (resp. $v'$) is any non zero vector in the line $\Lambda^j(\xi^{\perp_o})$ (resp. $\Lambda^j\xi'$).

The classical relation between Busemann functions and Gromov products is still satisfied in our framework.

\begin{lema}\label{lema gromov product is a gromov prod for betao}
Let $(\xi,\xi')\in\fvoc$ and $g\in\g$ be an element such that $(g,\xi)$ and $(g,\xi')$ belong to $\gfo$. Then the following equality holds:

\bc
$\gro_o(g\cdot\xi,g\cdot\xi')-\gro_o(\xi,\xi')=-(\iota_{\liea^+}\circ\beta^o(g,\xi)+\beta^o(g,\xi'))$.
\ec
\end{lema}

\begin{proof}

Fix $j=1,\dots,d-1$ and note that, by equation (\ref{eq sigma on lambda i commutes with sigma}) and Lemma \ref{lema action of sigma on flags}, if $v\in\Lambda^j(\xi^{\perp_o})$ then 

\bc
$\sigma^{o_j}(\Lambda^jg)\cdot v\in \Lambda^j((g\cdot\xi)^{\perp_o})$.
\ec

\noindent Let $v'\in\Lambda^j\xi'$ be a non zero vector. We have

\begin{align*}
\chi_j(\gro_o(g\cdot \xi,g\cdot \xi')) & =\dfrac{1}{2}\log\left\vert \dfrac{\langle \sigma^{o_j}(\Lambda^jg)\cdot v, \Lambda^jg\cdot v'\rangle_{o_j}\langle \sigma^{o_j}(\Lambda^jg)\cdot v, \Lambda^jg\cdot v'\rangle_{o_j}}{\langle \sigma^{o_j}(\Lambda^jg)\cdot v, \sigma^{o_j}(\Lambda^jg)\cdot v \rangle_{o_j}\langle \Lambda^jg\cdot v', \Lambda^jg\cdot v'\rangle_{o_j}} \right\vert\\
& =\dfrac{1}{2}\log\left\vert \dfrac{\langle  v,  v'\rangle_{o_j}\langle  v,  v'\rangle_{o_j}}{\langle \sigma^{o_j}(\Lambda^jg)\cdot v, \sigma^{o_j}(\Lambda^jg)\cdot v \rangle_{o_j}\langle \Lambda^jg\cdot v', \Lambda^jg\cdot v'\rangle_{o_j}} \right\vert,
\end{align*}

\noindent where the last equality holds by definition of $\sigma^{o_j}$. If we subtract to the previous equality the number

\bc
$\chi_j(\gro_o(\xi,\xi'))=\dfrac{1}{2}\log\left\vert\dfrac{\langle v,v'\rangle_{o_j}\langle v,v'\rangle_{o_j}}{\langle v,v\rangle_{o_j}\langle v',v'\rangle_{o_j}}\right\vert$
\ec

\noindent we obtain that $\chi_j(\gro_o(g\cdot \xi,g\cdot \xi'))-\chi_j(\gro_o( \xi, \xi'))$ equals

\bc
$\dfrac{1}{2}\log\left\vert \dfrac{\langle  v,  v\rangle_{o_j}\langle  v',  v'\rangle_{o_j}}{\langle \sigma^{o_j}(\Lambda^jg)\cdot v, \sigma^{o_j}(\Lambda^jg)\cdot v \rangle_{o_j}\langle \Lambda^jg\cdot v', \Lambda^jg\cdot v'\rangle_{o_j}} \right\vert$
\ec

\noindent and by Corollary \ref{cor iwasawa for sigma g} the result follows.

\end{proof}

We now discuss a geometric interpretation for the $o$-Gromov product, that will be of central importance in Section \ref{sec counting} (c.f. Corollary \ref{cor gromov is cross ratio for rho and estimate with bogamma}). Let $\bb$ be the \textit{vector valued cross-ratio}, defined by Benoist \cite[p. 6]{Ben1} in the following way. Given a $4$-tuple 

\bc
$(\xi_1,\xi_2,\xi_3,\xi_4)\in\mathsf{F}(V)^4$
\ec

\noindent such that $(\xi_i,\xi_k)$ belongs to $\fvc$ for all $(i,k)\in\lbrace(1,2),(1,4),(2,3),(3,4)\rbrace\rbrace$, the vector 

\bc
$\bb(\xi_1,\xi_2,\xi_3,\xi_4)\in\lieb$
\ec

\noindent is defined by the following equalities for $j=1,\dots,d-1$:

\bc
$\chi_j(\bb(\xi_1,\xi_2,\xi_3,\xi_4)):=\log\left\vert\dfrac{\vartheta_1(v_4)}{\vartheta_1(v_2)}	\dfrac{\vartheta_3(v_2)}{ \vartheta_3(v_4)}\right\vert$,
\ec

\noindent where for $i=2,4$, $v_i$ is any non zero vector in the line $\Lambda^j\xi_i\in\pp(\Lambda^jV)$ and for $k=1,3$, $\vartheta_k$ is any non zero linear functional in the line $\Lambda^j_*\xi_k\in\pp(\Lambda^jV^*)$.

\begin{cor}\label{cor gromov is cross ratio}
For every $(\xi,\xi')\in\fvoc$ the following equality holds:

\bc
$\gro_o(\xi,\xi')=-\frac{1}{2}\bb\left((\xi')^{\perp_o},\xi^{\perp_o},\xi,\xi'\right)$.
\ec
\end{cor}

\begin{proof}
Let $j=1,\dots,d-1$ and $v\in\Lambda^j(\xi^{\perp_o})$ and $v'\in\Lambda^j\xi'$ be non zero vectors. Then by Remark \ref{rem ordered basis induced by o generic and double perp and xi o generic iff gxi go generic} we have

\bc
$\langle v,\cdot \rangle_{o_j}\in \Lambda^j_*\xi$ and $\langle v',\cdot \rangle_{o_j}\in \Lambda^j_*((\xi')^{\perp_o})$.
\ec

\noindent It follows that

\bc
$\chi_j\left(\bb\left((\xi')^{\perp_o},\xi^{\perp_o},\xi,\xi'\right)\right)=\log\left\vert \dfrac{\langle v',v' \rangle_{o_j}\langle v,v \rangle_{o_j}}{\langle v',v \rangle_{o_j}\langle v,v' \rangle_{o_j}} \right\vert$
\ec

\noindent and the result is proven.

\end{proof}

\section{$(p,q)$-Cartan projection for Anosov representations}\label{sec pq Cartan for anosov}

In this section we begin the study of asymptotic properties of the $(p,q)$-Cartan projection for elements in the image of a $\Delta$-Anosov representation. In Subsection \ref{subsec def anosov} we briefly recall this notion and some of its main features. In Subsection \ref{subsec orho} we introduce the subset $\orho$ and discuss some examples. In Subsection \ref{subsec exponential rate} we prove Corollary \ref{cor bound for counting problem introduction} and in Subsection \ref{subsec asymptotic cone} we describe the $(p,q)$-asymptotic cone.

\subsection{Reminders on Anosov representations}\label{subsec def anosov}

A lot of work has been done in order to simplify Labourie's original definition of Anosov representations. Here we follow mainly the work of Bochi-Potrie-Sambarino \cite{BPS}, Guichard-Gu\'eritaud-Kassel-Wienhard \cite{GGKW} and Kapovich-Leeb-Porti \cite{KLP1}.

Let $\Gamma$ be a finitely generated group. Consider a finite symmetric generating set $\mathscr{S}$ of $\Gamma$ and take $\vert\cdot\vert_\Gamma$ to be the associated word length: for $\gamma$ in $\Gamma$, it is the minimum number required to write $\gamma$ as a product of elements\footnote{This number depends on the choice of $\mathscr{S}$. However, the set $\mathscr{S}$ will be fixed from now on hence we do not emphasize the dependence on this choice in the notation.} of $\mathscr{S}$.

Fix a Cartan decomposition $\g=\ko^\tau\exp(\liea^+)\ko^\tau$ and let $\Delta$ be the set of simple roots. A representation $\rho:\Gamma\too \g$ is said to be $\Delta$-\textit{Anosov} if there exist positive constants $c$ and $c'$ such that for all $\gamma\in\Gamma$ and all $\alpha\in\Delta$ one has
\begin{equation}\label{eq def anosov}
\alpha(a^\tau(\rho\gamma))\geq c\vert\gamma\vert_\Gamma-c'.
\end{equation}

By Kapovich-Leeb-Porti \cite[Theorem 1.4]{KLP3} (see also \cite[Section 3]{BPS}), condition (\ref{eq def anosov}) implies that $\Gamma$ is word hyperbolic\footnote{We refer the reader to the book of Ghys-de la Harpe \cite{GdlH} for definitions and standard facts on word hyperbolic groups.}. In this paper we assume that $\Gamma$ is non elementary. Let $\bg$ be the Gromov boundary of $\Gamma$ and $\gh$ be the set of infinite order elements in $\Gamma$. Every $\gamma$ in $\gh$ has exactly two fixed points in $\bg$: the attractive one denoted by $\gamma_+$ and the repelling one denoted by $\gamma_-$. The dynamics of $\gamma$ on $\bg$ is of type ``north-south".

As shown in \cite{BPS,GGKW,KLP1}, a central feature about $\Delta$-Anosov representations is that they admit a continuous equivariant map

\bc
$\xi_{\rho}:\bg\too\mathsf{F}(V)$
\ec

\noindent which is \textit{transverse}, i.e. for every $x\neq y$ in $\bg$ one has
\begin{equation} \label{eq transv condition}
(\xi_\rho(x),\xi_{\rho}(y))\in\fvc.
\end{equation}
\noindent Moreover, this map is \textit{dynamics-preserving}, i.e. for every $\gamma$ in $\gh$ the element $\xi_{\rho}(\gamma_+)$ (resp. $\xi_\rho(\gamma_-)$) is an attractive (resp. repelling) fixed point of $\rho\gamma$ acting on $\mathsf{F}(V)$. It follows that $\rho\gamma$ is loxodromic with

\bc
$\xi_{\rho}(\gamma_\pm)=(\rho\gamma)_\pm$.
\ec

\noindent In particular, $\xi_{\rho}$ is injective and uniquely determined by $\rho$: it is called the \textit{limit map} of $\rho$. This map varies in a continuous way with the representation and is H\"older continuous (see Guichard-Wienhard \cite[Theorem 5.13]{GW} and Bridgeman-Canary-Labourie-Sambarino \cite[Theorem 6.1]{BCLS}).

The \textit{limit set} of $\rho$ is the image of $\xi_\rho$ and admits the following useful characterization (see \cite[Theorem 5.3]{GGKW} or {\cite[Subsection 3.4]{BPS}} for a proof).

\begin{prop}\label{prop limit with S and U and Uuno cerca gammamas}
Let $\rho:\Gamma\too \g$ be a $\Delta$-Anosov representation. Then the limit set of $\rho$ coincides with the set of accumulation points of sequences of the form $\lbrace U^\tau(\rho\gamma_n)\rbrace_n$, where $\gamma_n\too\infty$. Moreover, given a (continuous) distance $d(\cdot,\cdot)$ on $\mathsf{F}(V)$ and a positive $\varepsilon$, one has
\bc
$d\left(U^\tau(\rho\gamma),(\rho\gamma)_+\right)<\varepsilon$ and $d(S^\tau(\rho\gamma),(\rho\gamma)_-)<\varepsilon$
\ec

\noindent for every $\gamma\in\gh$ with $\vert\gamma\vert_\Gamma$ large enough.
\end{prop}

Anosov representations have strong proximality properties (c.f. \cite[Subsection 5.2]{GW}). We now  establish one of them that will be useful in Section \ref{sec counting}: it will provide the correct framework to estimate the geometric quantity involved in our counting functions.

\begin{lema}[c.f. Sambarino {\cite[Lemma 5.7]{Sam}}]\label{lema sambarino lemma 5.7}
Let $\rho:\Gamma\too\g$ be a $\Delta$-Anosov representation and fix real numbers $0<\varepsilon\leq r$. Then there exists a positive $L$ with the following property: for every $\gamma\in\gh$ satisfying $\vert\gamma\vert_\Gamma>L$ and such that

\bc
$d(\Lambda^j(\rho\gamma)_+,\Lambda^j_* (\rho\gamma)_-)\geq 2r$
\ec

\noindent holds for every $j=1,\dots,d-1$, one has that $\rho\gamma$ is $(r,\varepsilon)$-loxodromic.

\end{lema}

\begin{proof}

Consider a sequence $\gamma_n\too\infty$ in $\gh$ such that 

\bc
$d(\Lambda^j(\rho\gamma_n)_+,\Lambda^j_* (\rho\gamma_n)_-)\geq 2r$
\ec

\noindent for all $n$ and $j$. By Proposition \ref{prop limit with S and U and Uuno cerca gammamas}, for every $n$ large enough the following holds

\bc
$b_{\frac{\varepsilon}{2}}(\Lambda^jU^\tau(\rho\gamma_n))\subset b_{\varepsilon}(\Lambda^j(\rho\gamma_n)_+)$ and $B_{\varepsilon}(\Lambda^j_*(\rho\gamma_n)_-)\subset B_{\frac{\varepsilon}{2}}(\Lambda^j_*S^\tau(\rho\gamma_n))$.
\ec

\noindent By Remark \ref{rem complemento de Sdmenosuno va en Uuno} and equation (\ref{eq def anosov}) the condition 

\bc
$\rho\gamma_n\cdot B_{\varepsilon}(\Lambda^j_*(\rho\gamma_n)_-)\subset b_{\varepsilon}(\Lambda^j(\rho\gamma_n)_+)$
\ec

\noindent is sa\-tis\-fied for sufficiently large $n$.

\end{proof}

\subsection{The set $\orho$}\label{subsec orho}
Given a $\Delta$-Anosov representation $\rho$ define the (open) set

\bc
$\orho:=\lbrace o\in\xsypq:\hspace{0,3cm} \xi_\rho(\bg)\subset\mathsf{F}(V)^o\rbrace$.
\ec

Let us discuss some examples of $\Delta$-Anosov representations into $\g$ for which the set $\orho$ is non empty. Observe that since the limit map of an Anosov representation varies in a continuous way, if $\orho$ is non empty for some specific $\rho$ the same will hold for every small enough deformation of $\rho$.

\begin{ex}\label{ex orho non empty in xsypq}\item

\begin{itemize}

\item The simplest way of constructing a $\Delta$-Anosov representation $\rho$ is to use a ``Schottky construction" (see Benoist \cite{Ben3}). If one fixes beforehand a basepoint $o\in\xsypq$, it is easy to construct this representation in such a way that $o\in\orho$.

\item Suppose that $p$ and $q$ are different modulo $2$ and consider a splitting

\bc
$V=\pi^+\oplus\pi^-$
\ec

\noindent where $\pi^+$ (resp. $\pi^-$) is a $p$-dimensional (resp. $q$-dimensional) subspace of $V$. Consider a representation

\bc
$\Lambda:\tn{SL}_2(\kk)\too\tn{SL}(V): \hspace{0,3cm} \Lambda:=\Lambda^{\tn{irr}}_+\oplus\Lambda^{\tn{irr}}_-$
\ec

\noindent where $\Lambda^{\tn{irr}}_\pm:\tn{SL}_2(\kk)\too\tn{SL}(\pi^\pm)$ are irreducible. Let $\rho_0$ be given by

\bc
$\Gamma\too\tn{SL}_2(\kk)\too\g$,
\ec

\noindent where the first arrow is an Anosov representation into $\tn{SL}_2(\kk)$ and the second arrow is induced by $\Lambda$. Then $\rho_0$ is $\Delta$-Anosov. Pick a form $o\in\xsypq$ for which the splitting 

\bc
$V=\pi^+\oplus\pi^-$
\ec

\noindent is orthogonal and such that $\pi^+$ (resp. $\pi^-$) is positive definite (resp. negative definite). Then the point $o$ belongs to $\pmb{\Omega}_{\rho_0}$.

\item Let $p=2$ and $q=1$ and consider this time a Hitchin representation $\rho:\Gamma_g\too\tn{PSL}_3(\rr)$, where $\Gamma_g$ is the fundamental group of a closed orientable surface of genus $g\geq 2$ (see Labourie \cite{Lab}). Suppose that $o$ is a point in $\mathsf{Q}_{2,1}$ such that for every $x\in\partial_\infty\Gamma_g$ either

\begin{enumerate}
\item[(i)] the line $\xi_\rho^1(x)$ is negative definite for $o$,
\end{enumerate}

\noindent or
\begin{enumerate}
\item[(ii)] the hyperplane $\xi_\rho^2(x)$ is positive definite for $o$.
\end{enumerate}

\noindent In any of these situations one has $o\in\orho$. Since the \textit{projective limit set} $\xi_\rho^1(\partial_\infty\Gamma_g)$ of $\rho$ is contained in an affine chart of $\pp(V)$ (see Choi-Goldman \cite{ChoiGoldman}), it is not hard to construct forms of signature $(2,1)$ on $V$ for which either (i) or (ii) above is satisfied. 

\item The previous example generalizes to all odd dimensions as follows. Let $\rho:\Gamma_g\too\tn{PSL}_{2k+1}(\rr)$ be a Hitchin representation. As shown by Danciger-Gu\'eritaud-Kassel \cite[Proposition 1.7]{DGK2} (see also Zimmer \cite[Corollary 1.33]{Zim}), there exists a $\Gamma_g$-invariant non empty open subset $C\subset\pp(\rr^{2k+1})$ disjoint from $\xi_\rho^{2k}(\partial_\infty\Gamma_g)$. Then any element $o\in\mathsf{Q}_{2k,1}$ whose projectivized isotropic cone is contained in $C$ belongs to $\orho$.

\end{itemize}
\end{ex}

\subsection{Proof of Corollary \ref{cor bound for counting problem introduction}}\label{subsec exponential rate}

For the rest of the paper we fix a $\Delta$-Anosov representation $\rho:\Gamma\too\g$ and a $(p,q)$-Cartan decomposition of $\bog$, for a basepoint $o\in\orho$. The following is a consequence of Lemma \ref{lema Uo cerca de Utau y lo mismo con S} and Proposition \ref{prop limit with S and U and Uuno cerca gammamas}.

\begin{cor}\label{cor sigmagammainvgamma is loxod and Uo close to Utau}

There exist $0<\varepsilon_0\leq r_0$ with the following property: for every $0<\varepsilon\leq r$ such that $\varepsilon\leq\varepsilon_0$ and $r\leq r_0$, there exists a positive $L$ such that if $\gamma\in\Gamma$ satisfies $\vert\gamma\vert_\Gamma >L$ then one has that 

\bc
$\sigma^{o}(\rho\gamma^{-1})\rho\gamma$
\ec

\noindent is $(2r,2\varepsilon)$-loxodromic. In particular $\rho\gamma$ belongs to $\bog$. Furthermore, given a positive $\delta$ the number $L>0$ can be chosen in such a way that
\bc
$d(U^o(\rho\gamma),U^\tau(\rho\gamma))<\delta$ and $d(S^o(\rho\gamma),S^\tau(\rho\gamma))<\delta$.
\ec

\end{cor}

The \textit{entropy} of $\rho$, introduced by Bridgeman-Canary-Labourie-Sambarino in \cite{BCLS}, is defined by

\bc
$h^1_\rho:=\displaystyle\limsup_{t\too\infty}\dfrac{\log\#\left\lbrace [\gamma]\in[\Gamma]: \hspace{0,3cm}  \lambda_1(\rho\gamma)\leq t\right\rbrace}{t}$.
\ec

\noindent Here $[\gamma]$ denotes the conjugacy class of the element $\gamma\in\Gamma$ and $\lambda_1(\cdot)$ is defined as in Subsection \ref{subsec reminders cartan and jordan} (for any given Weyl chamber of the system $\Sigma(\lieg,\liea)$). The entropy of $\rho$ is positive and finite (see \cite[Sections 4 \& 5]{BCLS}).

On the other hand, the \textit{projective critical exponent} of $\rho$ is defined by

\bc
$\delta^1_\rho:=\displaystyle\limsup_{t\too\infty}\dfrac{\log\#\left\lbrace \gamma\in\Gamma: \hspace{0,3cm}  a^\tau_1(\rho\gamma)\leq t\right\rbrace}{t}$.
\ec

\noindent A consequence of Sambarino's work \cite{Sam} is that the entropy of $\rho$ coincides with the projective critical exponent\footnote{In \cite{Sam} the author treats the case in which $\Gamma$ is the fundamental group of a closed negatively curved manifold. His results remain valid when $\Gamma$ is a word hyperbolic group admitting an Anosov representation, as proven by Glorieux-Tholozan-Monclair \cite[Theorem 2.31]{GloThoMon} (see also Appendix \ref{app proof of ditribution}).}. We conclude that $\delta_\rho^1$ is positive and finite, and therefore the \textit{critical exponent}

\bc
$\delta_\rho:=\displaystyle\limsup_{t\too\infty}\dfrac{\log\#\left\lbrace \gamma\in\Gamma: \hspace{0,3cm} \Vert a^\tau(\rho\gamma)\Vert_\lieb\leq t\right\rbrace}{t}$
\ec

\noindent of $\rho$ must also be positive and finite. 

We now prove Corollary \ref{cor bound for counting problem introduction}.

\begin{cor}\label{cor counting function and critical exponent}
The following holds:

\begin{enumerate}
\item 
The function

\bc
$t\mapsto\#\left\lbrace \gamma\in\Gamma: \hspace{0,3cm} \rho\gamma\in\bog \tn{ and } \Vert b^o(\rho\gamma)\Vert_\lieb\leq t\right\rbrace$
\ec

\noindent is finite for every positive $t$. Moreover, the following equality is satisfied

\bc
$\delta_\rho=\displaystyle\limsup_{t\too\infty}\dfrac{\log\#\left\lbrace \gamma\in\Gamma: \hspace{0,3cm} \rho\gamma\in\bog \tn{ and } \Vert b^o(\rho\gamma)\Vert_\lieb\leq t\right\rbrace}{t}$.
\ec

\noindent In particular, the right hand side of the last equality is finite, positive and independent on the choice of the basepoint $o\in\orho$.

\item Suppose that $\rho$ is Zariski dense. Then there exist positive constants $\mathtt{C}_1$ and $\mathtt{C}_2$ such that for every $t$ large enough one has

\bc
$\mathtt{C}_1e^{\delta_\rho t}\leq  \#\lbrace \gamma\in\Gamma:\hspace{0,3cm} \rho\gamma\in\bog \tn{ and } \Vert b^o(\rho\gamma)\Vert_\lieb\leq t\rbrace  \leq\mathtt{C}_2e^{\delta_\rho t}$.
\ec
\end{enumerate}

\end{cor}

\begin{proof}

Let $\liea^+$ be a $\lieb^+$-compatible Weyl chamber.
\begin{enumerate}
\item Follows from Propositions \ref{prop bo close to wdotcartan} and \ref{prop limit with S and U and Uuno cerca gammamas}, and $\wb$-invariance of $\Vert\cdot\Vert_\lieb$.
\item By the work of Sambarino \cite{Sam2} (see Theorem \ref{teo counting sambarino} for a proof in our setting) there exists a positive constant $\mathtt{C}$ such that 

\bc
$\mathtt{C}e^{-\delta_\rho t}  \#\lbrace \gamma\in\Gamma:\hspace{0,3cm}  \Vert a^\tau(\rho\gamma)\Vert_\lieb\leq t\rbrace\too 1$
\ec

\noindent as $t\too\infty$. Propositions \ref{prop bo close to wdotcartan} and \ref{prop limit with S and U and Uuno cerca gammamas} and $\wb$-invariance of $\Vert\cdot\Vert_\lieb$ finish the proof.
\end{enumerate}
\end{proof}

\subsection{$(p,q)$-asymptotic cone}\label{subsec asymptotic cone}

Fix a $\lieb^+$-compatible Weyl chamber $\tilde{\liea}^+$ for which the intersection 

\bc
$\xi_\rho(\bg)\cap\fvo_{\tilde{\liea}^+}$
\ec

\noindent is non empty and set

\bc
$\liea^+:=\iota_{\lieb^+}(\tilde{\liea}^+)$.
\ec

\noindent Let $a^\tau:\g\too\liea^+$ be the associated Cartan projection and denote by $\cont$ the \textit{asymptotic cone} of $\rho(\Gamma)$, introduced by Benoist in \cite{Ben4}. By definition, it is the subset of $\liea^+$ consisting on all possible limits of sequences of the form
\begin{equation}\label{eq def limit cont}
\dfrac{a^\tau(\rho\gamma_n)}{t_n}
\end{equation}
\noindent where $t_n\too\infty$. Benoist \cite{Ben4,Ben1} showed that, for Zariski dense subgroups of $\g$, the set $\cont$ coincides with the smallest closed cone containing $\lambda(\rho(\Gamma))$ (which is also showed to be convex and with non empty interior).

We define a new asymptotic cone, that we denote by $\cono$, and that consists on all possible limits of sequences of the form (\ref{eq def limit cont}) but with $a^\tau(\rho\gamma_n)$ replaced by $b^o(\rho\gamma_n)$. The main goal of this subsection is to give an explicit description of this new cone by means of Benoist's asymptotic cone: see Proposition \ref{prop limit cono} below. As a consequence of this description (c.f. Remark \ref{rem topology of the cone depends on o}), we note that the topology of $\cono$ may depend on the specific choice of the basepoint $o$. However, since $o\in\orho$ is fixed for the rest of the paper, we prefer not to stress this dependence in the notation.

Define

\bc
$\war:=\lbrace w\in\wb:\hspace{0,3cm} w\cdot\liea^+\subset \lieb^+\tn{ and } \xi_\rho(\bg)\cap\fvo_{\iota_{\lieb^+}(w\cdot\liea^+)}\neq\emptyset \rbrace$.
\ec

\begin{rem}\label{rem war equal to one iff limit set in unique open orbit}
By definition the identity element of $\wb$ belongs to $\war$. Moreover, the equality $\war=\lbrace 1\rbrace$ is equivalent to the inclusion

\bc
$\xi_\rho(\bg)\subset\fvo_{\tilde{\liea}^+}$.
\ec

\noindent In particular, if $\bg$ is connected we always have $\war=\lbrace 1\rbrace$. However, one can do a Schottky construction as in Example \ref{ex orho non empty in xsypq} to find $\Delta$-Anosov representations for which $\war$ does not consists on a single element: it suffices to play ``ping-pong" with matrices whose attractors and repellors belong to different  open orbits of the action of $\h^o$ on $\mathsf{F}(V)$ (c.f. Figure \ref{fig pingpontwopoints} below).
\end{rem}

\begin{rem}\label{rem for prop of limit cone and growth indicator}
Let $\lbrace\gamma_n\rbrace$ be a sequence in $\Gamma$ diverging to infinity and suppose that there exists a $\lieb^+$-compatible Weyl chamber $\hat{\liea}^+$ such that

\bc
$S^\tau(\rho\gamma_n)\in\fvo_{\hat{\liea}^+}$
\ec

\noindent for every $n$ large enough. Let $w\in\war$ be the element defined by the equality

\bc
$w\cdot\liea^+=\iota_{\lieb^+}(\hat{\liea}^+)$.
\ec

\noindent Then Proposition \ref{prop bo close to wdotcartan}, the definition of a $\Delta$-Anosov representation and Proposition \ref{prop limit with S and U and Uuno cerca gammamas} imply that the sequence

\bc
$\Vert b^o(\rho\gamma_n)-w\cdot a^\tau(\rho\gamma_n)\Vert_\lieb$
\ec

\noindent is bounded.
\end{rem}

\begin{prop}\label{prop limit cono}
The following equality holds:

\bc
$\cono=\displaystyle\bigcup_{w\in\war}w\cdot\cont$.
\ec
\end{prop}

\begin{proof}

We first prove the inclusion

\bc
$\cono\subset\displaystyle\bigcup_{w\in\war}w\cdot\cont$.
\ec

\noindent Let 

\bc
$X=\displaystyle\lim_{n\too\infty}\frac{b^o(\rho\gamma_n)}{t_n}$
\ec

\noindent be a point in $\cono$. By taking a subsequence if necessary we may assume

\bc
$S^\tau(\rho\gamma_n)\in\fvo_{\hat{\liea}^+}$,
\ec

\noindent for all $n$ and some Weyl chamber $\hat{\liea}^+\subset\lieb^+$. Take $w\in\war$ as in Remark \ref{rem for prop of limit cone and growth indicator}. Then the sequence

\bc
$\frac{1}{t_n}\Vert b^o(\rho\gamma_n)-w\cdot a^\tau(\rho\gamma_n)\Vert_{\lieb}$
\ec

\noindent converges to zero and we conclude that $X$ belongs to $w\cdot\cont$.

Conversely, let $w\in\war$ and

\bc
$X=\displaystyle\lim_{n\too\infty}\dfrac{a^\tau(\rho\gamma_n)}{t_n}$
\ec

\noindent be a point in $\cont$. Define $\hat{\liea}^+:=\iota_{\lieb^+}(w\cdot\liea^+)$, which is a $\lieb^+$-compatible Weyl chamber and, by definition of $\war$, the intersection 

\bc
$\xi(\bg)\cap\fvo_{\hat{\liea}^+}$
\ec

\noindent is non empty. By taking a subsequence if necessary we may suppose

\bc
$ S^\tau(\rho\gamma_n)\too\xi_\rho(y)$
\ec

\noindent as $n\too\infty$, for some $y\in\bg$. Since the intersection $\xi_\rho(\bg)\cap\fvo_{\hat{\liea}^+}$ is non empty we can fix an element $\gamma^0$ in $\Gamma$ such that

\bc
$(\rho\gamma^0)^{-1}\cdot\xi_\rho(y)\in\fvo_{\hat{\liea}^+}$.
\ec

\noindent Further, by Proposition \ref{prop limit with S and U and Uuno cerca gammamas} we may choose the element $\gamma^0$ in such a way that the flag $U^\tau(\rho\gamma^0)$ is transverse to $S^\tau(\rho\gamma_n)$ for all $n$ large enough. Hence we find a constant $D>0$ such that for every $j=1,\dots,d-1$ one has

\bc
$d\left(\Lambda^jU^\tau(\rho\gamma^0),\Lambda^j_*S^\tau(\rho\gamma_n)\right)\geq D$
\ec

\noindent for every $n$ large enough. Therefore the sequence

\bc
$\Vert a^\tau(\rho(\gamma_n\gamma^0))-a^\tau(\rho\gamma_n) \Vert_\lieb$
\ec

\noindent is bounded (see e.g. \cite[Lemma A.7]{BPS}) and we conclude that

\bc
$X=\displaystyle\lim_{n\too\infty}\dfrac{a^\tau(\rho\gamma_n)}{t_n}=\displaystyle\lim_{n\too\infty}\dfrac{a^\tau(\rho(\gamma_n\gamma^0))}{t_n}$.
\ec

\noindent Thanks to Remark \ref{rem for prop of limit cone and growth indicator} in order to finish it suffices to show that $S^\tau(\rho(\gamma_n\gamma^0))$ belongs to $\fvo_{\hat{\liea}^+}$ for every $n$ large enough. But applying \cite[Lemma A.5]{BPS} we have

\bc
$\displaystyle\lim_{n\too\infty} S^\tau(\rho(\gamma_n\gamma^0))=(\rho\gamma^0)^{-1}\cdot\xi_\rho(y)$
\ec

\noindent which by construction belongs to $\fvo_{\hat{\liea}^+}$.

\end{proof}

\begin{rem}\label{rem topology of the cone depends on o}
The topology of $\cono$ may depend on the choice of the basepoint $o$. Indeed, one can fix beforehand two points $o_1$ and $o_2$ in $\xsypq$ and find a Schottky representation $\rho$ such that $\xi_\rho(\bg)$ is contained in a single open orbit of the action $\h^{o_1}\curvearrowright\mathsf{F}(V)$, but in more than one open orbit for the action of $\h^{o_2}$ (see Figure \ref{fig pingpontwopoints}). Note however that, by Proposition \ref{prop limit cono}, the asymptotic cone $\cono$ is independent on the choice of basepoints $o$ for which the limit set is contained in a single open orbit of the action $\h^o\curvearrowright\mathsf{F}(V)$.
\end{rem}

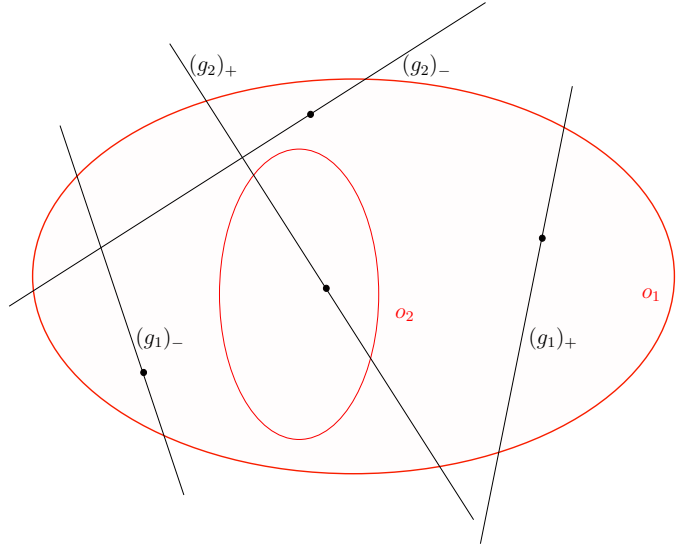
\begin{figure}[h!]
\bc
\scalebox{0.7}{%
\begin{overpic}[scale=1, width=1\textwidth, tics=5]{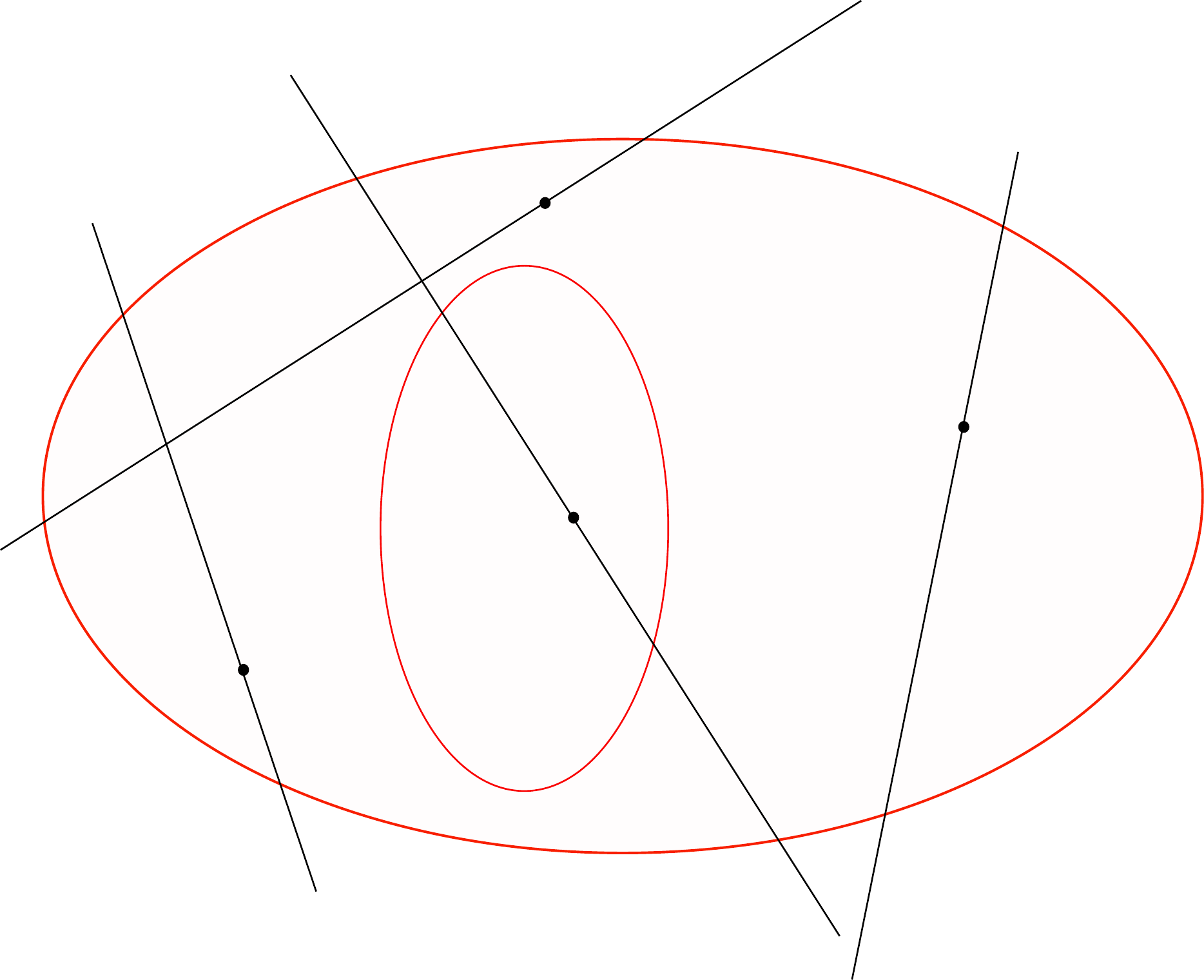}

\put (58,34) {\tc{red}{\Large$o_2$}}

\put (95,37) {\tc{red}{\Large$o_1$}}

\put (78,30) {\Large$(g_1)_+$}

\put (19,30) {\Large$(g_1)_-$}

\put (27,71) {\Large$(g_2)_+$}

\put (59,71) {\Large$(g_2)_-$}

  \end{overpic}}
 
\hspace{0,3cm}

\begin{changemargin}{3cm}{3cm}    
    \caption{An example of a free Schottky subgroup of $\g$ generated by (large enough) powers of $g_1,g_2\in\g$. The red conics represent the projectivized isotropic cones of points $o_1,o_2\in\mathsf{Q}_{2,1}$. The $(p,q)$-asymptotic cones for $o_1$ and $o_2$ have different topologies.}
  \label{fig pingpontwopoints}
\end{changemargin}

  \ec
\end{figure}

\section{Counting on a given direction}\label{sec counting}

Through this section we assume further that there exists a single open orbit of the action $\h^o\curvearrowright\mathsf{F}(V)$ that contains the limit set $\xi_\rho(\bg)$. The goal of this section is to prove Theorem \ref{TEO CONTEO DIRECCIONAL BO EN QPQ INTRODUCCION} (see Proposition \ref{prop linear counting for varphibo when torsion} below). The method we use is that of Roblin \cite{Rob} and Sambarino \cite{Sam}, therefore in Subsection \ref{subsec buseman and gromov for rho} we introduce a pair of dual H\"older cocycles over $\bg$ and a corresponding Gromov product to study our problem. In Subsection \ref{subsec PS} we introduce the corresponding Patterson-Sullivan measures and in Subsection \ref{subsec dist fixed points} we establish an equidistribution result for fixed points of elements in $\gh$, due to Sambarino \cite[Proposition 4.3]{Sam}. From this result we deduce our main theorem.

To introduce our pair of dual cocycles over $\bg$ we use the $o$-Busemann cocycle of Section \ref{sec busemann cocycles and gromov}. Recall that to define this cocycle we need to pick a $\lieb^+$-compatible Weyl chamber. Our assumption on the limit set of $\rho$ provide us with a $\lieb^+$-compatible Weyl chamber $\tilde{\liea}^+$ such that
\begin{equation}\label{eq limit set in unique open orbit}
\xi_\rho(\bg)\subset\fvo_{\tilde{\liea}^+}.
\end{equation}
\noindent We set $\liea^+:=\iota_{\lieb^+}(\tilde{\liea}^+)$ and with this Weyl chamber we will construct our H\"older cocycles over $\bg$. This is the reason why we need assumption (\ref{eq limit set in unique open orbit}) to hold.

By Remark \ref{rem war equal to one iff limit set in unique open orbit} we have 

\bc
$\war=\lbrace 1\rbrace$
\ec 

\noindent and Proposition \ref{prop limit cono} gives us the equality $\cono=\cont$.

Let $w\in\wb$ be the unique element of the Weyl group such that 

\bc
$w\cdot\liea^+=\tilde{\liea}^+$.
\ec

\noindent We have the following equalities:

\bc
$w_{\tilde{\liea}^+\iota_{\lieb^+}(\tilde{\liea}^+)}=w$ and $w_{\iota_{\lieb^+}(\tilde{\liea}^+)\liea^+}=1$.
\ec

\noindent Therefore the following is a consequence of Corollary \ref{cor cartan decomposition for Sg and Ug in single open orbit}, Proposition \ref{prop jordan versus generalized cartan} and Remark \ref{rem for prop of limit cone and growth indicator}.

\begin{cor}\label{cor if limit set in one orbit then HwexplieaplusH}

There exist positive constants $L$ and $D$ such that for every $\gamma$ in $\Gamma$ with $\vert\gamma\vert_\Gamma>L$ one has

\bc
$\rho\gamma\in\h^o w\exp\left(\frac{1}{2}\lambda(\sigma^o(\rho\gamma^{-1})\rho\gamma)\right)\h^o$
\ec

\noindent and

\bc
$\Vert b^o(\rho\gamma)- a^\tau(\rho\gamma)\Vert_\lieb\leq D$.
\ec
\end{cor}

\subsection{$o$-Busemann cocycle and $o$-Gromov product for $\rho$}\label{subsec buseman and gromov for rho}

Fix a linear functional $\varphi\in\lieb^*$ in the interior of the dual cone $\cont^*$. The $\varphi$-\textit{entropy} of $\rho$ is defined by

\bc
$h_\rho^\varphi:=\displaystyle\limsup_{t\too\infty}\dfrac{\log\#\left\lbrace [\gamma]\in[\Gamma]: \hspace{0,3cm} \varphi(\lambda(\rho\gamma))\leq t\right\rbrace}{t}$.
\ec

\noindent Since the entropy $h_\rho^1=h_\rho^{\varepsilon_1}$ of $\rho$ is positive and finite and $\varphi$ is positive on $\cont$, we conclude that $h_\rho^\varphi$ must also be positive and finite.

Let

\bc

$c_o^\varphi:\Gamma\times\bg\to \rr: \hspace{0,3cm} c_o^\varphi(\gamma,x):=\varphi(\beta^o(\rho\gamma,\xi_\rho(x)))$

\ec

\noindent and

\bc

$\overline{c}_o^\varphi:\Gamma\times\bg\to \rr:\hspace{0,3cm} \overline{c}_o^\varphi(\gamma,x):=\varphi(\iota_{\liea^+}\circ\beta^o(\rho\gamma,\xi_\rho(x)))$,

\ec

\noindent where $\beta^o$ is the $o$-Busemann cocycle (associated to $\liea^+$). These are called the $(\varphi,o)$-\textit{Busemann cocycles} of $\rho$.

Recall that a \textit{H\"older cocycle} over $\bg$ is a function 

\bc
$c:\Gamma\times\bg\too\rr$
\ec

\noindent satisfying that for every $\gamma_0,\gamma_1$ in $\Gamma$ and $x$ in $\bg$ one has 

\bc
$c(\gamma_0\gamma_1,x)=c(\gamma_0,\gamma_1\cdot x)+c(\gamma_1,x)$,
\ec

\noindent and such that the map $c(\gamma_0,\cdot)$ is H\"older continuous (with the same exponent for every $\gamma_0$). The \textit{period} of an element $\gamma\in \gh$ for such a cocycle is defined by

\bc
$p_{c}(\gamma):=c(\gamma,\gamma_+)$.
\ec

\noindent This is an invariant of the conjugacy class $[\gamma]$ of $\gamma$. If $c$ has positive periods, we let 

\bc
$h_c:=\displaystyle\limsup_{t\too\infty} \dfrac{\log\#\lbrace [\gamma]\in[\Gamma]:\hspace{0,3cm} p_c(\gamma)\leq t \rbrace}{t}\in[0,\infty]$
\ec

\noindent be the \textit{entropy} of $c$. A H\"older cocycle $\overline{c}$ over $\bg$ is said to be \textit{dual} to $c$ if the equality

\bc
$p_{\overline{c}}(\gamma)=p_{c}(\gamma^{-1})$
\ec

\noindent holds for every $\gamma\in\gh$. Note that dual cocycles have the same entropy.

The following lemma holds by direct computations (c.f. Lemma \ref{lema o busemann cocycle is a cocycle}).

\begin{lema}\label{lema covarphi and overline covarphi is a pair of dual cocycles and periods}
The pair $(\overline{c}_o^\varphi, c_o^\varphi)$ is a pair of dual H\"older cocycles. The periods of $c_o^\varphi$ are given by

\bc
$p_{c_o^\varphi}(\gamma)=\varphi(\lambda(\rho\gamma))>0$
\ec

\noindent for every $\gamma\in\gh$. In particular, one has the equalities

\bc
$h_{c_o^\varphi}=h_\rho^\varphi=h_{\overline{c}_o^\varphi}$.
\ec
\end{lema}

\begin{rem}\label{rem covarphi is cohomologous to ctauvarphi}

The definition of the $(\varphi,o)$-Busemann cocycles of $\rho$ takes inspiration from Sambarino's work \cite{Sam}. The author defines

\bc

$c_\tau^\varphi:\Gamma\times\bg\to \rr: \hspace{0,3cm} c_\tau^\varphi(\gamma,x):=\varphi(\beta^\tau(\rho\gamma,\xi_\rho(x)))$

\ec

\noindent and

\bc

$\overline{c}_\tau^\varphi:\Gamma\times\bg\to \rr:\hspace{0,3cm} \overline{c}_\tau^\varphi(\gamma,x):=\varphi(\iota_{\liea^+}\circ\beta^\tau(\rho\gamma,\xi_\rho(x)))$,

\ec

\noindent where $\beta^\tau$ is the $\tau$-Busemann cocycle of $\g$ (see \cite[Section 7]{Sam}). The cocycles $c_o^\varphi$ and $c_\tau^\varphi$ (resp. $\overline{c}_o^\varphi$ and $\overline{c}_\tau^\varphi$) are \textit{cohomologous} in the sense of Liv\v{s}ic \cite{Liv}. By definition this means that there exist H\"older continuous functions

\bc
$v:\bg\too\rr$ and $\overline{v}:\bg\too\rr$
\ec

\noindent such that for every $\gamma$ in $\Gamma$ and $x$ in $\bg$ one has

\bc
$c_o^\varphi(\gamma,x)-c_\tau^\varphi(\gamma,x)=v(\gamma\cdot x)-v(x)$
\ec

\noindent and 

\bc
$\overline{c}_o^\varphi(\gamma,x)-\overline{c}_\tau^\varphi(\gamma,x)=\overline{v}(\gamma\cdot x)-\overline{v}(x)$.
\ec

\noindent Indeed, this follows directly from Remark \ref{rem o Busemann cohomologous to tau Busemann}.
\end{rem}

Define

\bc
$[\cdot,\cdot]_o^\varphi:\bgc\too\rr: \hspace{0,3cm} [x,y]_o^\varphi:=\varphi(\gro_o(\xi_\rho(x),\xi_\rho(y)))$.
\ec

\noindent The following is a consequence of Lemma \ref{lema gromov product is a gromov prod for betao}.

\begin{cor}\label{cor gromov o is gromov for co and overline co}
The map $[\cdot,\cdot]_o^\varphi$ is a Gromov product for the pair $(\overline{c}_o^\varphi,c_o^\varphi)$, that is, for every $\gamma\in\Gamma$ and every $(x,y)\in\bgc$ one has

\bc
$[\gamma\cdot x,\gamma\cdot y]_o^\varphi-[x,y]_o^\varphi=-(\overline{c}_o^\varphi(\gamma,x)+c_o^\varphi(\gamma,y))$.
\ec
\end{cor}

We now state a crucial result that allows us to compare $\varphi(b^o(\rho\gamma))$ with the period $\varphi(\lambda(\rho\gamma))$ by means of the Gromov product defined above. This is the analogue of C. \cite[Lemma 6.6(4)]{CarHpq} in the present framework.

\begin{cor}\label{cor gromov is cross ratio for rho and estimate with bogamma}

Fix a positive $\delta$ and $A$ and $B$ two disjoint compact subsets of $\bg$. Then there exists a positive $L$ such that for every $\gamma\in\gh$ satisfying $\vert\gamma\vert_\Gamma >L$ and $(\gamma_-,\gamma_+)\in A\times B$ one has

\bc
$\vert\varphi( b^o(\rho\gamma))-\varphi(\lambda(\rho\gamma))+[\gamma_-,\gamma_+]_o^\varphi\vert < \delta$.
\ec
\end{cor}

\begin{proof}

From Corollaries \ref{cor sigmaginver has gaps/hyperbolic and attractrep} and \ref{cor gromov is cross ratio} we know that 

\bc
$[\gamma_-,\gamma_+]_o^\varphi=-\frac{1}{2}\varphi(\bb(\sigma^o(\rho\gamma^{-1})_-,\sigma^o(\rho\gamma^{-1})_+,(\rho\gamma)_-,(\rho\gamma)_+))$
\ec

\noindent holds for every $\gamma\in\gh$. Because of Corollary \ref{cor if limit set in one orbit then HwexplieaplusH} we can suppose that

\bc
$b^o(\rho\gamma)=\frac{1}{2}\lambda(\sigma^o(\rho\gamma^{-1})\rho\gamma)$
\ec

\noindent holds as well. To finish the proof we apply Benoist \cite[Lemme 3.4]{Ben1}. Indeed, by transversality condition (\ref{eq transv condition}) we can find a positive $r$ for which for every $\gamma$ as in the statement and every $j=1,\dots,d-1$ one has:

\bc
$d(\Lambda^j(\rho\gamma)_+,\Lambda^j_*(\rho\gamma)_-)\geq 2r$.
\ec

\noindent Given a positive $\varepsilon\leq r$, Lemma \ref{lema sambarino lemma 5.7} states that we can assume further that $\rho\gamma$ is $(r,\varepsilon)$-loxodromic and, because of Corollary \ref{cor sigmaginver has gaps/hyperbolic and attractrep}, we can suppose that analogue assertions hold for $\sigma^o(\rho\gamma^{-1})$. Moreover, by changing $r$ by a smaller constant if necessary we have

\bc
$d(\Lambda^j\sigma^o(\rho\gamma^{-1})_+,\Lambda^j_*(\rho\gamma)_-)\geq 6r$ and $d(\Lambda^j(\rho\gamma)_+,\Lambda^j_*\sigma^o(\rho\gamma^{-1})_-)\geq 6r$
\ec

\noindent for every $j=1,\dots,d-1$. By \cite[Lemme 3.4]{Ben1} the proof is finished.
\end{proof}

\subsection{Patterson-Sullivan measures}\label{subsec PS}

Let $c$ be a H\"older cocycle over $\bg$ and $\delta$ be a positive number. A probability measure $\mu_c$ on $\bg$ is called a \textit{Patterson-Sullivan measure of dimension $\delta$} for $c$ if the equality\footnote{Recall that if $f:X\too Y$ is a map and $m$ is a measure on $X$ then $f_*m$ denotes the measure on $Y$ defined by $A\mapsto m(f^{-1}(A))$.}
\begin{equation}\label{eq PS}
\dfrac{d\gamma_*\mu_c}{d\mu_c}(x)=e^{-\delta c(\gamma^{-1},x)}
\end{equation}
\noindent is satisfied for every $\gamma\in\Gamma$.

The $\varphi$-\textit{critical exponent} of $\rho$ is defined by the equality

\bc
$\delta_\rho^\varphi:=\displaystyle\limsup_{t\too\infty} \dfrac{\log\#\lbrace \gamma\in\Gamma:\hspace{0,3cm} \varphi(a^\tau(\rho\gamma))\leq t \rbrace}{t}$.
\ec

\noindent It is positive and finite (because $\delta_\rho^1$ is).

Assume from now on that $\rho$ is Zariski dense. Quint \cite[Th\'eor\`eme 8.4]{Qui2} shows the existence of Patterson-Sullivan measures $\mu_\tau^\varphi$ and $\overline{\mu}_\tau^\varphi$ of dimension $\delta_\rho^\varphi$ for the cocycles $c_\tau^\varphi$ and $\overline{c}_\tau^\varphi$ respectively\footnote{Indeed, since $\varphi$ is positive in the limit cone $\cont$, the linear form $\delta_\rho^\varphi \varphi$ is tangent to the growth indicator of $\rho$ in a direction contained in $\cont\subset\tn{int}(\liea^+)$.}. Because of Remark \ref{rem covarphi is cohomologous to ctauvarphi}, we find Patterson-Sullivan measures $\mu_o^\varphi$ and $\overline{\mu}_o^\varphi$ of the same dimension for the cocycles $c_o^\varphi$ and $\overline{c}_o^\varphi$ respectively. We mention here that, in the case of fundamental groups of negatively curved closed manifolds, the existence (and uniqueness) of these probability measures is also shown by Ledrappier \cite{Led}.

The following lemma is well-known and will be used in Subsection \ref{subsec distrib pq cartan attractors}. We include a proof for completeness.

\begin{lema}\label{lema PS no atoms}
The probability measures $\mu_o^\varphi$ and $\overline{\mu}_o^\varphi$ have no atoms. 
\end{lema}

\begin{proof}

It suffices to show that $\mu_\tau^\varphi$ and $\overline{\mu}_\tau^\varphi$ have no atoms. We only do it for $\mu_\tau^\varphi$ (the other case being analogous). Let $\nu:={\xi_\rho}_*\mu_\tau^\varphi$ and $\varepsilon>0$. We will find a covering of $\xi_\rho(\bg)$ by open sets of $\nu$-measure less than or equal to $\varepsilon$.

By Pozzetti-Sambarino-Wienhard \cite[Lemma 5.15]{PozSamWieLipschtiz}, there exists a positive $\delta_0$ such that for every $0<\delta<\delta_0$ there exists $L=L(\delta)>0$ satisfying that for every $\gamma\in\Gamma$ with $\vert\gamma\vert_\Gamma>L$ one has

\bc
$\nu(\rho\gamma\cdot B_\delta(S^\tau(\rho\gamma)))\leq\varepsilon$.
\ec

\noindent Here we denote

\bc
$B_\delta(S^\tau(\rho\gamma)):=\lbrace\xi\in\mathsf{F}(V):\hspace{0,3cm} \Lambda^j\xi\in B_\delta(\Lambda^j_*S^\tau(\rho\gamma)) \tn{ for all } j=1,\dots,d-1\rbrace$.
\ec

\noindent We now follow the outline of \cite[Proposition 3.5]{PozSamWieLipschtiz} to show that $\delta$ can be chosen in such a way that

\bc
$\lbrace \rho\gamma\cdot B_\delta(S^\tau(\rho\gamma)): \hspace{0,3cm} \vert\gamma\vert_\Gamma>L \rbrace$
\ec

\noindent is a covering of $\xi_\rho(\bg)$.

By Bochi-Potrie-Sambarino \cite[Lemma 2.5]{BPS}, there exist positive constants $\eta_\rho$ and $L_\rho$ such that for every geodesic segment $(\gamma_i)_{i=0}^k$ in $\Gamma$ passing through the identity element of $\Gamma$ and such that $\vert \gamma_0\vert_\Gamma>L_\rho$ and $\vert \gamma_k\vert_\Gamma>L_\rho$ one has

\bc
$d\left(\Lambda^jU^\tau(\rho\gamma_k),\Lambda^j_*U^\tau(\rho\gamma_0)\right)\geq\eta_\rho$
\ec

\noindent for every $j=1,\dots,d-1$. 

Let $\delta<\eta_\rho$, we may take $L=L(\delta)$ to be larger than $L_\rho$. Fix $y\in\bg$ and a geodesic ray $(\gamma_i)_{i=0}^\infty$ in $\Gamma$ starting at the identity element of $\Gamma$ and converging to $y$. Let also $i_0\geq 0$ be chosen in such a way that if $\gamma:=\gamma_{i_0}$ then one has 

\bc
$\vert \gamma^{-1}\vert_\Gamma>L$.
\ec

\noindent Applying \cite[Lemma 2.5]{BPS} to the geodesic ray $\gamma^{-1}\cdot(\gamma_i)_{i=0}^\infty$ we find that for every $k$ large enough one has

\bc
$d\left(\Lambda^jU^\tau(\rho(\gamma^{-1}\gamma_k)),\Lambda^j_*U^\tau(\rho\gamma^{-1})\right)\geq\delta$
\ec

\noindent for every $j=1,\dots,d-1$. By \cite[Lemma 4.7]{BPS} we have

\bc
$U^\tau(\rho(\gamma^{-1}\gamma_k))\too\rho\gamma^{-1}\cdot\xi_\rho(y)$
\ec

\noindent as $k\too\infty$. Therefore up to changing $\delta$ by a smaller constant if necessary we may assume that

\bc
$d\left(\Lambda^j(\rho\gamma^{-1}\cdot\xi_\rho(y)),\Lambda^j_*S^\tau(\rho\gamma)\right)\geq\delta$
\ec

\noindent holds for every $j=1,\dots,d-1$. Hence $\rho\gamma^{-1}\cdot\xi_\rho(y)\in B_{\delta}(S^\tau(\rho\gamma))$.

\end{proof}

\subsection{Distribution of fixed points}\label{subsec dist fixed points}

For a metric space $X$, we denote by $C_c^*(X)$ the dual of the space of compactly supported continuous functions on $X$ equipped with the weak-star topology. For a point $x\in X$, we let $\delta_x\in C_c^*(X)$ be the Dirac mass at $x$.

Denote by $\bgc$ the space of ordered pairs of distinct points in $\bg$. When $\Gamma$ is the fundamental group of a closed negatively curved manifold, Sambarino \cite{Sam} shows the following (for a proof in our setting see Appendix \ref{app proof of ditribution}).

\begin{prop}[c.f. {\cite[Proposition 4.3 \& Theorem C]{Sam}}]\label{prop distribution of periodic orbits for UovarphiGamma}

The number $\delta_\rho^\varphi$ coincides with the $\varphi$-entropy $h_\rho^\varphi$ of $\rho$ and there exists a positive constant $\mathtt{m} =\mathtt{m}_{\rho,o,\varphi}$ such that

\bc
$\mathtt{m} e^{-h_\rho^\varphi t}\displaystyle\sum_{\gamma\in\gh, \varphi(\lambda(\rho\gamma))\leq t} \delta_{\gamma_-}\otimes\delta_{\gamma_+}\too e^{-h_\rho^\varphi[\cdot,\cdot]_o^\varphi}\overline{\mu}_o^\varphi\otimes\mu_o^\varphi$
\ec

\noindent as $t\too\infty$ on $C_c^*(\bgc)$.
\end{prop}

\subsection{Proof of Theorem \ref{TEO CONTEO DIRECCIONAL BO EN QPQ INTRODUCCION}}\label{subsec distrib pq cartan attractors}

As noticed by Roblin \cite{Rob}, equidistribution statements as that of Proposition \ref{prop distribution of periodic orbits for UovarphiGamma} can be used to study counting problems. In \cite{Sam} Sambarino applies Roblin's method to obtain a counting theorem for the operator norm of elements in the image of a (projective) Anosov representation and in C. \cite{CarHpq} we applied this method to study counting problems in pseudo-Riemannian hyperbolic spaces. 

The following proposition is an intermediate step towards the proof of Theorem \ref{TEO CONTEO DIRECCIONAL BO EN QPQ INTRODUCCION}. It implies 

\bc
$\mathtt{m}e^{-h_\rho^\varphi t}\#\lbrace \gamma\in\gh:\hspace{0,3cm} \rho\gamma\in\bog \tn{ and }\varphi( b^o(\rho\gamma))\leq t\rbrace\too 1$
\ec

\noindent as $t\too\infty$. In order to obtain Theorem \ref{TEO CONTEO DIRECCIONAL BO EN QPQ INTRODUCCION} we must also count the amount of torsion elements $\gamma$ for which $\varphi( b^o(\rho\gamma))\leq t$. This will be a consequence of Proposition \ref{prop linear counting for varphibo when no torsion}.

\begin{prop}\label{prop linear counting for varphibo when no torsion}
There exists a positive constant $\mathtt{m} =\mathtt{m}_{\rho,o,\varphi}$ such that

\bc
$\mathtt{m} e^{-h_\rho^\varphi t}\displaystyle\sum_{\gamma\in\gh,\varphi(b^o(\rho\gamma))\leq t}\delta_{\gamma_-}\otimes\delta_{\gamma_+}\too\overline{\mu}_o^\varphi\otimes\mu_o^\varphi$
\ec

\noindent as $t\too\infty$ on $C^*(\bg\times\bg)$.

\end{prop}

The sum in Proposition \ref{prop linear counting for varphibo when no torsion} is taken over all elements $\gamma\in\gh$ for which $\rho\gamma\in\bog$ and $\varphi(b^o(\rho\gamma))\leq t$. To make the formula more readable we do not emphasize the fact that $\rho\gamma$ must belong to $\bog$ (recall from Corollary \ref{cor sigmagammainvgamma is loxod and Uo close to Utau} that this holds with only finitely many exceptions $\gamma\in\Gamma$).

\begin{proof}[Proof of Proposition \ref{prop linear counting for varphibo when no torsion}]

Let

\bc
$\theta_t:=\mathtt{m} e^{-h_\rho^\varphi t}\displaystyle\sum_{\gamma\in\gh,\varphi(b^o(\rho\gamma))\leq t}\delta_{\gamma_-}\otimes\delta_{\gamma_+}$.
\ec

The proof follows line by line the proof of \cite[Theorem 6.5]{Sam} or \cite[Proposition 7.11]{CarHpq}. Namely, for open subsets $A$ and $B$ of $\bg$ with disjoint closure and negligible boundary, the convergence

\bc
$\theta_t(A\times B)\too\overline{\mu}_o^\varphi(A)\mu_o^\varphi(B)$
\ec

\noindent is implied by Proposition \ref{prop distribution of periodic orbits for UovarphiGamma} and Corollary \ref{cor gromov is cross ratio for rho and estimate with bogamma}. On the other hand, since $\overline{\mu}_o^\varphi$ and $\mu_o^\varphi$ have no atoms (Lemma \ref{lema PS no atoms}), one has

\bc
$\overline{\mu}_o^\varphi\otimes\mu_o^\varphi(\lbrace (x,x): \hspace{0,3cm} x\in\bg\rbrace)=0$.
\ec

\noindent In order to finish the proof it suffices to show the following: for every positive $\varepsilon_0$ there exists an open covering $\mathscr{U}$ of $\bg$ such that 

\bc
$\displaystyle\limsup_{t\too\infty}\theta_t\left(\displaystyle\bigcup_{U\in\mathscr{U}}U\times U\right)\leq\varepsilon_0$.
\ec

\noindent Provided Lemma \ref{lema triangle inequality} below, the proof of this fact follows exactly as the proof of the analogue fact in \cite[Theorem 6.5]{Sam} or in \cite[Proposition 7.11]{CarHpq}.

\end{proof}

\begin{lema}[c.f. {\cite[Proposition 6.10]{CarHpq}}]\label{lema triangle inequality}
Fix an element $\gamma_0\in\Gamma$. Then there exist positive constants $L$ and $D_{\gamma_0}$ such that for every $\gamma$ in $\Gamma$ satisfying $\vert\gamma\vert_\Gamma>L$ one has

\bc
$\Vert b^o(\rho(\gamma_0\gamma ))-  b^o(\rho\gamma )\Vert_\lieb \leq D_{\gamma_0}$.
\ec
\end{lema}

\begin{proof}
By Corollary \ref{cor if limit set in one orbit then HwexplieaplusH} there exist positive constants $L$ and $D$ such that for every $\gamma$ with $\vert\gamma\vert_\Gamma>L$ one has

\bc
$\Vert b^o(\rho(\gamma_0\gamma))-b^o(\rho\gamma)\Vert_\lieb\leq \Vert a^\tau(\rho(\gamma_0\gamma))-a^\tau(\rho\gamma)\Vert_\lieb +D$.
\ec

\noindent To finish observe that there exists a positive constant $d_{\gamma_0}$ for which the inequality

\bc
$\Vert a^\tau(\rho(\gamma_0\gamma))-a^\tau(\rho\gamma)\Vert_\lieb\leq d_{\gamma_0}$
\ec

\noindent is satisfied for every $\gamma\in\Gamma$.
\end{proof}

The following proposition finishes the proof of Theorem \ref{TEO CONTEO DIRECCIONAL BO EN QPQ INTRODUCCION}.

\begin{prop}\label{prop linear counting for varphibo when torsion}
There exists a positive constant $\mathtt{m} =\mathtt{m}_{\rho,o,\varphi}$ such that

\bc
$\mathtt{m} e^{-h_\rho^\varphi t}\displaystyle\sum_{\gamma\in\Gamma,\varphi(b^o(\rho\gamma))\leq t}\delta_{S^o(\rho\gamma)}\otimes\delta_{U^o(\rho\gamma)}\too{\xi_\rho}_*\overline{\mu}_o^\varphi\otimes{\xi_\rho}_*\mu_o^\varphi$
\ec

\noindent as $t\too\infty$ on $C^*(\mathsf{F}(V)\times\mathsf{F}(V))$.

\end{prop}

\begin{proof}
The proof is analogous to the proof of \cite[Proposition 7.13]{CarHpq}, the main steps being:

\begin{itemize}
\item Let 

\bc
$\nu_t^\tn{H}:=\mathtt{m} e^{-h_\rho^\varphi t}\displaystyle\sum_{\gamma\in\gh,\varphi(b^o(\rho\gamma))\leq t}\delta_{S^o(\rho\gamma)}\otimes\delta_{U^o(\rho\gamma)}$.
\ec

\noindent Provided Corollary \ref{cor sigmagammainvgamma is loxod and Uo close to Utau} and Proposition \ref{prop limit with S and U and Uuno cerca gammamas}, the convergence

\bc
$\nu_t^\tn{H}\too {\xi_\rho}_*\overline{\mu}_o^\varphi\otimes{\xi_\rho}_*\mu_o^\varphi$
\ec

\noindent follows from Proposition \ref{prop linear counting for varphibo when no torsion}.

\item Let 

\bc
$\nu_t:=\mathtt{m} e^{-h_\rho^\varphi t}\displaystyle\sum_{\gamma\in\Gamma,\varphi(b^o(\rho\gamma))\leq t}\delta_{S^o(\rho\gamma)}\otimes\delta_{U^o(\rho\gamma)}$,
\ec

\noindent which differs from $\nu_t^\h$ only from the fact that we allow torsion elements in the defining sum. Fix a continuous function $f$ on $\mathsf{F}(V)\times\mathsf{F}(V)$ whose support $\tn{supp}(f)$ is contained in $\fvc$. An application of Corollary \ref{cor sigmagammainvgamma is loxod and Uo close to Utau} and Benoist \cite[Lemme 6.2]{Ben1} yields

\bc
$\#\lbrace \gamma\in\Gamma:\hspace{0,3cm} (S^o(\rho\gamma),U^o(\rho\gamma))\in\tn{supp}(f) \tn{ and } \gamma\notin\gh\rbrace<\infty$.
\ec

\noindent This implies the convergence $\nu_t^\tn{H}(f)-\nu_t(f)\too 0$.

\item To finish it remains to analyze the asymptotic behaviour of the measure $\nu_t$ over the set $\mathscr{D}:=\mathsf{F}(V)^2\setminus\fvc$.  As in \cite[Proposition 7.13]{CarHpq}, for every positive $\varepsilon_0$ one can find an open covering $\mathscr{U}$ of $\mathscr{D}$ such that

\bc
$\displaystyle\limsup_{t\too\infty}\nu_t\left(\bigcup_{U\in\mathscr{U}} U\right)\leq\varepsilon_0$.
\ec

\noindent Since the number ${\xi_\rho}_*\overline{\mu}_o^\varphi\otimes{\xi_\rho}_*\mu_o^\varphi(\mathscr{D})$ equals zero (Lemma \ref{lema PS no atoms}), the proof is complete.

\end{itemize}
\end{proof}

\appendix

\section{Distribution of fixed points, mixing and counting}\label{app proof of ditribution}

We now explain why the results of Sambarino \cite{SamEXPGROWTH,Sam,Sam2} hold when we replace the fundamental group of a closed negatively curved manifold by a general word hyperbolic group $\Gamma$ admitting an Anosov representation. The central dynamical tool used by the author along the above works is the thermodynamical formalism for the geodesic flow of the manifold. Provided the work of Bridgeman-Canary-Labourie-Sambarino \cite{BCLS}, the thermodynamical formalism also applies to the geodesic flow of $\Gamma$.

\subsection{The geodesic flow}

Let $\Gamma$ be a word hyperbolic group. Gromov \cite{Gro} introduced a compact space $\tn{U}\Gamma$ endowed with a transitive H\"older flow, called the \textit{geodesic flow} of $\Gamma$, which is well defined up to H\"older reparametrization (see Mineyev \cite{Min} for details). If $\Gamma$ admits an Anosov representation, the geodesic flow of $\Gamma$ is \textit{metric Anosov} \cite[Section 5]{BCLS}. By work of Pollicott \cite{Pol}, the geodesic flow of $\Gamma$ admits then a \textit{Markov coding} and therefore the techniques coming from the thermodynamical formalism of subshifts of finite type apply (see \cite{BCLS,CLT}).

The following way of constructing parametrizations of the geodesic flow is useful. Let $\tn{L}$ be a one dimensional H\"older real vector bundle over $\bgc$. Let $\widehat{\tn{L}}$ be the $\rr$-principal bundle over $\bgc$ whose fibers are

\bc
$\widehat{\tn{L}}_{(x,y)}:=(\tn{L}_{(x,y)}\setminus\lbrace 0\rbrace)/\sim$,
\ec

\noindent where $v \sim -v$. Here the action of $t\in\rr$ on $\widehat{\tn{L}}$ is given by $t:[v]\mapsto[e^tv]$. Suppose furthermore that $\tn{L}$ is endowed with an action of $\Gamma$ by bundle automorphisms. For every $\gamma\in\Gamma_\h$, let $p_{\tn{L}}(\gamma)$ be the real number such that for every $v\in\tn{L}_{(\gamma_-,\gamma_+)}$ one has

\bc
$\gamma\cdot v=\pm e^{p_{\tn{L}}(\gamma)}v$.
\ec

\noindent The following proposition is essentially proven in \cite[Proposition 4.2]{BCLS} (see also \cite[Proposition 2.4]{BCLSSIMPLEROOTS}).

\begin{prop}\label{prop line bundles and reparam}
Let $\rho:\Gamma\too\g$ be a $\Delta$-Anosov representation. Assume that there exists a positive constant $\alpha$ such that the inequality

\bc
$p_{\tn{L}}(\gamma)\geq\alpha \lambda_1(\rho\gamma)$
\ec

\noindent holds for every $\gamma\in\Gamma_\h$. Then the action of $\Gamma$ on $\widehat{\tn{L}}$ is properly discontinuous and the quotient space $\tn{U}_{\widehat{\tn{L}}}$ is H\"older homeomorphic to $\tn{U}\Gamma$. Furthermore, the flow on $\tn{U}_{\widehat{\tn{L}}}$ induced by the action of $\rr$ on $\widehat{\tn{L}}$ is H\"older conjugate to a H\"older reparametrization of the geodesic flow of $\Gamma$.
\end{prop}

\begin{proof}

Let

\bc
$\phi_t:\tn{U}_\rho\Gamma\too\tn{U}_\rho\Gamma$
\ec

\noindent be the \textit{geodesic flow of} $\rho$, which is H\"older conjugate to a H\"older reparametrization of the geodesic flow of $\Gamma$ (see \cite[Section 4]{BCLS}). Let $\xi_\rho=(\xi^1_\rho,\dots,\xi^{d}_\rho)$ be the limit map of $\rho$. We recall that $\tn{U}_\rho\Gamma$ is the quotient, under the natural action of $\Gamma$, of the fiber bundle over $\bgc$ whose fiber over $(x,y)$ is

\bc
$\lbrace (\vartheta,v): \vartheta\in\xi_\rho^{d-1}(x), \hspace{0,1cm} v\in\xi_\rho^1(y), \hspace{0,1cm} \vartheta(v)=1\rbrace/\sim$,
\ec

\noindent where $(\vartheta,v)\sim(-\vartheta,-v)$. In particular, for every periodic orbit $a$ of $\phi$ we may find an element $\gamma_a\in\Gamma_\h$ such that the period of $a$ coincides with $\lambda_1(\rho\gamma_a)$ when $\kk=\rr$, or twice this number when $\kk=\cc$ (see  \cite[Section 4]{BCLS} and \cite[Corollary 2.10]{BPSW}).

Let $\tn{E}$ be the vector bundle over $\tn{U}_\rho\Gamma$ whose fiber over $[x,y,\vartheta,v]$ is $\tn{L}_{(x,y)}$. The geodesic flow on $\tn{U}_\rho\Gamma$ naturally lifts to a flow $\tilde{\phi}_t$ on $\tn{E}$.

\begin{cla}\label{claim the lifted flow is contracting}
Let $\Vert\cdot\Vert$ be a H\"older norm on $\tn{E}$. There exists $t_0>0$ such that for all $t>t_0$, all $z\in\tn{U}_\rho\Gamma$ and all $u\in\tn{E}_{z}\setminus\lbrace 0\rbrace$ one has

\bc
$\dfrac{\Vert\tilde{\phi}_t(u)\Vert}{\Vert u\Vert}<\dfrac{1}{4}$.
\ec
\end{cla}

Provided Claim \ref{claim the lifted flow is contracting}, the same proof of \cite[Proposition 4.2]{BCLS} applies to conclude the proof of Proposition \ref{prop line bundles and reparam}.

\begin{proof}[Proof of Claim \ref{claim the lifted flow is contracting}]

Let $\kappa:\tn{U}_\rho\Gamma\times\rr\to\rr$ be given by

\bc
$\kappa(z,t):=\log\dfrac{\Vert u\Vert}{\Vert\tilde{\phi}_t(u)\Vert}$,
\ec

\noindent where $u\in\tn{E}_z$ is any non zero vector. Then $\kappa$ is a H\"older cocycle, i.e. for all $z\in\tn{U}_\rho\Gamma$ and $t,s\in\rr$ one has

\bc
$\kappa(z,t+s)=\kappa(\phi_t(z),s)+\kappa(z,t)$.
\ec

Fix a positive $c>0$ and define $j_c:\tn{U}_\rho\Gamma\times\rr\to\rr$ by

\bc
$j_c(z,t):=\dfrac{1}{2c}\displaystyle\int_{-c}^c\kappa(z,t+s)\operatorname{d}s$.
\ec

\noindent Let also $f:\tn{U}_\rho\Gamma\to\rr$ be given by

\bc
$f(z):=\left.\dfrac{\operatorname{d}}{\operatorname{d}t}\right\vert_{t=0}j_c(z,t)$.
\ec

\noindent Finally, let $\kappa_c:\tn{U}\Gamma\times\rr\to\rr$ be defined by

\bc
$\kappa_c(z,t):=\displaystyle\int_{0}^tf(\phi_s(z))\operatorname{d}s$.
\ec

Then $\kappa_c$ is a H\"older cocycle and its periods coincide with those of $\kappa$ (recall that the \textit{period} according to $\kappa$ of a periodic orbit $a$ of $\phi$ of period $p_\phi(a)$ is defined by $\kappa(z,p_\phi(a))$, for $z\in a$). Livsic's Theorem \cite{Liv} guarantees that $\kappa$ and $\kappa_c$ are then \textit{cohomologous}. This means that we may find a H\"older continuous function $U:\tn{U}_\rho\Gamma\to\rr$ such that for all $z\in\tn{U}_\rho\Gamma$ and $t\in\rr$ one has
\begin{equation}\label{eq cohomology}
\kappa(z,t)-\kappa_c(z,t)=U(\phi_t(z))-U(z).
\end{equation}
Furthermore, let $a$ be a periodic orbit of $\tn{U}_\rho\Gamma$ and $\gamma_a\in\Gamma_\h$ be such that $p_\phi(a)=\lambda_1(\rho\gamma_a)$. Then

\bc
$\displaystyle\int_{a}f=\kappa(z,p_\phi(a))=p_{\tn{L}}(\gamma_a)\geq \alpha \lambda_1(\rho\gamma_a)>0$,
\ec

\noindent for $z\in a$. Since $\phi:\tn{U}_\rho\Gamma\to\tn{U}_\rho\Gamma$ is a transitive metric Anosov flow we find a positive H\"older continuous function $h:\tn{U}_\rho\Gamma\to\rr$ cohomologous to $f$ (see e.g. \cite[Lemma 2.5]{PS}). Combining this fact with (\ref{eq cohomology}) we find a H\"older continuous function $V:\tn{U}_\rho\Gamma\to\rr$ such that for all $z\in\tn{U}_\rho\Gamma$, all $t\in\rr$ and all $u\in\tn{E}_z\setminus\lbrace 0 \rbrace$ one has

\bc
$\log\dfrac{\Vert u\Vert}{\Vert\tilde{\phi}_t(u)\Vert}-\displaystyle\int_{0}^th(\phi_s(z))\operatorname{d}s=V(\phi_t(z))-V(z)$.
\ec

\noindent Since $h>0$ the proof is complete.

\end{proof}

\end{proof}

\subsection{The flow associated to the $(\varphi,\tau)$-Busemann cocycle}

Fix a $\Delta$-Anosov representation $\rho:\Gamma\to\g$. Let $\tau\in\xsyg$ be a basepoint and $\liea^+$ be a closed Weyl chamber of the system $\Sigma(\lieg,\liea)$, for some Cartan subspace $\liea\subset\liep^\tau$. Let $\varphi$ be a functional in the interior $\tn{int}(\cont^*)$ of the dual limit cone $\cont^*$. Let $c_\tau^\varphi$ and $\overline{c}_\tau^\varphi$ be the $(\varphi,\tau)$-\textit{Busemann cocycles} of $\rho$ (c.f. Remark \ref{rem covarphi is cohomologous to ctauvarphi}) and consider the action of $\Gamma$ on $\bgc\times\rr$ given by

\bc
$\gamma\cdot(x,y,s):=(\gamma\cdot x,\gamma\cdot y,s-c_\tau^\varphi(\gamma,y))$.
\ec

\noindent The \textit{translation flow} is the flow $\psi_t=\psi_t^{\rho,\tau,\varphi}$ on the quotient space of this action induced by the action of $\rr$ on $\bgc\times\rr$ given by

\bc
$t:(x,y,s)\mapsto(x,y,s-t)$.
\ec

We assume from now on that $\rho$ is Zariski dense. Let 

\bc
$[\cdot,\cdot]_\tau^\varphi:\bgc\too\rr$
\ec

\noindent be the Gromov product associated to the pair $(\overline{c}_\tau^\varphi,c_\tau^\varphi)$ (see \cite[Lemma 7.10]{Sam}) and recall that $\overline{\mu}_\tau^\varphi$ and $\mu_\tau^\varphi$ are Patterson-Sullivan probability measures for the cocycles $\overline{c}_\tau^\varphi$ and $c_\tau^\varphi$ respectively, of dimension $\delta_\rho^\varphi$.

Applying Ledrappier's framework \cite{Led}, the following theorem is proved in \cite[Theorem 3.2 \& Theorem C]{Sam} for fundamental groups of closed negatively curved manifolds. We briefly explain the main lines of the proof in our setting.

\begin{teo}\label{teo reparametrizing}
Let $\rho:\Gamma\too\g$ be a Zariski dense $\Delta$-Anosov representation and $\varphi\in\tn{int}(\cont^*)$. Then the following holds:
\begin{enumerate}
\item The action of $\Gamma$ on $\bgc\times\rr$ induced by $c_\tau^\varphi$ is properly discontinuous and the translation flow $\psi_t$ is conjugate, by a H\"older homeomorphism, to a H\"older reparametrization of the geodesic flow of $\Gamma$. 

\item For every periodic orbit of $\psi$ one may find an element $\gamma\in\Gamma_\h$ such that the period of this periodic orbit coincides with $\varphi(\lambda(\rho\gamma))$.
\item Let $h_{\tn{top}}(\psi)$ be the topological entropy of $\psi$. One has the equalities $h_{\tn{top}}(\psi)=\delta_\rho^\varphi=h_\rho^\varphi$ and the measure

\bc
$e^{-h_\rho^\varphi[\cdot,\cdot]_\tau^\varphi}\overline{\mu}_\tau^\varphi\otimes\mu_\tau^\varphi\otimes dt$
\ec

\noindent on $\bgc\times\rr$ descends to a measure on the quotient space that maximizes entropy for $\psi_t$.
\end{enumerate}
\end{teo}

\begin{proof}
\begin{enumerate}
\item Endow $\tn{L}:=\bgc\times\rr$ with the action of $\Gamma$ by bundle automorphisms given by

\bc
$\gamma:(x,y,s)\mapsto(\gamma\cdot x,\gamma\cdot y, e^{c_\tau^\varphi(\gamma,y)}s)$.
\ec

\noindent For every $\gamma\in\Gamma_\h$ one has

\bc
$p_{\tn{L}}(\gamma)=\varphi(\lambda(\rho\gamma))$
\ec

\noindent and since $\varphi$ is positive in the asymptotic cone $\cont$ we find a positive constant $\alpha$ for which

\bc
$p_{\tn{L}}(\gamma)\geq\alpha\lambda_1(\rho\gamma)$
\ec

\noindent holds. To finish apply Proposition \ref{prop line bundles and reparam}.

\item Direct computation.

\item By \cite[Subsection 3.5]{Pol} the flow $\psi_t$ admits a unique probability measure maximizing entropy (called the \textit{Bowen-Margulis measure}). This measure is ergodic (see Bowen-Ruelle \cite{BowenRuelle}).

It can be seen that the quotient measure of

\bc
$e^{-\delta_\rho^\varphi[\cdot,\cdot]_\tau^\varphi}\overline{\mu}_\tau^\varphi\otimes\mu_\tau^\varphi\otimes dt$
\ec

\noindent is absolutely continuous with respect to a measure $\nu$ that can locally be written as 

\bc
$\nu=\nu_{\tn{loc}}^{ss}\otimes\nu_{\tn{loc}}^{cu}$,
\ec

\noindent where the families $\lbrace\nu_{\tn{loc}}^{cu}\rbrace$ and $\lbrace\nu_{\tn{loc}}^{ss}\rbrace$ satisfy the following: each measure $\nu^{cu}_{\tn{loc}}$ (resp. $\nu^{ss}_{\tn{loc}}$) is a finite Borel measure on local leaves $\mathscr{W}^{cu}_{\tn{loc}}$ (resp. $\mathscr{W}^{ss}_{\tn{loc}}$) of the central unstable (resp. strong stable) lamination of $\psi_t$, and for every $t\in\rr$ the following equalities are satisfied

\bc
${\phi_t}_*\nu^{cu}_{\tn{loc}}=e^{-\delta_\rho^\varphi t}\nu^{cu}_{\tn{loc}}$ and ${\phi_t}_*\nu^{ss}_{\tn{loc}}=e^{\delta_\rho^\varphi t}\nu^{ss}_{\tn{loc}}$.
\ec

\noindent Standard techniques of the theory of (metric) Anosov flows (see e.g. Katok-Hasselblatt \cite[Section 5 of Chapter 20]{KH}) imply that $\delta_\rho^\varphi$ must coincide with the topological entropy $h_{\tn{top}}(\psi)$ of $\psi_t$ and that $\nu$ is proportional to the Bowen-Margulis measure of this flow. By ergodicity, the measure

\bc
$e^{-\delta_\rho^\varphi[\cdot,\cdot]_\tau^\varphi}\overline{\mu}_\tau^\varphi\otimes\mu_\tau^\varphi\otimes dt$
\ec

\noindent descends to a measure proportional to the Bowen-Margulis measure of $\psi$.

Finally, the inequalities 

\bc
$h_{\tn{top}}(\psi)\leq h_\rho^\varphi\leq \delta_\rho^\varphi$
\ec

\noindent are easy to check, and this completes the proof.

\end{enumerate}
\end{proof}

\subsection{Distribution of fixed points}\label{subsec distribution appendix}
We now finish the proof of Proposition \ref{prop distribution of periodic orbits for UovarphiGamma}.

\begin{proof}[Proof of Proposition \ref{prop distribution of periodic orbits for UovarphiGamma}]

The Zariski density assumption over $\rho$ implies that the set of periods of periodic orbits of $\psi_t$ span a dense subgroup of $\rr$ (see Benoist \cite[Main Proposition]{Ben1}). On the other hand, by Remark \ref{rem covarphi is cohomologous to ctauvarphi} one has that the equality

\bc
$e^{-h_\rho^\varphi[\cdot,\cdot]_o^\varphi}\overline{\mu}_o^\varphi\otimes\mu_o^\varphi\otimes dt=e^{-h_\rho^\varphi[\cdot,\cdot]_\tau^\varphi}\overline{\mu}_\tau^\varphi\otimes\mu_\tau^\varphi\otimes dt$
\ec

\noindent holds up to scaling. Using the thermodynamical formalism the proof of \cite[Proposition 4.3]{Sam} can now be adapted to our setting. Indeed, the only difference is that in our general setting the group $\Gamma$ may have torsion elements. However, the proof adapts as follows (c.f. Blayac \cite[Section 8.3.3]{BlaThesis}).

Let $\Gamma_{\tn{SP}}\subset\Gamma_\h$ be the set of \textit{strongly primitive} elements, i.e. elements $\gamma\in\Gamma_\h$ for which the period of the corresponding periodic orbit of $\psi$ is precisely $\varphi(\lambda(\rho\gamma))$. By  Pollicott \cite[Subsection 3.5]{Pol} there exists a positive $\mathtt{m}$ such that

\bc
$\sigma_t:=\mathtt{m} t e^{-h_\rho^\varphi t}\displaystyle\sum_{\gamma\in\Gamma_{\tn{SP}}, \varphi(\lambda(\rho\gamma))\leq t} \dfrac{1}{\varphi(\lambda(\rho\gamma))}\delta_{\gamma_-}\otimes\delta_{\gamma_+}\too e^{-h_\rho^\varphi[\cdot,\cdot]_o^\varphi}\overline{\mu}_o^\varphi\otimes\mu_o^\varphi$.
\ec

Pick a subset $\Gamma^0\subset\Gamma_{\tn{SP}}$ in one to one correspondence with the set of \textit{axes} 

\bc

$\lbrace(\gamma_-,\gamma_+)\times\rr\rbrace_{\gamma\in\gh}$. 

\ec

\noindent For $\gamma\in\Gamma^0$ let $S_\gamma\subset\Gamma$ be the set of elements that act trivially on the corresponding axis $(\gamma_-,\gamma_+)\times\rr$. We have

\bc
$\sigma_t= \mathtt{m} e^{-h_\rho^\varphi t}\displaystyle\sum_{\gamma\in\Gamma^0, \varphi(\lambda(\rho\gamma))\leq t} \dfrac{(\# S_\gamma)t}{\varphi(\lambda(\rho\gamma))}\delta_{\gamma_-}\otimes\delta_{\gamma_+}$.
\ec

On the other hand, let

\bc
$\nu_t:=\mathtt{m} e^{-h_\rho^\varphi t}\displaystyle\sum_{\gamma\in\Gamma_\h, \varphi(\lambda(\rho\gamma))\leq t} \delta_{\gamma_-}\otimes\delta_{\gamma_+}$.
\ec

\noindent It follows easily that

\bc
$\nu_t=\mathtt{m} e^{-h_\rho^\varphi t}\displaystyle\sum_{\gamma\in\Gamma^0, \varphi(\lambda(\rho\gamma))\leq t}  \#S_\gamma \left\lfloor\dfrac{t}{\varphi(\lambda(\rho\gamma))}\right\rfloor  \delta_{\gamma_-}\otimes\delta_{\gamma_+}$.
\ec

\noindent The convergence $\sigma_t-\nu_t\too 0$ can now be proven exactly as in \cite[Proposition 4.3]{Sam}.
\end{proof}

\subsection{Mixing and counting}
Let $\rr_+ X_\varphi\subset\cont$ be the direction \textit{dual} to $\rr_+\varphi\subset\tn{int}(\cont^*)$. By definition, it is the unique direction in $\cont$ in which the form $h_\rho^\varphi\varphi$ is tangent to \textit{Quint's growth indicator} $\pst$ (c.f. Quint \cite{Qui3} and \cite[Theorem 4.20]{Sam2}). Here we pick the vector $X_\varphi$ in such a way that $\varphi(X_\varphi)=1$. Given a Euclidean norm $\vert\cdot\vert$ on $\liea$ let

\bc
$I(X):=\dfrac{\vert X\vert^2\vert X_\varphi\vert^2-\langle X,X_\varphi\rangle^2}{\vert X_\varphi\vert^2}$
\ec

\noindent for $X\in\ker\varphi$.

On the other hand, the \textit{Weyl chamber flow} is the action of $\liea$ on the space $\rho(\Gamma)\backslash\g/\mb$ given by

\bc
$X:\rho(\Gamma) g\mb\mapsto\rho(\Gamma) g\exp(X)\mb$
\ec

\noindent for every $X\in\liea$ and $g\in\g$. The $\tau$-Busemann cocycle of $\g$ induces a $\g$-equivariant identification $\g/\mb\cong\fvc\times\liea$ and, for a given a Lebesgue measure $\tn{Leb}_\liea$ on $\liea$, the pushforward of

\bc
$e^{-h_\rho^\varphi[\cdot,\cdot]_\tau^\varphi}\overline{\mu}_\tau^\varphi\otimes\mu_\tau^\varphi\otimes \tn{Leb}_\liea$
\ec

\noindent under the map $\xi_\rho\times\xi_\rho\times\tn{id}_{\liea}$ descends to a measure $\chi_\varphi$ on $\rho(\Gamma)\backslash\g/\mb$ invariant under the Weyl chamber flow. This measure is called the $\varphi$-\textit{Bowen-Margulis measure} of $\rho(\Gamma)$.

Given two functions $f_0,f_1:\rho(\Gamma)\backslash\g/\mb\too\rr$ and an element $X\in\liea$ denote by $(X\cdot f_0) f_1$ the function $\rho(\Gamma)\backslash\g/\mb\too\rr$ given by

\bc
$z\mapsto f_0(X\cdot z)f_1(z)$.
\ec

\noindent The proof of the following theorem follows line by line \cite[Theorem 4.23]{Sam2}. Indeed, the central tools involved in the proof of that result are the Reparametrizing Theorem \cite[Theorem 3.2]{Sam} (here replaced by Theorem \ref{teo reparametrizing}) and the thermodynamical formalism for the flow $\psi_t$.

\begin{teo}\label{teo mixing}
There exists a positive constant $\mathtt{c}$ and a Euclidean norm $\vert\cdot\vert$ on $\liea$ such that for every pair of compactly supported continuous functions $f_0,f_1:\rho(\Gamma)\backslash\g/\mb\too\rr$ and every $X_0\in\ker\varphi$ one has

\bc
$(2\pi t)^{(\dim\liea -1)/2}\chi_\varphi(((tX_\varphi+\sqrt{t}X_0)\cdot f_0)f_1)\too\mathtt{c}e^{-I(X_0)/2}\chi_\varphi(f_0)\chi_\varphi(f_1)$
\ec

\noindent as $t\too\infty$.
\end{teo}

Equipped with Theorem \ref{teo mixing}, the contents of \cite[Section 5]{Sam2} remain valid in our setting and give the following counting theorem (recall that $\delta_\rho$ is the critical exponent of $\rho$).

\begin{teo}\label{teo counting sambarino}
There exists a positive constant $\mathtt{C}$ such that

\bc
$\mathtt{C}e^{-\delta_\rho t}\# \lbrace\gamma\in\Gamma: \hspace{0,3cm} \Vert a^\tau(\rho\gamma)\Vert_\liea\leq t\rbrace\too 1$
\ec

\noindent as $t\too\infty$, where $\Vert\cdot\Vert_\liea$ is the Euclidean norm on $\liea$ induced by the Riemannian structure on $\xsyg$.
\end{teo}

\bibliographystyle{plain}
\bibliography{countingquadraticforms(corrected)}
\end{document}